\documentclass[11pt]{amsart} 
\usepackage[bottom=1in]{geometry}
\usepackage{esint} 
\usepackage{graphicx, amssymb, hyperref }
\usepackage[mathscr]{euscript}
\usepackage{todonotes}
\usepackage[english]{babel}

\usepackage{enumitem}  

\newtheorem{theorem}{Theorem}

\newtheorem{definition}[theorem]{Definition}
\newcommand{\ssize}{\text{size}\,}
\newcommand{\eenergy}{\text{energy}\,}
\newtheorem{lemma}[theorem]{Lemma}
\newtheorem{corollary}[theorem]{Corollary}
\newtheorem{proposition}[theorem]{Proposition}
\newtheorem*{remark}{Remark:}
\newtheorem*{notation}{Notation:}
\newcommand{\sssize}{\widetilde{\text{size}\,}}

\newcommand{\one}{\mathbf{1}}
\newcommand{\dist}{\text{\,dist\,}}
\newcommand{\rr}{\mathbb}
\newcommand{\ii}{\mathscr}
\newcommand{\ic}{\mathcal}
\newcommand{\ci}{\tilde{\chi}}

\newcommand{\ds}{\displaystyle}

\newtheorem{question*}{Question}
\newtheorem*{main*}{\underline{Induction statement}}

\newcommand{\nn}{\langle n \rangle}

\newcommand{\lft}{\big|}
\newcommand{\rg}{\big|}
\newcommand{\supp}{\text{supp\,}}

\author{Cristina Benea}
\address{Cristina Benea, Universit\'{e} de Nantes, Laboratoire Jean Leray, Nantes 44322, France}
\email{cristina.benea@univ-nantes.fr}

\author[Camil Muscalu]{Camil Muscalu*}
\thanks{$^*$The author is also a Member of the ``Simion Stoilow" Institute of Mathematics of the Romanian Academy}
\address{Camil Muscalu, Department of Mathematics, Cornell University, Ithaca, NY 14853, USA}
\email{camil@math.cornell.edu}

 \title[Quasi-Banach Valued Inequalities via the Helicoidal method]{Quasi-Banach Valued Inequalities via the Helicoidal method
} 
\begin{document}

\begin{abstract}
We extend the helicoidal method from \cite{vv_BHT} to the quasi-Banach context, proving in this way multiple Banach and quasi-Banach vector-valued inequalities for paraproducts $\Pi$ and for the bilinear Hilbert transform $BHT$. As an immediate application, we obtain mixed norm estimates for $\Pi \otimes \Pi$ in the whole range of Lebesgue exponents.

One of the novelties in the quasi-Banach framework (that is, when $0<r<1$), which we expect to be useful in other contexts as well, is the ``linearization" of the operator $ \left( \sum_k \vert T(f_k, g_k) \vert^r  \right)^{1/r}$, achieved by dualizing its \emph{weak-}$L^p$ quasinorms through $L^r$ (see Proposition \ref{prop:Lr-dualization}). Another important role is played by the sharp evaluation of the operatorial norm $\|  T_{I_0}(f \cdot \one_F, g \cdot \one_G) \cdot \one_{H'}\|_r$, which is obtained by dualizing the \emph{weak-}$L^p$ quasinorms through $L^\tau$, with $\tau \leq r$. In the Banach case, the linearization of the operator and the sharp estimates for the localized operatorial norm can be both achieved through the classical (generalized restricted type) $L^1$ dualization.

\end{abstract}

\maketitle

\section{Introduction}

The present work is a natural continuation of our prior article \cite{vv_BHT}, where we introduced a new method (termed the helicoidal method) for proving various multiple vector-valued inequalities in harmonic analysis. This technique, initially developed for the bilinear Hilbert transform $BHT$, reduces to H\"older's inequality and some very precise local estimates for the operator in question.

More precisely, let $T$ be an $m$-linear operator so that $T: L^{p_1} \times \ldots \times L^{p_m} \to L^p$ with $\frac{1}{p_1}+\ldots +\frac{1}{p_m}=\frac{1}{p}$, $1< p_j \leq \infty$ and $\frac{1}{m}<p<\infty$ (whenever $T$ satisfies such estimates, we say $(p_1, \ldots, p_m, p  ) \in Range(T)$). We want to prove the vector-valued inequality of ``depth $n$":
\begin{equation}
\label{eq:vv-depth-n}
\vec{T}_n : L^{p_1}\big( \rr{R}; L^{R_1}(\ii W, \mu) \big) \times \ldots \times L^{p_m}\big( \rr{R}; L^{R_m}(\ii W, \mu) \big) \to L^{p}\big( \rr{R}; L^{R}(\ii W, \mu) \big),
\end{equation}
where $\big( \ii W, \Sigma, \mu \big)$ is a totally $\sigma$-finite measure space and the $n$-tuples $R_k=\big( r^1_k, \ldots, r^n_k \big), R=\big( r^1, \ldots, r^n \big)$ satisfy for every $1 \leq k \leq m$,  $1 \leq j \leq n$
\begin{equation}
\label{eq:cond-r-tuples}
1< r_k^j \leq \infty, \quad \frac{1}{m}<r^j<\infty \quad \text{and}\quad \frac{1}{r_1^j}+\ldots +\frac{1}{r_m^j}=\frac{1}{r^j}.
\end{equation}
The $L^{R}$ norm on $(\ii W, \Sigma, \mu)=\big( \prod_{j=1}^n \ii W_j, \prod_{j=1}^n \Sigma_j, \prod_{j=1}^n \mu_j   \big)$ is defined as
\[
\| \vec f  \|_{L^R(\ii W, \mu)}:= \big(  \int_{\ii W_1} \ldots \big( \int_{\ii W_n} \vert \vec f(w_1, \ldots , w_n)   \vert^{r^n} d \mu_n(w_n)     \big)^{r^{n-1}/r^n} \ldots d \mu_1(w_1) \big)^{1/{r^1}}
\]
If all $ 1 \leq r^j <\infty$ and $1\leq p <\infty$, $\vec T_n$ can be understood through the multilinear form $\Lambda_{\vec T_n}$ associated to it and \eqref{eq:vv-depth-n} becomes equivalent to proving 
\[
\Lambda_{\vec T_n} : L^{p_1}\big( \rr{R}; L^{R_1}(\ii W, \mu) \big) \times \ldots \times L^{p_m}\big( \rr{R}; L^{R_m}(\ii W, \mu) \big) \times L^{p'}\big( \rr{R}; L^{R'}(\ii W, \mu) \big) \to \rr C
\]
where $R'=\big( \big( r^1 \big)', \ldots , \big( r^n \big)'  \big)$. That is, for every $1 \leq k \leq m$ and every $\vec f_k \in L^{p_k}\big( \rr{R}; L^{R_k}(\ii W, \mu) \big)$ and $\vec h \in  L^{p'}\big( \rr{R}; L^{R'}(\ii W, \mu) \big)$, we want to prove
\begin{equation}
\label{eq:intr-multilin-form}
\vert \Lambda_{\vec T_n}(\vec f_1, \ldots, \vec f_m, \vec h)  \vert \lesssim  \|\vec f_1\|_{ L^{p_1}\big( \rr{R}; L^{R_1}(\ii W, \mu) \big)} \cdot \ldots \cdot\|\vec f_m\|_{L^{p_m}\big( \rr{R}; L^{R_m}(\ii W, \mu) \big)} \cdot\| \vec h\|_{ L^{p'}\big( \rr{R}; L^{R'}(\ii W, \mu) \big)}.
\end{equation}

Using restricted weak-type interpolation, we were able in \cite{vv_BHT} to treat also the case $\frac{1}{m}<p<1$, when $1 \leq \vec R <\infty$ (the last inequality is to be read component-wise). The advantage to this approach is that we can dualize $L^{p, \infty}$ norms even for $p<1$. The inequality equivalent to \eqref{eq:intr-multilin-form} in this case is, morally speaking, 
\begin{equation}
\label{eq:intr-multilin-form-restr}
\vert \Lambda_{\vec T_n}(\vec f_1, \ldots, \vec f_m, \vec h)  \vert \lesssim |F_1|^{\frac{1}{p_1}} \cdot \ldots \cdot |F_{m}|^{\frac{1}{p_m}} |H|^{\frac{1}{p'}},
\end{equation}
for all $\vec{f_k}, \vec h$ so that $\|\vec f_k\|_{  L^{R_k}(\ii W, \mu)} \leq\one_{F_k}, \|\vec h\|_{ L^{R'}(\ii W, \mu)} \leq\one_{H'}$, where $F_k , H \subseteq \rr R$ are sets of finite measure, and $H'\subseteq H$ is a major subset of $H$ to be constructed in the process.

The helicoidal method is a recursive procedure in which vector-valued estimates of depth $n$ corresponding to $\vec T_{n}$ are proved using \emph{localized} versions of the $\left(n-1\right)$-depth vector-valued operator $\vec T_{n-1}$. Aiming to prove \eqref{eq:intr-multilin-form-restr} for fixed sets $F_1, \ldots, F_m, H, H'$, we need to exercise great care in evaluating 
\[
\Lambda_{\vec T_{n-1}; I_0}^{F_1, \ldots, F_m, H'}:=\Lambda_{\vec T_{n-1}; I_0}\big( \vec f_1 \cdot \one_{F_1}, \ldots, \vec f_m \cdot \one_{F_m}, \vec h \cdot \one_{H'}  \big).
\]
There is another localization associated to the spatial dyadic interval $I_0$, hence the notation $\Lambda_{\vec T_{n-1}; I_0}$. This will be made precise later, but such a localization is natural in the time-frequency analysis setting, where operators are decomposed into wave packets that retain both spatial and frequential information.

The estimates needed for $\Lambda_{\vec T_{n-1}; I_0}^{F_1, \ldots, F_m, H'}$ correspond to Lebesgue exponents between $1$ and $\infty$, as a result of the assumption $1 < R_k, R' \leq \infty$. So one could estimate the multilinear form as in \eqref{eq:intr-multilin-form}, but as a matter of fact, a sharper result can be obtained by using the fact that the functions $\vec f_k$ are supported inside the sets $F_k$:
\begin{align}
\label{eq:intr-multilin-form-local}
\big\vert \Lambda_{\vec T_{n-1}; I_0}^{F_1, \ldots, F_m, H'}(\vec f_1, \ldots, \vec f_m, \vec h)  \big\vert &\lesssim \|\Lambda_{\vec T_{n-1}; I_0}^{F_1, \ldots, F_m, H'} \| \cdot\\
& \|\vec f_1\|_{ L^{r_1}\big( \rr{R}; L^{\tilde R_1}(\ii W, \mu) \big)} \cdot \ldots \cdot\|\vec f_m\|_{L^{r_m}\big( \rr{R}; L^{ \tilde R_m}(\ii W, \mu) \big)} \cdot\| \vec h\|_{ L^{r'}\big( \rr{R}; L^{ \tilde R'}(\ii W, \mu) \big)}. \nonumber
\end{align}

Some information on $\|\vec f_k\|_{  L^{\tilde R_k}(\ii W, \mu)}$ is preserved in the operatorial norm $\|\Lambda_{\vec T_{n-1}; I_0}^{F_1, \ldots, F_m, H'} \|$, which is necessary in the induction step. Obtaining the desired vector-valued inequalities amounts to transforming $L^{r_k}$ estimates into $L^{p_k}$ estimates, and this resembles an extrapolation principle. If $r_k \leq p_k$, H\"older's inequality and localizations play an important role, but in the case $r_k >p_k$, the sharp evaluation of $\|\Lambda_{\vec T_{n-1}; I_0}^{F_1, \ldots, F_m, H'} \|$ is essential. The method of the proof is described in greater detail in Section \ref{sec:def-layout}, after introducing some necessary definitions.

However, if some $r^j \notin [1, \infty)$ (recall that $R=(r^1, \ldots, r^n)$ corresponds to the target space in \eqref{eq:vv-depth-n}), we cannot expect to have an inequality comparable to \eqref{eq:intr-multilin-form-local} for the multilinear form. The difficulty consists, for example,  in associating a trilinear form to an operator $T:L^p(\ell^{r_1}) \times L^q(\ell^{r_2}) \to L^s(\ell^r)$ when $r<1$. The linearization of such an operator is achieved in the quasi-Banach case by dualizing the $L^{p, \infty}$ quasinorm through the Lebesgue space $L^r$ (see Section \ref{sec:quasi-Banach}).

We improve the helicoidal method from \cite{vv_BHT} by substituting \eqref{eq:intr-multilin-form-local} with
{\fontsize{9}{10}
\begin{equation}
\label{eq:intr-op-local-sharp}
\big\|\vec{T}_{n-1; I_0}(\vec f_1 \cdot \one_{F_1}, \ldots, \vec f_m \cdot \one_{F_m}) \cdot \one_{H'}\big\|_{L^r} \lesssim  \|\vec T_{n-1; I_0}^{F_1, \ldots, F_m, H'} \| \|\vec f_1\|_{ L^{r_1}\left( \rr{R}; L^{\tilde R_1}(\ii W, \mu) \right)} \cdot \ldots \cdot\|\vec f_m\|_{L^{r_m}\left( \rr{R}; L^{ \tilde R_m}(\ii W, \mu) \right)}.
\end{equation}}
Again, optimal estimates are needed for the operatorial norm $\|\vec T_{n-1; I_0}^{F_1, \ldots, F_m, H'} \|$, which are in some sense finer than those for the operatorial norm $\| \Lambda_{\vec T_{n-1}; I_0}^{F_1, \ldots, F_m, H'}\|$ of the multilinear form. 

All things considered, we are able to prove that whenever $(r_1, \ldots, r_m ,r) \in Range (T)$, we have also $Range(\vec T_{\vec r}^1) \neq \emptyset $. Here we write $\vec T^n_{\vec R}$ for the vector-valued operator of depth $n$ associated to the tuple of vectors $\vec R =(R_1, \ldots, R_m, R)$. Recursively, whenever $(r_1, \ldots, r_m, r) \in Range(\vec T^{n-1}_{\left(\tilde R_1, \ldots, \tilde R_m ,\tilde R  \right)})$, we are able to give a characterization of $Range(\vec T^{n}_{\left( R_1, \ldots, R_m , R  \right)})$, where, for all $1 \leq k \leq m$, $R_k:=(r_k, \tilde R_k)$.

We will illustrate how the helicoidal method produces quasi-Banach valued inequalities for two bilinear operators: the paraproduct $\Pi$ and the bilinear Hilbert transform $BHT$. Then the techniques extend to allow for certain multiple Banach or quasi-Banach valued inequalities (for us, that corresponds to multiple $L^p$ spaces with $0 < p \leq \infty$). It turns out that 
\[
Range(\vec{\Pi}^n_{\vec R})=Range(\Pi),
\]
i.e. for paraproducts, vector-valued extensions exist for all the Lebesgue exponents in the range of the scalar operator. This is the case for linear Calder\'on-Zygmund operators as well. For $BHT$ the situation is more complicated due to its singularity; we can prove nevertheless that whenever $Range(\overrightarrow{BHT}^n_{\vec R}) \neq \emptyset$, it contains the local $L^2$ range:
\[
\lbrace  \left(p, q, s  \right): 2\leq p, q, s' \leq \infty, \frac{1}{p}+\frac{1}{q}+\frac{1}{s'}=1 \rbrace \subseteq Range(\overrightarrow{BHT}^n_{\vec R}),
\]
for any $n \geq 1$ and any tuple of vectors $\vec R=(R_1, R_2, R)$ satisfying for all $1 \leq j \leq n$ $$(r_1^j, r_2^j, r^j) \in Range(\overrightarrow{BHT}^{n-j}_{\left( \left(r_1^{j+1}, \ldots, r_1^n \right), \left(r_2^{j+1}, \ldots, r_2^n \right), \left( r^{j+1}, \ldots, r^n  \right)  \right)}).$$

Our main motivation was finding the full range of mixed norm estimates for the operator $\Pi \otimes \Pi$, i.e. the biparameter paraproduct. Such estimates imply Leibniz rules in mixed-norm $L^p$ spaces and they can prove useful in the study of nonlinear dispersive PDE (particular cases of these inequalities were used in \cite{kenig2004_vvLeibniz}). In \cite{vv_BHT}, we proved that 
\begin{equation}
\label{eq:mix-norm-biparam-paraprod}
\big\|  \big\| \Pi \otimes \Pi(f, g)   \big\|_{L^{s_2}_y}  \big\|_{L^{s_1}_x} \lesssim \big\|  \big\| f \big\|_{L^{p_2}_y}  \big\|_{L^{p_1}_x} \big\|  \big\| g \big\|_{L^{q_2}_y}  \big\|_{L^{q_1}_x}
\end{equation}
whenever $\frac{1}{p_j}+\frac{1}{q_j}=\frac{1}{s_j}$, with $1<p_j , q_j \leq \infty$ for $j=1, 2$ and $\frac{1}{2} <s_1<\infty, 1 \leq s_2 <\infty$. 

A similar result was proved using different techniques in \cite{francesco_UMDparaproducts}. Both approaches  invoke vector-valued inequalities in the study of multi-parameter multilinear operators. These operators are intriguing because they don't always behave as expected. For example, in \cite{bi-parameter_paraproducts} and \cite{multi-parameter_paraproducts} it was shown that $\Pi \otimes \Pi$ is bounded, but $BHT\otimes BHT$ doesn't satisfy any $L^p$ estimates of H\"older type. The range of boundedness for $\Pi \otimes BHT$ was only recently understood in \cite{vv_bht-Prabath} and \cite{vv_BHT}.

The question left open from \cite{vv_BHT} and \cite{francesco_UMDparaproducts} concerns an inequality similar to \eqref{eq:mix-norm-biparam-paraprod} when $\frac{1}{2}<s_2<1$, corresponding to a quasi-Banach vector space. Following the methods from \cite{vv_BHT}, such an estimate is implied by the multiple vector-valued estimate 
\begin{equation}
\label{eq:proof-mixed-norm}
\Pi_x : L_x^{p_1}\big(  L_y^{p_2} \big(  \ell^\infty \big) \big) \times  L_x^{q_1}\big(  L_y^{q_2} \big(  \ell^2 \big) \big) \to L_x^{s_1}\big(  L_y^{s_2} \big(  \ell^2 \big) \big).
\end{equation}
Here we extend the helicoidal method to the context of Banach or quasi-Banach spaces, which will allow us to prove estimates as above. Hence, combined with our previous results from \cite{vv_BHT}, we have
\begin{theorem}
\label{thm:kenig-two-param}
Let $1<p_i, q_i \leq \infty$ and $\frac{1}{2}<s_i<\infty$, be so that $\ds \frac{1}{p_i}+\frac{1}{q_i}=\frac{1}{s_i}$ for any index $i=1, 2$. Then the bi-parameter paraproduct $\Pi \otimes \Pi$ satisfies the following mixed norm estimates:
\begin{equation}
\label{eq:mixed-norms-quasiBanach}
\Pi \otimes \Pi : L^{p_1}_xL^{p_2}_y \times  L^{q_1}_xL^{q_2}_y \to L^{s_1}_xL^{s_2}_y.
\end{equation}
\end{theorem}

The Leibniz rule implied by this theorem, as will be shown in Section \ref{sec:mixed_norm-est}, can be formulated in the following way: 
\begin{theorem}
\label{thm:Leibniz}
For any $\alpha, \beta >0$
\begin{align*}
\Big \| D_1^\alpha D^\beta_2 (f \cdot g)  \Big \|_{L^{s_1}_xL^{s_2}_y} &\lesssim \Big \| D_1^\alpha D^\beta_2 f  \Big \|_{L^{p_1}_xL^{p_2}_y} \cdot \|g\|_{L^{q_1}_xL^{q_2}_y} +\|f\|_{L^{p_3}_xL^{p_4}_y} \cdot \Big \| D_1^\alpha D^\beta_2 g  \Big \|_{L^{q_3}_xL^{q_4}_y} \\
&+\Big \| D_1^\alpha f  \Big \|_{L^{p_5}_xL^{p_6}_y} \cdot \Big \| D^\beta_2 g  \Big \|_{L^{q_5}_xL^{q_6}_y} +\Big \| D_2^\beta f  \Big \|_{L^{p_7}_xL^{p_8}_y} \cdot \Big \| D^\alpha_1 g  \Big \|_{L^{q_7}_xL^{q_8}_y} ,
\end{align*}
whenever $1< p_j,q_j \leq \infty$, $\frac{1}{1+\alpha} <s_1<\infty$, $\max \big( \frac{1}{1+\alpha}, \frac{1}{1+\beta}  \big) < s_2<\infty$, and the indices satisfy the natural H\"{o}lder-type conditions.
\end{theorem}

That $s_2$ should be greater than both $\frac{1}{1+\alpha}$ and $\frac{1}{1+\beta}$ is sensible, since this is also the natural condition in the case when $s_1=s_2$, as proved in \cite{bi-parameter_paraproducts}. On the other hand, for $s_1$ we only have one constraint: $\frac{1}{1+\alpha} <s_1<\infty$.

Paraproducts correspond to bilinear Calder\'on-Zygmund operators, but we will only consider bilinear Fourier multipliers as in \cite{CoifMeyer-ondelettes}, for simplicity. These are denoted $\Pi$ and it is known that $\Pi : L^p \times L^q \to L^s$ for all $1<p, q \leq \infty, \frac{1}{2}<s<\infty$, provided $\frac{1}{p}+\frac{1}{q}=\frac{1}{s}$.

The ``depth $1$" vector-valued inequality is formulated in the following theorem:

\begin{theorem}
\label{thm:one-quasiBanach-Paraprod}
Let $r_1, r_2$ and $r$ be positive numbers such that $1<r_1, r_2 \leq \infty$, $\dfrac{1}{2}< r <\infty$, and $\ds \frac{1}{r_1}+\frac{1}{r_2}=\frac{1}{r}$. Then the paraproduct $\Pi$ satisfies
\begin{equation}
\label{eq:quasi-banch-paraprod}
\big\| \big( \sum_{k} \lft \Pi \big(f_k, g_k  \big) \rg^r \big)^{1/r} \big\|_s \lesssim \big\| \big( \sum_{k} \lft f_k \rg^{r_1} \big)^{1/{r_1}} \big\|_p \cdot \big\| \big( \sum_{k} \lft g_k\rg^{r_2} \big)^{1/{r_2}} \big\|_q,
\end{equation}
for any $p, q, s$ with $1< p, q \leq \infty, \dfrac{1}{2}< s <\infty$, and $\ds \frac{1}{p}+\frac{1}{q}=\frac{1}{s}$.  
\end{theorem}


In the proof of Theorem \ref{thm:kenig-two-param}, we need a more general result, corresponding to a ``depth $n$" vector-valued inequality:

\begin{theorem}
\label{thm:main-thm-paraprod}
Consider the tuples $R_1=\big(r_1^1, \ldots, r_1^n \big), R_2=\big(r_2^1, \ldots, r_2^n \big)$ and $R=\big(r^1, \ldots, r^n \big)$ satisfying for every $1 \leq j \leq n: 1<r_1^j, r_2^j \leq \infty, \frac{1}{2}<r^j <\infty$, and $\frac{1}{r_1^j}+\frac{1}{r_2^j}=\frac{1}{r^j}$. Then the paraproduct $\Pi$ satisfies the estimates
\[
\Pi : L^p\big( \rr R; L^{R_1}\big( \ii W, \mu  \big)  \big) \times L^q\big( \rr R; L^{R_2}\big( \ii W, \mu  \big)  \big) \to L^s\big( \rr R; L^{R}\big( \ii W, \mu  \big)  \big),
\]  
for any $1<p, q \leq \infty, \frac{1}{2}<s<\infty$ with $\frac{1}{p}+\frac{1}{q}=\frac{1}{s}$. 
\end{theorem}

The multiple vector-valued estimates for $\Pi$ seem to be new only in the case when $L^\infty$ spaces are involved (that is, when one of $p, q, r_1^j$ or $r_2^j$ equals $\infty$). This is the case in \eqref{eq:proof-mixed-norm}, which in turn is necessary for Theorem \ref{thm:kenig-two-param}. In fact, in the bi-parameter analysis, the estimate \eqref{eq:proof-mixed-norm} corresponds to the boundedness of the square function and of the maximal function in the one-parameter study of paraproducts. Otherwise, multiple vector-valued inequalities for multilinear Calder\'on-Zygmund operators can be obtained as in \cite{extrapolation_multilinear} or \cite{multi-extrap-CU-martell-perez} by extrapolation, from weighted estimates.

The bilinear Hilbert transform is an operator given by
\[
BHT(f, g)(x):=p.v. \int_{\rr R} f(x-t) g(x+t) \frac{dt}{t}.
\]
It is known (\cite{initial_BHT_paper}) to satisfy $L^p \times L^q \to L^s$ estimates for $\frac{2}{3}<s<\infty$, if $\frac{1}{p}+\frac{1}{q}=\frac{1}{s}$, but the method of the proof breaks down for $\frac{1}{2}<s \leq \frac{2}{3}$, leaving open the question concerning the optimal range of boundedness. The multiplier of $BHT$ is $sgn (\xi -\eta)$, making it the prototype of a bilinear operator with a one-dimensional singularity in frequency.

When proving vector-valued inequalities for the bilinear Hilbert transform, there are more constraints appearing than in the paraproduct case. Our approach to vector-valued inequalities for $BHT$ uses the boundedness of the scalar operator and its localizations in an essential way, and in consequence it is not surprising that restrictions similar to the scalar case appear. We are however able to provide a wide range of vector-valued inequalities for $BHT$. Combining the present work with the result from our previous \cite{vv_BHT}, we have:

\begin{theorem}
\label{thm-main-vvBHT-1vs}
For any triple $( r_1, r_2, r)$ with $1<r_1, r_2 \leq \infty, \frac{2}{3}<r <\infty$ and so that $\frac{1}{r_1}+\frac{1}{r_2}=\frac{1}{r}$, there exists a non-empty set $\ii D_{ r_1, r_2, r}$ of triples $(p, q, s)$ satisfying $\frac{1}{p}+\frac{1}{q}=\frac{1}{s}$, for which
\[
\big \|  \big\| BHT(\vec f, \vec g)\big\|_{L^r\left( \ii W, \mu \right)} \big\|_{L^s\left( \rr R \right)} \lesssim \big \|  \big\|  \vec f  \big\|_{L^{r_1}\left( \ii W, \mu \right)}  \big\|_{L^p\left( \rr R \right)} \big \|  \big\|  \vec g \big\|_{L^{r_2}\left( \ii W, \mu \right)}  \big\|_{L^q\left( \rr R \right)}.
\]
In brief, $\ii D_{r, r_1, r_2}:=Range(\overrightarrow{BHT}_{\vec r})$ is conditioned by the existence of certain $0 \leq \theta_1, \theta_2, \theta_3<1$ with $ \theta_1+\theta_2+\theta_3=1$ so that 
\[
\frac{1}{r_1}<\frac{1+\theta_1}{2}, \quad \frac{1}{r_2}<\frac{1+\theta_2}{2}, \quad \frac{1}{r'}<\frac{1+\theta_3}{2}, 
\]
and, at the same time, 
\[
\frac{1}{p}<\frac{1+\theta_1}{2}, \quad \frac{1}{q}<\frac{1+\theta_2}{2}, \quad \frac{1}{s'}<\frac{1+\theta_3}{2}.
\]
\end{theorem}

The set $\ii D_{ r_1, r_2, r}$ can be given an explicit characterization, depending on the values of $r_1, r_2, r$:
\begin{enumerate}[label=(\roman*), ref=\roman*, leftmargin=*]
\item \label{lbl:localL2} If $\ds \frac{1}{r_1}, \frac{1}{r_2}, \frac{1}{r'} \leq \frac{1}{2} $, then $\ii{D}_{r_1, r_2, r}=Range(BHT)$.
\item \label{lbl:r_1<2}  If $\ds\frac{1}{r_2},  \frac{1}{r'} \leq \frac{1}{2},  \frac{1}{r_1}>\frac{1}{2}$, then $\ii{D}_{r_1, r_2, r}$ corresponds to those $(p, q, s)\in Range(BHT)$ for which  $\ds 0\leq \frac{1}{q}< \frac{3}{2}-\frac{1}{r_1}.$                                                                                                                                                          
\item \label{lbl:r_2<2}  If $\ds\frac{1}{r_1},  \frac{1}{r'} \leq \frac{1}{2},  \frac{1}{r_2}>\frac{1}{2}$, then  the range of exponents is similar to the  one in $ii)$, with the roles of $r_1$ and $r_2$ interchanged. That is, $\ii{D}_{r_1, r_2, r}$ consists of tuples  $(p, q, s)\in Range(BHT)$ for which $\ds 0\leq \frac{1}{p}< \frac{3}{2}-\frac{1}{r_2}.$
\item \label{lbl:r'<2} If $\ds\frac{1}{r_1},  \frac{1}{r_2} \leq \frac{1}{2},  \frac{1}{r'}>\frac{1}{2}$, then $\ii{D}_{r_1, r_2, r}$ corresponds to those $(p, q, s) \in Range(BHT)$ for which $\ds 0 \leq \frac{1}{p}, \frac{1}{q}<\frac{1}{2}+\frac{1}{r}, \quad -\frac{1}{r} < \frac{1}{s'}<1.$
\item \label{lblr<1andr_1<2} If $\ds \frac{1}{r_1}>\frac{1}{2}, \frac{1}{r_2} \leq \frac{1}{2}$ and $-\dfrac{1}{2} < \dfrac{1}{r'} <0$, then $\ii{D}_{r_1, r_2, r}$ corresponds to those $(p, q, s) \in Range(BHT)$ for which $\ds 0 \leq \frac{1}{q} <\frac{3}{2}-\frac{1}{r_1}, \frac{1}{s'}<\frac{3}{2}-\frac{1}{r_1}$.
\item \label{lblr<1andr_2<2} If $\ds \frac{1}{r_2}>\frac{1}{2}, \frac{1}{r_1} \leq \frac{1}{2}$ and $-\dfrac{1}{2} < \dfrac{1}{r'} <0$, then $\ii{D}_{r_1, r_2, r}$ corresponds to those $(p, q, s) \in Range(BHT)$ for which $\ds 0 \leq \frac{1}{p} <\frac{3}{2}-\frac{1}{r_2}, \frac{1}{s'}<\frac{3}{2}-\frac{1}{r_2}$.
\item \label{lblr<1andr_2<2andr_1<2} If $\ds \frac{1}{r_1}>\frac{1}{2}, \frac{1}{r_2} > \frac{1}{2}$ and $-\dfrac{1}{2} < \dfrac{1}{r'} <0$, then $\ii{D}_{r_1, r_2, r}$ corresponds to those $(p, q, s) \in Range(BHT)$ for which $\ds 0 \leq \frac{1}{p}<\frac{3}{2}-\frac{1}{r_2}, 0 \leq \frac{1}{q} <\dfrac{3}{2}-\frac{1}{r_1}$ and $\dfrac{1}{s'}< 2 -\dfrac{1}{r}$.
\end{enumerate}

While \eqref{lbl:localL2}-\eqref{lbl:r'<2} are from \cite{vv_BHT}, the rest of the estimates are new and correspond to the quasi-Banach case $\frac{2}{3}<r<1$.  We note that whenever $(r_1, r_2, r)$ is contained in the ``local $L^2$" triangle (that is, $0 \leq \frac{1}{r_1}, \frac{1}{r_2}, \frac{1}{r'} \leq \frac{1}{2}$), then $Range(\overrightarrow{BHT}_{\vec r}) =Range(BHT)$. In return, if $(p, q, s)$ is contained in the ``local $L^2$" triangle, the bilinear Hilbert transform admits a vector-valued extension $\overrightarrow{BHT}_{\vec r} :  {L^p\big( \rr R ; {L^{r_1}\big( \ii W, \mu \big)} \big)}\times {L^q\big( \rr R ; {L^{r_2}\big( \ii W, \mu \big)} \big)} \to {L^s\big( \rr R ; {L^{r}\big( \ii W, \mu \big)} \big)}$, for any $(r_1, r_2, r) \in Range(BHT)$.

Finally, we can obtain a result similar to Theorem \ref{thm:main-thm-paraprod}, for multiple vector spaces:
\begin{theorem}
\label{thm:main-thm-BHT}
Consider the tuples $R_1=(r_1^1, \ldots, r_1^n ), R_2=(r_2^1, \ldots, r_2^n)$ and $R=(r^1, \ldots, r^n )$ satisfying for every $1 \leq j \leq n: 1<r_1^j, r_2^j \leq \infty, \frac{2}{3}<r^j <\infty$, and $\frac{1}{r_1^j}+\frac{1}{r_2^j}=\frac{1}{r^j}$. If for every $1 \leq j \leq n-1$ we have $( r_1^j, r_2^j, r^j ) \in \ii D_{r_1^{j+1}, r_2^{j+1}, r^{j+1}}$, then there exists a non-empty set $\ii D_{R_1, R_2, R}:=Range(\overrightarrow{BHT}^n_{\vec R})$ of triples $(p, q, s)$ for which 
\[
BHT: L^p\big( \rr R; L^{R_1}\big( \ii W, \mu \big) \big) \times L^q\big( \rr R; L^{R_2}\big( \ii W, \mu \big) \big) \to L^s\big( \rr R; L^{R}\big( \ii W, \mu \big) \big).\]
Furthermore, $\ds \ii D_{R_1, R_2, R} =\bigcap_{1 \leq j \leq n} \ii D_{r_1^j, r_2^j, r^j}$.
\end{theorem}

The paper is organized as follows: in the next two sections, we set the stage for the actual proofs: in Section \ref{sec:quasi-Banach} we recall a few useful things about quasi-Banach spaces, and in Section \ref{sec:def-layout} we present a few technical results, as well as the outline of the general method of the proof. Theorem \ref{thm:one-quasiBanach-Paraprod} is proved in Section \ref{sec:quasi-banach-Pi}, Theorem \ref{thm:main-thm-paraprod} in Section \ref{sec:multiple vector-valued inequalities}, and Theorem \ref{thm:main-thm-BHT} in Section \ref{sec:BHT}. The bi-parameter paraproducts and the Leibniz rules in Theorem \ref{thm:Leibniz} are discussed in Section \ref{sec:mixed_norm-est}. 
\subsection*{Acknowledgements} The first author was partially supported by NSF grant DMS 1500262  and ERC project FAnFArE no. 637510; the second author was partially supported by NSF grant DMS 1500262.

\section{Quasi-Banach spaces: a short review}
\label{sec:quasi-Banach}
The quasi-Banach spaces concerned in the present paper are $L^r$ spaces, with $\frac{1}{2}<r<1$. In a first instance, they will represent target spaces for bilinear operators. For our most general results, which are Theorems \ref{thm:main-thm-paraprod} and \ref{thm:main-thm-BHT}, we consider multiple vector spaces that are either quasi-Banach or Banach $L^p$ spaces. More exactly, let $R^n=( r^1, \ldots, r^n )$ be an $n$-tuple so that for every $1 \leq j \leq n$, $\frac{1}{2}<r^j <\infty$. We will be using the notation $\ds \| \Phi \|_{L^{R^n}}$ for the mixed $L^p$ norm:
\[
\| \Phi \|_{L^{R^n}}:=\| \Phi \|_{L^{r^1}( \ldots ( L^{r^n} ) )}.
\]
The following observation will be used throughout the paper:
\begin{proposition}
\label{prop:reorder}
Let $R^n=( r^1, \ldots, r^n  )$, with at least one Lebesgue exponent corresponding to a quasi-Banach space. Then $\ds \| \cdot \|_{L^{R^n}}^{r^{j_0}}$ is subadditive, where $j_0$ is any index for which $\ds r^{j_0}=\min_{1 \leq j \leq n } r^j$.
\begin{proof}
We prove the above statement by induction. The case $n=1$ is trivial, as there is only one $L^r$ space, which is quasi-Banach, and $\| \cdot \|_r^r$ is subadditive.

We assume the statement to be true for any tuple of length $n-1$, and prove it for a tuple $R^n=( r^1, \ldots, r^n )$. First, we note that $R^n=(r^1, \tilde R^{n-1})$, where $\tilde R^{n-1}=(r^2, \ldots, r^{n})$ is a tuple of length $n-1$. We don't know for sure if $\tilde R^{n-1}$ satisfies the hypothesis in the proposition, but if it doesn't, all the Lebesgue exponents are $\geq 1$, and $\| \cdot \|_{\tilde R^{n-1}}$ is subadditive. It also means that $r_1<1$, and in consequence, $\| \cdot \|_{R^n}^{r_1}$ is subadditive.

Otherwise, $\tilde R^{n-1}$ satisfies the induction hypotheses, and  $\ds \| \cdot \|_{L^{\tilde{R}^{n-1}}}^{r^{j_0}}$ is subadditive, where $\ds r^{j_0} :=\min_{2 \leq r^j \leq n} r^j$. We note that 
\[
 \| \cdot \|_{L^{R^n}}= \big\|   \big( \big\| \cdot \big\|_{L^{\tilde R^{n-1}}}^{r^{j_0}} \big)^{1/{r^{j_0}}} \big\|_{r^1}=
 \big \|    \big\| \cdot \big\|_{L^{\tilde R^{n-1}}}^{r^{j_0}} \big\|_{L^{\frac{r^1}{r^{j_0}}}}^{\frac{1}{r^{j_0}}}.
\]

If $\dfrac{r^1}{r^{j_0}} \geq 1$, then $\ds \| \cdot \|_{L^{{R}^{n}}}^{r^{j_0}}$ is subadditive; otherwise, $\ds \| \cdot \|_{L^{R^n}}^{r^{1}}$
is subadditive, which means the induction statement is true for an index $j_1$ corresponding to the minimal $r^j$.

\end{proof}
\end{proposition}

\begin{remark}
We can reformulate the proposition above in a way that includes the Banach case: if $R^n:=( r^1, \ldots, r^n )$, with $0 < r^j \leq\infty$, then $\ds  \| \cdot \|_{L^{R^n}}^{r^{j_0}}$ is subadditive, where $\ds r^{j_0}:=\min ( 1, \min_{1 \leq j \leq n} r^j  )$. 
\end{remark} 

\subsection*{Dualization through $L^r$ spaces}~\\
\label{subsection:dualization-L^r}
Because the triangle inequality is missing, the duals of $L^p$ quasi-Banach spaces are either too simple (the case of non-atomic spaces, where the dual is $\lbrace 0 \rbrace$), or too complicated (atomic spaces, such as $\ell^p$, the dual of which contains $\ell^\infty$, but doesn't have an exact  characterization). However, for \emph{weak} $L^p$ spaces, the quasinorm can be ``dualized" by using generalized restricted weak-type estimates:
\begin{equation}
\label{eq:dualization-L1}
\big\|f\big\|_{p, \infty} \sim \sup_{E,  0<\lft E \rg< \infty } \inf_{\substack{  E' \subseteq E \\ \text{major subset}} }
\frac{\lft \left \langle  f, \one_{E'}   \right \rangle\rg}{\lft E \rg^{1-\frac{1}{p}}},
\end{equation}
where we say $E'$ is a major subset of $E$ if $E'\subseteq E$ and $\ds \lft E'\rg \geq \lft E \rg/2$.

One can have an equivalent statement by making use of $L^r$ norms. This is very similar to Lemma 2.6 of \cite{multilinear_harmonic}:

\begin{proposition}
\label{prop:Lr-dualization}
The following are equivalent:
\begin{itemize}
\item[(i)]$\ds \|  f \|_{p, \infty} \leq A$
\item[(ii)] For any set $E$ of finite measure, there exists a major subset $\tilde{E} \subseteq E$ so that 
\[
 \|  f \cdot \one_{\tilde{E}} \|_{r} \lesssim A \cdot \lft E \rg^{\frac{1}{r}-\frac{1}{p}}.
 \]
\end{itemize}
This can be reformulated as
\begin{equation}
\label{eq:dualization-Lr}
\| f \|_{p, \infty} \sim \sup_{0< \lft E \rg< \infty} \inf_{\substack{\tilde E \subseteq E \\ \text{major subset}}} \frac{\| f \cdot \one_{\tilde E}  \|_r}{\lft E \rg^{\frac{1}{r}-\frac{1}{p}}}.
\end{equation}
\begin{proof}
$\ds `` (i) \Rightarrow (ii) "$ Let $E$ be a set of finite measure and consider the set
\[
\Omega:=\big\lbrace x: \lft f(x)\rg> C \cdot \frac{A}{\lft E \rg^{1/p}} \big\rbrace.
\]
Pick $\tilde{E}$ to be $\ds \tilde{E}:=E \setminus \Omega$, which is a major subset of $E$ for $C$ large enough; this is the only place where we are using inequality $(i)$. Then it's easy to see that
\[
\| f \cdot \one_{\tilde{E}}\|_r \lesssim C \cdot \frac{A}{\lft E \rg^{1/p}}  \cdot \lft E \rg^{1/r} \lesssim A \cdot \lft E \rg^{\frac{1}{r}-\frac{1}{p}}.
\]

$\ds `` (ii) \Rightarrow (i)"$ Let $\lambda >0$ and set $\ds E:=\lbrace  x: \lft f(x)\rg>\lambda \rbrace$. From $(ii)$, we know there exists a major subset $\tilde{E}$ of $E$ for which we have
\[
\lambda \vert \tilde{E} \vert^{1/r} < \|  f \cdot \one_{\tilde{E}} \|_{r} \lesssim A \cdot \lft E \rg^{\frac{1}{r}-\frac{1}{p}}.
\]
Since $\lft E \rg$ and $\vert \tilde{E}\vert$ are comparable, we get that $\ds \lambda \lft E \rg^{1/p} \lesssim A$. Since $\lambda >0$ was arbitrary, we deduce $(i)$.
\end{proof}
\end{proposition}

\begin{remark}
Regarding the notation, from now on we will denote $E'$ a major subset of $E$ whenever we dualize the $\|\cdot \|_{L^{p, \infty}}$ quasinorm by identity \eqref{eq:dualization-L1}, and $\tilde E$ when we use a different $L^r$ space. 
\end{remark}

\subsection*{An Interpolation Theorem}~\\
\label{subsec:interpolation}
Interpolation for linear or multi-linear Banach-valued operators is almost identical to the scalar case, with the difference that $|  \cdot |_{\rr C}$ is replaced by the norm of the Banach space $X$, $\| \cdot \|_X$. In the case of bilinear operators, we will prove a quasi-Banach-valued interpolation result, which is the natural modification of the Banach-valued case. First, we recall a new definitions:

\begin{definition}
A tuple $( \alpha_1, \alpha_2, \alpha_3 )$ is called \emph{admissible} if $\alpha_1+\alpha_2+\alpha_3=1$, $-1<\alpha_1, \alpha_2, \alpha_3<1$ and $\alpha_j \leq 0$ for at most one index.

In a similar way, a triple of Lebesgue exponents $(  p, q, s )$ is called \emph{admissible} if $\big(  \frac{1}{p}, \frac{1}{q} , \frac{1}{s'}\big)$ is admissible according to the above definition. In most cases, we will have $1<p, q \leq \infty$ and $\frac{1}{2}< s< \infty$, with $\frac{1}{p}+\frac{1}{q}=\frac{1}{s}$.
\end{definition}

\begin{proposition}
\label{prop:interpolation}
Let $ \frac{1}{2} <r < 1$, and $1 < r_1, r_2 \leq \infty$ be so that $\ds \frac{1}{r}=\frac{1}{r_1}+\frac{1}{r_2}$. The tuple $\left( p, q, s \right)$ is so that $\frac{1}{2}< s <\infty, 1 <p, q < \infty $ and $\ds \frac{1}{s}=\frac{1}{p}+\frac{1}{q}$, and $T$ is a bilinear operator satisfying the restricted type estimate:

for any sets of finite measure $E_1$ and $E_2$, and any sequences of functions $\lbrace f_k \rbrace_k, \lbrace g_k \rbrace_k$ so that $\ds \big( \sum_k \lft f_k \rg^{r_1} \big)^{1/{r_1}}\leq \one_{E_1}$ and $\ds \big( \sum_k \lft g_k \rg^{r_2} \big)^{1/{r_2}} \leq \one_{E_2}$ respectively, the estimate
\begin{equation}
\label{eq:interp-restricted-weak-type}
\big\| \big( \sum_k \lft T( f_k, g_k) \rg^{r} \big)^{1/{r}}  \big\|_{\tilde s , \infty} \leq K_{ s_1, s_2, \tilde s} \lft E_1\rg^{1/{s_1}} \lft E_2\rg^{1/{s_2}}
\end{equation}
holds for all admissible tuples $\left( s_1, s_2 , \tilde s \right)$ in a neighborhood of $\left( p, q, s \right)$, with the constant $K_{s_1, s_2, \tilde s}$ depending continuously on $s_1, s_2, \tilde s$. 

Then $T$ is of strong type $\left( p, q, s \right)$, in the sense that 
\begin{equation}
\label{eq:qB-interp-strong}
\big\| \big( \sum_k \lft T( f_k, g_k) \rg^{r} \big)^{1/{r}}  \big\|_s \leq K_{ p, q, s} \big\| \big( \sum_k \lft f_k \rg^{r_1} \big)^{1/{r_1}} \big\|_p \big\|\big( \sum_k \lft g_k \rg^{r_2} \big)^{1/{r_2}} \big\|_q
\end{equation}
for any sequences of functions for which the RHS is finite.
\end{proposition}

The proof will be postponed to the last Section \ref{sec:proofs-interp-etc}, but it is nothing more than an adaptation of the classical argument.

A similar result can be formulated for mixed norm $L^p$ spaces analogous to those appearing in Proposition \ref{prop:reorder}, and for that reason proving restricted weak-type estimates as in \eqref{eq:interp-restricted-weak-type} will be sufficient for establishing Theorem \ref{thm:main-thm-paraprod} or Theorem \ref{thm:main-thm-BHT}. Moreover, the measures involved can be arbitrary (we will need this result with Lebesgue measures replaced by $\ci_{I_0} dx$ measures).

\begin{proposition}
\label{prop:interp-multi}
Let $R_1=\big(r_1^1, \ldots, r_1^n \big), R_2=\big(r_2^1, \ldots, r_2^n \big)$ and $R=\big(r^1, \ldots, r^n \big)$ be three tuples satisfying for every $1 \leq j \leq n: 1<r_1^j, r_2^j \leq \infty, \frac{1}{2}<r^j <\infty$, and $\frac{1}{r_1^j}+\frac{1}{r_2^j}=\frac{1}{r^j}$. Assume $T$ is a bilinear operator satisfying the restricted type estimate:

for any sets of finite measure $E_1$ and $E_2$, and any functions $\vec f, \vec g$ so that $\big\|  \vec f(x) \big\|_{L^{R_1}\left(\ii  W, \mu \right)} \leq \one_{E_1}(x)$ and $\big\|  \vec g(x) \big\|_{L^{R_2}\left(\ii  W, \mu \right)} \leq \one_{E_2}(x)$ respectively, the estimate
\begin{equation}
\label{eq:multi-restric-weak}
\big\|   \| T(\vec f, \vec g) \|_{L^R\left(\ii W, \mu  \right)} \big\|_{L^{\tilde s, \infty}\left( \nu \right)} \leq K_{ s_1, s_2, \tilde s} \cdot \nu_1 \left( E_1\right)^{1/{s_1}} \nu_2 \left(E_2\right)^{1/{s_2}}
\end{equation}
holds for all admissible tuples $\left( s_1, s_2, \tilde s \right)$ in a neighborhood of $\left( p, q, s \right)$, with the constant $K_{ s_1, s_2, \tilde s}$ depending continuously on $ s_1, s_2, \tilde s$. 

Then $T$ is of strong type $\left(p, q, s \right)$, in the sense that 
\[
\big\|   \| T(\vec f, \vec g) \|_{L^R\left(\ii W, \mu  \right)} \big\|_{L^s\left( \nu \right)} \leq K_{ p, q, s} \big\|  \| \vec f  \|_{L^{R_1}} \big\|_{L^p\left( \nu_1 \right)} \big\|  \| \vec g  \|_{L^{R_2}} \big\|_{L^q\left( \nu_2 \right)}.
\]
\end{proposition}

\section{Definitions and a layout of the proof}
\label{sec:def-layout}
In this section, we present the main definitions, as well as a sketch of the method of the proof of our theorems. In addition, we introduce some notation conventions and discuss in detail certain technicalities that are recurrent in the paper.

\subsection{Definitions}

\begin{notation}
Given an interval $I$, we denote by $\ci_I$ the function 
\begin{equation}
\label{def:adapted-fnc}
\ci_I(x):= \big( 1+\frac{ \dist(x, I)}{|I|} \big)^{-10}.
\end{equation}

The exponent in the above expression can change all through the presentation, and it can even depend on certain values of $p, q, s$. This will be only implicit in our estimates, as we attempt to keep the notation simple. 

In a similar way, the \emph{sizes} that will be introduced shortly, will appear with exponent $1-\epsilon$. The $\epsilon$ represents a small error, but we will not be tracking its exact value. For example, upon an application of H\"older's inequality, we obtain $\epsilon_{\text{final}}= \epsilon_{\text{initial}} \cdot \frac{r_j}{r}$; however, we denote both expressions by $\epsilon$, since the final error term can be made arbitrarily small.
\end{notation}

\begin{definition}
A collection $\lbrace\phi_I\rbrace_I$ of smooth bump functions associated to a family $\ii I$ of dyadic intervals is called \emph{lacunary} if  $$\supp \hat{\phi}_I \subseteq \big[ \frac{1}{|I|}, \frac{2}{|I|} \big], \quad \text{and} \quad |\partial^\alpha \phi_I(x) | \lesssim |I|^{-\frac{1}{2}-|\alpha|} \ci_I^M(x).$$

Similarly, a collection $\lbrace\phi_I\rbrace_I$ of smooth bump functions associated to a family $\ii I$ of dyadic intervals is called \emph{non-lacunary} if  $$\supp \hat{\phi}_I \subseteq \big[ 0, \frac{1}{|I|} \big], \quad \text{and} \quad |\partial^\alpha \phi_I(x) | \lesssim |I|^{-\frac{1}{2}-|\alpha|} \ci_I^M(x).$$

For lacunary collections we use the notation $\lbrace \psi_I \rbrace_I$, and for non-lacunary, $\lbrace\varphi_I\rbrace_I$.
\end{definition}

\begin{definition}
\label{def:paraproduct}
The \emph{discretized paraproduct} $\Pi_{\ii I}$ associated to a family $\ii I$ of dyadic intervals is the bilinear expression
\begin{equation}
\label{eq:def-discr-par}
\Pi_{\ii I} (f,g)(x)=\sum_{I \in \ii I} c_I \frac{1}{\lft I \rg^{\frac{1}{2}}}  \langle f, \varphi_I \rangle \langle g, \psi_I \rangle \psi_I(x),
 \end{equation}
where $\lbrace c_I \rbrace_{I \in \ii I}$ is a bounded sequence of complex numbers.
\end{definition}

In proving the Leibniz rule in Section \ref{sec:mixed_norm-est}, a special role is played by the paraproducts arising form the classical decomposition into ``low" and ``high" frequencies. We have 
{\fontsize{10}{10}
\begin{align}
\label{def:classicalParap}
f\cdot g (x)&=\sum_{k } \big(f \ast \varphi_k      \cdot g \ast \psi_k   \big) \ast \psi_k(x) +\sum_{k } \big(f \ast \psi_k    \cdot g \ast \varphi_k   \big) \ast \psi_k(x) +\sum_{k } \big(f \ast \psi_k      \cdot g \ast \psi_k   \big) \ast \varphi_k(x)\\
&=\sum_{k} Q_k(P_kf \cdot Q_k g)(x)+\sum_{k} Q_k(Q_kf \cdot P_k g) +\sum_{k} P_k(Q_kf \cdot Q_k g). \nonumber
\end{align}}
Here $\psi_k(x)=2^k \psi(2^k x)$, $\varphi_k(x)=2^k \varphi(2^k x)$,  $\hat{\varphi}(\xi) \equiv 1$ on $\ds \left[-1/2, 1/2 \right]$, is supported on $[-1,1]$  and $\hat{\psi}(\xi)=\hat{\varphi}(\xi/2)-\hat{\varphi}(\xi) $. The $\ds \lbrace Q_k \rbrace_k$ represent Littlewood-Paley projections onto the frequency $|\xi| \sim 2^k$, while $\ds \lbrace P_k \rbrace_k$ are convolution operators associated with dyadic dilations of a nice bump function of integral 1. We refer to any of the expressions on the right hand side of \eqref{def:classicalParap} as \emph{classical paraproducts}.

Bilinear Fourier multipliers of the form
\[
(f, g) \mapsto \int_{\rr R^2} \hat{f}(\xi) \hat g(\eta) m(\xi, \eta) e^{2 \pi i \left( \xi+\eta \right)} d \xi d \eta
\]
given by a multiplier $m(\cdot, \cdot)$ which is smooth away from the origin, can be expressed as superpositions of classical paraproducts, and hence as a superposition of operators analogous to those in \eqref{eq:def-discr-par}. The boundedness of the discretized paraproducts will imply that of the classical paraproducts and of the bilinear Fourier multipliers above. For this reason, we only study the discretized paraproducts.

\begin{definition}
Let $\mathscr{I}$ be a family of dyadic intervals. For any $1 \leq j \leq 3$, we define
\[
\ssize_{\mathscr{I}}\big( \langle f, \phi_I^j  \rangle_{I \in \mathscr{I}}   \big)=\sup_{I \in \ii{I}} \frac{|\langle f, \phi_I^j \rangle|}{|I|^{1/2}}, \quad \text{ if $(\phi_I^j)_I$ is non-lacunary and}
\]
\[
\ssize_{\mathscr{I}}\big( \langle f, \phi_I^j  \rangle_{I \in \mathscr{I}}   \big)=\sup_{I_0 \in \ii{I}}  \frac{1}{|I_0|} \| \big( \sum_{\substack{I \subseteq I_0\\ I \in \ii{I}}} \frac{|\langle f, \phi_I^j \rangle|^2}{|I|}  \cdot \one_I \big)^{1/2}    \|_{1, \infty}, \quad \text{ if $(\phi_I^j)_I$ is lacunary.}
\]
The energy is defined as
\[
\eenergy^j_{\ii{I}} \big( \langle f, \phi_I^j   \rangle_{I \in \ii{I}}   \big):=\sup_{n \in \rr{Z}} 2^n \sup_{\rr{D}} (\sum_{I \in \rr{D}}|I|)
\]
where $\rr{D}$ ranges over all collections of disjoint intervals $I_0$ with the property that 
\begin{align*}
& \frac{|\langle f, \phi_{I_0}^j \rangle|}{|I_0|^{1/2}} \geq 2^n, \quad \text{ if $(\phi_{I}^j)_I$ is non-lacunary and} \\
&\frac{1}{|I_0|^{1/2}} \| \big( \sum_{\substack{I \subseteq I_0\\ I \in \ii{I}}} \frac{|\langle f, \phi_I^j \rangle|^2}{|I|}  \cdot \one_I \big)^{1/2}    \|_{1, \infty} \geq 2^n, \quad \text{ if $(\phi_I^j)_I$ is lacunary.}
\end{align*}
\end{definition}

\begin{lemma}[Lemma 2.13 of \cite{multilinear_harmonic}]\label{lemma:size-est-paraprod}
 If $F$ is an $L^1$ function and $1 \leq j \leq 3$, then
\[
\ssize_{\ii{I}}^j(\langle F, \phi_I^j  \rangle_{I \in \ii{I}}) \lesssim \sup_{I \in \ii{I}}\frac{1}{|I|} \int_{\rr{R}} |F| \ci_{I}^M dx
\]
for any $M>0$, with implicit constants depending on $M$.
\end{lemma}

\begin{lemma}[Lemma 2.14 of \cite{multilinear_harmonic}]
\label{lemma:energy-est-paraprod}
If $F$ is an $L^1$ function and $1 \leq j \leq 3$, then 
$$\ds \eenergy_{\ii{I}}^j(\langle F, \phi_I^j  \rangle_{I \in \ii{I}}) \lesssim \| F \|_1.$$
\end{lemma}

It is customary to study the trilinear form associated to $\Pi_{\ii I}$ rather than the operator itself. The trilinear form is the expression 
\[
\Lambda_{\Pi}(f, g, h):=\sum_{I \in \ii{I}} c_I \frac{1}{|I|^{1/2}} \langle f, \phi_I^1 \rangle \langle g, \phi_I^2 \rangle \langle h, \phi_I^3 \rangle,
\]
and it can be estimated using the above sizes and energies.

\begin{proposition}[Proposition 2.12 of \cite{multilinear_harmonic}]
\label{paraproduct estimates}
Given a paraproduct $\Pi$ associated to a family $\ii{I}$ of intervals,
\begin{align*}
 \Big |\Lambda_{\Pi}(f_1, f_2, f_3) \Big| \lesssim \prod_{j=1}^3 \big( \ssize_{\ii{I}}^{(j)} ( \langle f_j, \phi_I^j  \rangle_{I \in \ii{I}} )  \big)^{1-{\theta_j}} \big( \eenergy_{\ii{I}}^{(j)} (\langle f_j, \phi_I^j  \rangle_{I \in \ii{I}}   )  \big)^{\theta_j},
\end{align*}
for any $0 \leq \theta_1, \theta_2, \theta_3 <1$ such that $\theta_1+\theta_2+\theta_3=1$, where the implicit constant depends on $\theta_1, \theta_2, \theta_3$ only.
\end{proposition}

While the above proposition is the main ingredient, we need localized estimates.

\begin{definition}
\label{def:localized-collections}
If $I_0$ is some fixed dyadic interval, we define
\[
\ii I \left( I_0 \right):=\left \lbrace I \in \ii I : I \subseteq I_0  \right\rbrace.
\]

For a collection $\ii I$ of intervals, we denote
\[
\ii I ^+ :=\left \lbrace J \text{ dyadic interval }:  \exists I \in \ii I \text{ such that } I \subseteq J  \right\rbrace.
\]
In the particular case of the collection $\ii I \left( I_0  \right)$, we have
\[
\ii I ^+\left( I_0 \right):= \left \lbrace J \text{ dyadic interval }:  J \subseteq 3 I_0 \text{  and  } \exists I \in \ii I \text{ such that } I \subseteq J  \right\rbrace.
\]
\end{definition}

\begin{definition}[Modified Size]
We define the following size, which is more suitable for localizations:
\begin{equation}
\label{eq:def-mod-size}
\sssize_{\ii I} f:= \sup_{J \in \ii I ^+} \frac{1}{\lft J \rg} \int_{\rr R} \lft f \rg \cdot \ci_{J}^M dx.
\end{equation}
We note that, thanks to Lemma \ref{lemma:size-est-paraprod}, we have $\ds \ssize_{\ii I } f \lesssim \sssize_{\ii I} f$.

In the particular case of a collection localized to a certain interval $I_0$, we use the notation
\begin{equation}
\label{def-size_I_0}
\sssize_{I_0} f :=\sssize_{\ii I ^+\left( I_0 \right)} f.
\end{equation}
\end{definition}
\vspace{.3 cm}

\subsection{A few Technical Results: Localizations}~\\
Throughout this section, we consider $I_0$ to be a fixed interval, and we use the notation $\ds \Pi_{I_0}:=\Pi_{\ii I \left( I_0\right)}$. That is, 
\[
\Pi_{I_0}(f, g)(x)=\sum_{\substack{I \in \ii{I} \\ I \subseteq I_0}} c_I \frac{1}{|I|^{1/2}} \langle f, \varphi_I \rangle \langle g, \psi_I \rangle \psi_I(x). 
\]

\begin{lemma}[Refinement of Lemma \ref{lemma:energy-est-paraprod}]
\label{lemma-ref-energy-est-paraprod}
If $f$ is a function whose support has the property that 
\[
2^{k-1} \leq 1+ \frac{\dist \left( \supp f, I_0 \right) }{\lft I_0 \rg} \leq 2^k,
\]
for some $k \geq 1$, then $\eenergy_{\ii I \left( I_0 \right)} (f) \lesssim 2^{-kM} \big\| f \big\|_1$. 
\begin{proof}
The proof is almost identical to that of Lemma \ref{lemma:energy-est-paraprod}, with the only difference that now we have, for any interval $I \subseteq I_0$,
\[
\frac{1}{| I |} \lft \langle f, \ci_I^{2M} \rangle \rg \lesssim 2^{-kM} \inf_{y \in I} \ic M f(y).
\]
\end{proof}
\end{lemma}

\begin{lemma}[Refinement of Proposition \ref{paraproduct estimates}]
\label{lemma:localized-paraprod-est}
For any functions $f, g$ and $h$, the trilinear form associated to the paraproduct $\Pi_{I_0}$ satisfies
\begin{align*}
\lft \Lambda_{I_0} \left( f, g, h \right) \rg & \lesssim \big( \sssize_{I_0} f \big)^{1-\theta_1} \big( \sssize_{I_0} g \big)^{1-\theta_2} \big( \sssize_{I_0} h \big)^{1-\theta_3} \big\| f \cdot \ci_{I_0} \big\|_1^{\theta_1} \big\| g\cdot \ci_{I_0} \big\|_1^{\theta_2} \big\| h \cdot \ci_{I_0} \big\|_1^{\theta_3}
\end{align*}
for any $0 \leq \theta_1, \theta_2, \theta_3<1$ such that $\theta_1+\theta_2+\theta_3=1$, where the implicit constants depend on $\theta_j$ only.
\begin{proof}
We first assume that $0< \theta_j <1$. Then write $f$ as $f=\sum_{k_1 \geq 0}f_{k_1}$, where $f_{k_1}=f \cdot \one_{\left \lbrace x:  \dist \left( x, I_0 \right) \sim \left( 2^{k_1}-1\right) | I_0 | \right\rbrace}$. Similarly,
\[
g=\sum_{k_2 \geq 0} g_{k_2}, \qquad \text{and}  \qquad h=\sum_{k_3 \geq 0} h_{k_3}.
\]

From Lemma \ref{lemma:energy-est-paraprod}  we have
\begin{align*}
&\lft \Lambda_{\Pi\left( I_0 \right)} \left( f, g, h \right) \rg \lesssim \sum_{k_1, k_2, k_3} \lft \Lambda_{\Pi\left( I_0 \right)} \left( f, g, h \right) \rg \\
&\lesssim \sum_{k_1, k_2, k_3} \big( \ssize_{I_0} f_{k_1}  \big)^{1-\theta_1} \big( \ssize_{I_0} g_{k_2}  \big)^{1-\theta_1}  \big( \ssize_{I_0} h_{k_3}  \big)^{1-\theta_3} \\
&\qquad \cdot \big( \eenergy_{I_0} f_{k_1} \big)^{\theta_1} \big( \eenergy_{I_0} g_{k_2} \big)^{\theta_2} \big( \eenergy_{I_0} h_{k_3} \big)^{\theta_3}.
\end{align*}
For the energy of $f_{k_1}$ we use the estimate in Lemma \ref{lemma-ref-energy-est-paraprod}, bounding the expression above by
\begin{align*}
&  \sum_{k_1, k_2, k_3} \big( \ssize_{I_0} f_{k_1}  \big)^{1-\theta_1} \big( \ssize_{I_0} g_{k_2}  \big)^{1-\theta_2}  \big( \ssize_{I_0} h_{k_3}  \big)^{1-\theta_3} \\
&\qquad \cdot \left( 2^{-k_1M}\| f_{k_1} \|_1\right)^{\theta_1} \left(2^{-k_2M}\| g_{k_2}\|_1 \right)^{\theta_2} \left( 2^{-k_3 M}\| h_{k_3} \|_1\right)^{\theta_3}.
\end{align*} 

For the sizes, we use the trivial estimate $\ds \ssize_{\ii I} f_{k_1} \lesssim \sssize_{\ii I} f$. Now we note that H\"older's inequality implies
\begin{align*}
\sum_{k_1 \geq 0} 2^{-k_1M \theta_1}\| f_{k_1} \|_1^{\theta_1} &\lesssim \left( \sum_{k_1 \geq 0} 2^{-k_1M \theta_1/{2 \left( 1-\theta_1\right)}} \right)^{1-\theta_1} \left(\sum_{k_1 \geq 0} 2^{-k_1M/{2 }}\| f_{k_1}\|_1 \right)^{\theta_1}\\
&\lesssim\left(\sum_{k_1 \geq 0} \| f_{k_1} \cdot \ci_{I_0}\|_1 \right)^{\theta_1} \lesssim \| f \cdot \ci_{I_0}\|_1^{\theta_1},
\end{align*}
since the $f_{k_1}$ all have almost disjoint support.

The proof in the case when some $\theta_j=0$ is identical, with the difference that we sum in two indices $k_\iota$, for $\iota \neq j$. 

\end{proof}
\end{lemma}

As immediate consequences, we obtain
\begin{corollary}
\label{cor:local-paraprod}
For any functions $f, g$ and $h$, we can bound the trilinear form $\Lambda_{\Pi\left( I_0\right)} \left( f, g, h  \right)$ by
\[
\lft \Lambda_{\Pi\left( I_0\right)} \left( f, g, h  \right) \rg \lesssim \sssize_{I_0} f \cdot \sssize_{I_0} g \cdot  \sssize_{I_0} h \cdot  |I_0 |.\]
\begin{proof}
We need only to note that in Lemma \ref{lemma:localized-paraprod-est}, we have 
\[
\|  f \cdot \ci_{I_0} \|_1 \lesssim | I_0 | \, \sssize_{I_0} f.
\]
\end{proof}
\end{corollary}

\begin{corollary}
\label{cor:local-paraprod-L1}
If $\tilde E$ is a fixed set of finite measure, then
\[
\big\| \Pi_{I_0}\left( f, g \right) \cdot \one_{\tilde E}  \big\|_1 \lesssim \sssize_{I_0} f \cdot \sssize_{I_0} g \cdot \sssize_{I_0} \one_{\tilde E} \cdot \lft I_0 \rg.
\]
\end{corollary}

In what follows, we will need a different kind of localization; that is, we fix $F, G, \tilde E$ sets of finite measure and define
\begin{equation}
\label{eq-def-loc-paraprod}
\Pi_{I_0}^{F, G, \tilde E}\left( f, g \right)(x):=\Pi_{I_0}\left( f \cdot \one_F,  g \cdot \one_G \right)(x) \cdot \one_{\tilde E}(x).
\end{equation}
This is part of our approach to proving multiple-vector-valued inequalities. We recover also the following result, which first appeared in \cite{vv_BHT}:

\begin{proposition}
\label{prop:Localization for local L^1 paraproducts}
Let $I_0$ be a fixed dyadic interval and $F, G, \tilde E \subset \rr{R}$ sets of finite measure. Then
{\fontsize{9}{10}
\begin{align*}
\label{eq:local-tril-form-paraprod}
\lft\Lambda_{\Pi(I_0)}^{F, G, \tilde E}(f, g, h)\rg\lesssim \big(\sssize_{I_0}\one_{F}\big)^{\frac{1}{r_1'}-\epsilon} \cdot \big( \sssize_{I_0}\one_{G}\big)^{\frac{1}{r_2'}-\epsilon} \cdot  \big(\sssize_{I_0}\one_{\tilde E}\big)^{\frac{1}{r}-\epsilon} \cdot \| f \cdot \ci_{I_0}\|_{r_1}   \| g \cdot \ci_{I_0} \|_{r_2} \| h \cdot \ci_{I_0} \|_{r'} 
\end{align*}}
whenever $\ds \frac{1}{r_1}+\frac{1}{r_2}+\frac{1}{r'}=1$, and $ 1< r_1, r_2, r' \leq \infty$. Here $\epsilon$ is some small positive number that will be chosen later.
\begin{proof}
This result was proved in \cite{vv_BHT} in detail for the bilinear Hilbert transform operator. For paraproducts, we note that it follows from Lemma \ref{lemma:localized-paraprod-est}, in the case of restricted-type functions: that is, functions that are bounded above by characteristic functions of finite sets. If $\lft f(x) \rg \leq \one_{E_1}(x), \lft g(x) \rg \leq \one_{E_2}(x)$ and $\lft h(x) \rg \leq \one_{E_3}(x)$, then the conclusion is immediate; the general case follows through interpolation.

If there are any $L^\infty$ spaces involved (for example, if $r_2=\infty$), we only need to notice that 
\[
\sssize_{I_0\,} g \cdot \one_G :=\sup_{I \in \ii I^+(I_0)} \frac{1}{|I|} \int_{\rr R} |g| \cdot \one_G \cdot \ci_I^N dx \leq \|  g \cdot \ci_{I_0} \|_\infty \cdot \sssize_{I_0} \one_G.
\]
We fix $g \in L^\infty$ and the trilinear form becomes a bilinear form $\left( f, h \right) \mapsto \Lambda_{\Pi\left( I_0 \right)}^{F,G, \tilde E}(f,g,h)$. The desired inequality is proved again for restricted-type functions, which is sufficient, in view of interpolation theory.  

For characteristic functions of sets, we have the equivalence
\[
\big\|  \one_F \cdot \ci_{I_0} \big\|_{r_1} \sim \| \one_F \cdot \ci'_{I_0} \|_1^{\frac{1}{r_1}},
\]
where $\ci'_{I_0}$ is also an $L^\infty$-adapted bump function associated to the interval $I_0$, as defined in \eqref{def:adapted-fnc}. The function $\ci'_{I_0}$ is of the form $\ci'_{I_0}=\ci_{I_0}^\alpha$, where $\alpha$ depends on $r_1$. For our purposes however, the difference between $\ci_{I_0}$ and $\ci'_{I_0}$ is not important, and we will denote both of them simply by $\ci_{I_0}$.
\end{proof}
\end{proposition}

An adaptation of Proposition 4.6 from \cite{vv_BHT} is the following:
\begin{lemma}
\label{lemma-loc-parad-L1-target}
If $1 <r_1, r_2<\infty$ are so that $\ds \frac{1}{r_1}+\frac{1}{r_2}=1$, then
\begin{equation}
\label{eq:loc-parad-Lone-target}
\big\| \Pi_{I_0}^{F, G, \tilde E} (f, g)  \big\|_1 \lesssim \big( \sssize_{I_0}\one_{F}\big)^{\frac{1}{r_1'}-\epsilon} \cdot \big(\sssize_{I_0}\one_{G}\big)^{\frac{1}{r_2'}-\epsilon} \cdot  \big( \sssize_{I_0}\one_{\tilde E}\big)^{1-\epsilon} \cdot \| f \cdot \ci_{I_0}\|_{r_1}   \| g \cdot \ci_{I_0} \|_{r_2}. 
\end{equation}
\begin{proof}

Let $h$ be a function so that $\| h\|_{\infty}=1$, and define the bilinear form 
\[
\tilde \Lambda_h(f, g):=\Lambda_{I_0}(f\cdot \one_F, g \cdot \one_G, h \cdot \one_{\tilde E}).
\]

Then we have from Lemma \ref{lemma:localized-paraprod-est}, with $\theta_3=0$, that for any sets of finite measure $E_1$ and $E_2$, and any functions $f, g$ so that $\lft f \rg\leq \one_{E_1}, \lft g \rg\leq \one_{E_2}$, 
\[
\big| \tilde \Lambda_h(f, g) \big | \lesssim \big( \ssize_{I_0} \one_F  \big)^{1/{\tilde r_1'}} \big( \ssize_{I_0} \one_G  \big)^{1/{\tilde r_2'}} \big( \ssize_{I_0} \one_{\tilde E}  \big) \|\one_{E_1} \cdot \ci_{I_0} \|_{\tilde r_1} \|\one_{E_2} \cdot \ci_{I_0} \|_{\tilde r_2},
\]
for any tuple $(\tilde r_1, \tilde r_2)$ in a neighborhood of $(r_1, r_2)$, with the property that $\frac{1}{\tilde r_1}+\frac{1}{\tilde r_2}=1$.

Interpolation theory then implies the inequality
\[
\big| \tilde \Lambda_h(f, g) \big | \lesssim \big( \ssize_{I_0} \one_F  \big)^{1/{ r_1'}} \big( \ssize_{I_0} \one_G  \big)^{1/{ r_2'}} \big( \ssize_{I_0} \one_{\tilde E}  \big) \| f \cdot \ci_{I_0} \|_{\tilde r_1} \| g \cdot \ci_{I_0} \|_{\tilde r_2},
\]
for any functions $f$ and $g$. We note that the implicit constants do not depend on $h$.

The estimate \eqref{eq:loc-parad-Lone-target} follows, since $$\big\| \Pi_{I_0}^{F, G, \tilde E} (f, g)  \big\|_1 = \sup_{ \| h \|_\infty =1 } \big| \tilde \Lambda_{ h}(f,g)\big|.$$
\end{proof}
\end{lemma}

In the case of quasi-Banach spaces, when $ \frac{1}{r_1}+\frac{1}{r_2}=\frac{1}{r} >1$, a result resembling Corollary \ref{cor:local-paraprod-L1} holds. This cannot be obtained directly from the estimate for $\Lambda_{\Pi\left( I_0 \right)}(f, g, h \cdot \one_{\tilde E})$, but requires an extra decomposition and handling of the sizes. A similar argument will be used repeatedly throughout the paper, but the details of the decomposition will not be reproduced.

\begin{lemma}
\label{lemma:sub-unit-local-est-paraprod}
If $\tau<1$, then for any $\epsilon>0$ small enough, we have
\begin{equation}
\label{eq:sub-unit-local-est-paraprod}
\big\| \Pi_{I_0} \left( f, g \right) \cdot \one_{\tilde{E}}\big\|_\tau^\tau \lesssim \big( \sssize_{I_0} f  \big)^\tau \cdot  \big( \sssize_{I_0} g \big)^\tau \cdot  \big( \sssize_{I_0} \one_{\tilde E}  \big)^{1-\epsilon} \cdot \lft I_0 \rg.
\end{equation}
\begin{proof}
Let $\tau_0>0$ be so that $\frac{1}{\tau}=1+\frac{1}{\tau_0}$.

As in Lemma \ref{lemma:localized-paraprod-est}, we split 
$$\one_{\tilde E}(x):=\sum_{k_3 \geq 0}\one_{\tilde E _{k_3}}, \quad \text{where  } \one_{\tilde E_{k_3}}=\one_{\tilde E} \cdot \one_{\left \lbrace x:  \dist \left(x, I_0 \right) \sim \left(2^{k_3} -1\right) | I_0 |\right\rbrace}.$$

We use the subadditivity of $\| \cdot \|_\tau^\tau$, and H\"older's inequality to get
\begin{align*}
\big\| \Pi_{I_0}\left( f, g \right) \cdot \one_{\tilde E}\big\|_\tau^\tau &\lesssim \sum_{k_3 \geq 0} \big\| \Pi_{I_0} \left( f, g \right) \cdot \one_{\tilde E_{k_3}} \big\|_\tau^\tau \\
&\lesssim \sum_{k_3 \geq 0} \big\| \Pi_{I_0} \left( f, g \right) \cdot \one_{\tilde E_{k_3}} \big\|_1^\tau \| \one_{\tilde E_{k_3}} \|_{\tau_0}^\tau.
\end{align*}

From Corollary \ref{cor:local-paraprod-L1}, we have 
\[
\big\| \Pi_{I_0} \left( f, g \right) \cdot \one_{\tilde E_{k_3}} \big\|_1 \lesssim \sssize_{I_0} f \cdot \sssize_{I_0} g \cdot \sssize_{I_0} \one_{\tilde E_{k_3}} \cdot \lft I_0 \rg.
\]
Using this and the observation that 
\[
\sssize_{I_0} \one_{\tilde E_{k_3}} \lesssim 2^{-k_3 M}, \qquad \sssize_{I_0} \one_{\tilde E_{k_3}} \lesssim \sssize_{I_0} \one_{\tilde E},
\]
we can estimate the desired expression by

\begin{align*}
&\big\| \Pi_{I_0}\left( f, g \right) \cdot \one_{\tilde E}\big\|_\tau^\tau \lesssim \sum_{k_3 \geq 0} \big(\sssize_{I_0} f\big)^\tau \cdot \big( \sssize_{I_0} g \big)^\tau \cdot \big( \sssize_{I_0} \one_{\tilde E}\big)^{\tau-\epsilon} \cdot \lft I_0 \rg^\tau 2^{-k_3 M \epsilon \tau} \big\| \one_{\tilde E_{k_3}} \big\|_{\tau_0}^\tau.
\end{align*}

Similarly to Lemma \ref{lemma:localized-paraprod-est}, we use H\"older's inequality in order to sum in $k_3$ and bring into play the fast decay when $\tilde E$ is away from $I_0$:
\begin{align}
\label{eq:handling_sizes}
\sum_{k_3 \geq 0} 2^{-k_3 M \epsilon \tau} \| \one_{\tilde E_{k_3}} \|_1^{\frac{\tau}{\tau_0}} \lesssim \sum_{k_3 \geq 0} 2^{-\frac{k_3 M \epsilon \tau}{2}}\big\| \one_{\tilde E_{k_3}} \cdot \ci_{I_0} \big\|_1^{\frac{\tau}{\tau_0}} \lesssim \big\|  \one_{\tilde E} \cdot \ci_{I_0} \big\|_1^{\frac{\tau}{\tau_0}} \lesssim \big( \sssize_{I_0} \one_{\tilde E} \big)^{\frac{\tau}{\tau_0}} \cdot \lft I_0 \rg^{\frac{\tau}{\tau_0}}.
\end{align}

Since $\ds 1= \tau+ \frac{\tau}{\tau_0}$, we have in fact obtained 
\[
\big\| \Pi_{I_0}\left( f, g \right) \cdot \one_{\tilde E}\big\|_\tau^\tau \lesssim \big(\sssize_{I_0} f\big)^\tau \cdot \big( \sssize_{I_0} g \big)^\tau \cdot \big( \sssize_{I_0} \one_{\tilde E}\big)^{1-\epsilon} \cdot | I_0 |.
\]
\end{proof}
\end{lemma}

\vspace{.3 cm}
\subsection{The method of the proof}~\\
\label{subsec-method-proof}
In \cite{vv_BHT}, we proved the vector-valued inequalities $T : L^p( \rr R; L^{R^1}( \ii W, \mu  )) \times L^q ( \rr R; L^{R^2} ( \ii W, \mu )) \to L^s ( \rr R; L^{R}( \ii W, \mu ) )$, whenever $1< r_1^j, r_2^j \leq \infty$ and $1\leq r^j <\infty$, where $T$ is either the bilinear Hilbert transform $BHT$ or a paraproduct $\Pi$. In the present paper, we are concerned with the case when at least one of the $r^j$ is $<1$.

Whenever $1 \leq r< \infty$, the $\| \cdot \|_{L^r\left( \ii W, \mu \right)}$ norm can be dualized and the problem reduces to estimating the trilinear form $\Lambda_T(\vec f, \vec g, \vec h)$. We recall that in the discrete case, this corresponds to $\sum_k \Lambda_T(f_k, g_k, h_k)$. 

There are two coupled statements for the localized trilinear forms $\Lambda_{T;I_0}^{F, G, H'}(\vec f,\vec g, \vec h):=\Lambda_{T; I_0}(\vec f \cdot \one_F, \vec g \cdot \one_G, \vec h \cdot \one_{H'})$, that are at the very core of our method from \cite{vv_BHT}:
\begin{align}
\label{eq-helM-n}
\tag*{$\ii P (n)$}
 \lft \Lambda_{T;I_0}^{F, G, H'}(\vec f,\vec g, \vec h) \rg &\lesssim \big( \sssize_{I_0} \one_F \big)^{\frac{1}{r_1'}-\epsilon} \big( \sssize_{I_0} \one_G \big)^{\frac{1}{r_2'}-\epsilon} \big( \sssize_{I_0} \one_{H'} \big)^{\frac{1}{r}-\epsilon}\\
& \qquad \cdot  \big\| \big\|  \vec f\big\|_{L^{R^n_1}} \cdot \ci_{I_0} \big\|_{r_1} \big\| \big\| \vec g  \big\|_{L^{R^n_2}}\cdot \ci_{I_0} \big\|_{r_2} \big\| \big\| \vec h \big\|_{L^{\left(R^n\right)'}}\cdot \ci_{I_0} \big\|_{r'} \nonumber
\end{align}
and, for functions $\vec f, \vec g, \vec h$ satisfying $\big\|\vec f(x) \big\|_{L^{R^n_1}}\leq \one_F(x), \big\| \vec g(x) \big\|_{L^{R^n_2}}\leq \one_G(x)$ and $\big\| \vec h(x) \big\|_{L^{\left(R^n\right)'}}\leq \one_{H'}(x)$ respectively, 
\begin{align}
\label{eq-helM-n-sizes-only}
\tag*{$\ii P^* (n)$}
 \lft \Lambda_{T;I_0}^{F, G, H'}(\vec f, \vec g, \vec h) \rg \lesssim \big( \sssize_{I_0} \one_F \big)^{1-\epsilon} \big( \sssize_{I_0} \one_G \big)^{1-\epsilon} \big( \sssize_{I_0} \one_{H'} \big)^{1-\epsilon} \cdot \lft I_0 \rg.
\end{align}

For paraproducts, the exponents in $\ii P^*(n)$ are $1-\epsilon$, while for $BHT$ they are of the form $\frac{1+\theta_j}{2}-\epsilon$, where $0 \leq \theta_1,\theta_2, \theta_3<1$ and $ \theta_1+\theta_2+\theta_3=1$.

If $r<1$, an argument employing the localized trilinear form is not available. Instead, we will use only estimates for the localized operators $\Pi_{I_0}^{F, G, \tilde E}$ and $BHT_{I_0}^{F, G, \tilde E}$. Here we focus on the paraproduct case, for clarity. The localized induction statements are
{\fontsize{10}{10}
\begin{align}
\label{eq-helM-n-oper}
\tag*{$\ii P(n):$}
\big\|   \big\|  \Pi_{I_0}^{F, G, \tilde E}(\vec f, \vec g)   \big\|_{L^{R^n}}   \big\|_{s} \lesssim \big( \sssize_{I_0} \one_F \big)^{\frac{1}{p'}-\epsilon} \big( \sssize_{I_0} \one_G \big)^{\frac{1}{q'}-\epsilon} \big( \sssize_{I_0} \one_{\tilde E} \big)^{\frac{1}{s}-\epsilon} 
\big\| \| \vec f  \|_{L^{R^n_1}} \cdot \ci_{I_0} \big\|_p \big\| \big\| \vec g  \big\|_{L^{R^n_2}} \cdot \ci_{I_0} \big\|_q,
\end{align}}
and, for functions $\vec f, \vec g$ satisfying $\big\| \vec f(x) \big\|_{L^{R^n_1}}\leq \one_F(x)$ and $\big\| \vec g(x) \big\|_{L^{R^n_2}}\leq \one_G(x)$ respectively, 
\begin{equation}
\label{eq-helM-sizes-oper} \tag*{$\ii{P}^*\left( n\right):$}
\big\|   \big\|  \Pi_{I_0}^{F, G, \tilde E}(\vec f, \vec g)   \big\|_{L^{R^n}}   \big\|_{s} \lesssim \big( \sssize_{I_0} \one_F \big)^{1-\epsilon} \big( \sssize_{I_0} \one_G \big)^{1-\epsilon} \big( \sssize_{I_0} \one_{\tilde E} \big)^{\frac{1}{s}-\epsilon} \lft I_0 \rg^{\frac{1}{s}}.
\end{equation}

The proof of the induction step $\ii P(n-1) \Rightarrow \ii P(n)$ is presented in Theorem \ref{thm:genral_case}. The statement $\ii P(0)$ represents the content of Proposition \ref{prop:localization-lemma}, and relies on $\ii P^*(0)$. On the other hand, $\ii P ^*(0)$ follows from Lemma \ref{lemma:sub-unit-local-est-paraprod}, where the local estimate for the trilinear form is used. 

In what follows, we will show how to use $\ii P(0)$ in order to obtain the $\ell^r$-valued estimates, for $r<1$. We want to estimate $\big\| \big( \sum_k | \Pi(f_k, g_k) |^r  \big)^{1/r} \big\|_{\tilde s, \infty}$ under the assumption that $\|  \vec f(x) \|_{\ell^{r_1}} \leq \one_F(x)$ and $\| \vec g(x) \|_{\ell^{r_2}} \leq \one_G(x)$. In order to deal with the $\ell^r$ quasinorm inside, we dualize through $L^r$; given $E$ a set of finite measure, we can construct a major subset $\tilde E \subseteq E$ so that 
\begin{equation}
\label{eq:vector-lr}
\big\| \big( \sum_k \lft \Pi(f_k, g_k) \rg^r  \big)^{1/r} \big\|_{\tilde s, \infty} \sim \big\| \big( \sum_k \lft \Pi(f_k, g_k) \rg^r  \big)^{1/r} \cdot \one_{\tilde E} \big\|_{r} \cdot \lft E \rg^{\frac{1}{\tilde s}-\frac{1}{r}}.
\end{equation}
The advantage is that $\big\| \big( \sum_k | \Pi(f_k, g_k) |^r  \big)^{1/r} \cdot \one_{\tilde E} \big\|_r^r=\sum_k \big\| \Pi(f_k, g_k)  \cdot \one_{\tilde E} \big\|_r^r$, and even more, $\| \cdot \|_r^r$ is subadditive. We use dualization through $L^r$ in order to ``linearize" the expression $ \big( \sum_k | \Pi(f_k, g_k) |^r  \big)^{1/r}$ when classical Banach space techniques are not available.

Afterwards we employ the helicoidal method as in \cite{vv_BHT}. Through a triple stopping time that will be described shortly, and using the subadditivity of $\| \cdot \|_r^r$, \eqref{eq:vector-lr} is reduced to obtaining ``sharp estimates" for $\big\| \Pi_{I_0}(\tilde f, \tilde g)  \cdot \one_{\tilde E} \big\|_r^r$ for scalar functions $\tilde f$ and $\tilde g$. If we use the trilinear form associated to $\Pi_{I_0}$, we get 
\[
\big\| \Pi_{I_0}(\tilde f, \tilde g)  \cdot \one_{\tilde E} \big\|_r^r\lesssim \big(\sssize_{I_0} \one_{E_1}\big)^{r-\epsilon} \big(\sssize_{I_0} \one_{E_2}\big)^{r-\epsilon} \big(\sssize_{I_0} \one_{\tilde E}\big)^{r-\epsilon} \cdot | I_0 |,
\]
where we assume $\lft \tilde f\rg \leq \one_{E_1}, \lft \tilde g\rg \leq \one_{E_2}$. However, we can obtain a better estimate, which is precisely Lemma \ref{lemma:sub-unit-local-est-paraprod}:
\[
\big\| \Pi_{I_0}(\tilde f, \tilde g)  \cdot \one_{\tilde E} \big\|_r^r\lesssim \big(\sssize_{I_0} \one_{E_1}\big)^{r-\epsilon} \big(\sssize_{I_0} \one_{E_2}\big)^{r-\epsilon} \big(\sssize_{I_0} \one_{\tilde E}\big)^{1-\epsilon} \cdot | I_0 |.
\]
This improvement (since the sizes are subunitary and $r<1$, $(\sssize_{I_0} \one_{\tilde E})^{1-\epsilon} \leq \big(\sssize_{I_0} \one_{\tilde E}\big)^{r-\epsilon}$) allows us to prove vector-valued inequalities for paraproducts within the whole $Range(\Pi)$.

Similarly, the statements $\ii P(n)$ and $\ii P^* (n)$ for $R^n=(r^1, \tilde R^{n-1})$ follow from $\ii P (n-1)$ by using $L^{r^1}$-dualization.
\vspace{.3cm}
\subsection{The Triple Stopping Time}~\\
\label{subsection-stopping_times}
All through the paper, we will need estimates along the line
\begin{equation}
\label{eq:imp-est}
\big\|  \Pi_{\ii I \left( I_0 \right)}^{F, G, \tilde E}(\vec f, \vec  g) \cdot \one_{\tilde E_3} \big\|_\tau \lesssim \big\| \Pi_{\ii I \left( I_0 \right)}^{F, G, \tilde E} \big\| \|  \one_{E_1} \cdot \ci_{I_0} \|_{s_1} \|  \one_{E_2} \cdot \ci_{I_0} \|_{s_2} |E_3|^{\frac{1}{\tau} -\frac{1}{s}},
\end{equation} 
where $\tilde E_3 \subseteq E_3$ is a major subset of $E_3$, and the functions $\vec f$ and $\vec g$ satisfy $\| \vec f(x) \|_{R^n_1}\leq \one_{E_1}(x)$ and $\| \vec g(x) \|_{R^n_2}\leq \one_{E_2}(x)$. Here $\big\| \Pi_{\ii I \left( I_0 \right)}^{F, G, \tilde E} \big\|$ represents the operatorial norm, as introduced in \eqref{eq:intr-op-local-sharp}.

We note that sometimes we might have $F=E_1$, $G=E_2$ and $\tilde E=\tilde E_3$; also, the estimate doesn't need to be local, and in this case we regard $I_0$ as being the whole real line, and $\ci_{I_0}(x) \equiv 1$ a. e.

The exceptional set represents the set where the values of $\| \vec f(x) \|_{R^n_1}$ and $\| \vec g(x) \|_{R^n_2}$ are too large:
\[
\tilde \Omega:=\Big\lbrace x:  \ic{M}\left( \one_{E_1} \cdot \ci_{I_0}  \right)(x) > C \frac{\| \one_{E_1} \cdot \ci_{I_0} \|_1}{\lft  E_3\rg}  \Big\rbrace \cup \Big\lbrace x:  \ic{M}\left( \one_{E_2} \cdot \ci_{I_0}  \right)(x) > C \frac{\| \one_{E_2} \cdot \ci_{I_0} \|_1}{\lft  E_3\rg}  \Big\rbrace,
\]
and $\tilde E_3=E_3 \setminus \tilde \Omega$. Now we partition the collection $\ii I$ of intervals into subcollection $\ii I_d$ so that for any $I \in \ii I_d$, we have $\dist (I, \tilde \Omega^c) \sim \left( 2^d-1\right) |I|$. This will allow us to gain some information on the sizes of $\one_{E_1}$ and $\one_{E_2}$:
\[
\sup_{I \in \ii I_d } \frac{1}{|I|} \int_{\rr R} \one_{E_1} \cdot \ci_I dx \lesssim 2^d  \frac{\| \one_{E_1} \cdot \ci_{I_0} \|_1}{|E_3|}, \qquad \sup_{I \in \ii I_d } \frac{1}{|I|} \int_{\rr R} \one_{E_2} \cdot \ci_I dx \lesssim 2^d  \frac{\| \one_{E_2} \cdot \ci_{I_0} \|_1}{|E_3|}.
\] 
Also, note that $\ds \sup_{I \in \ii I_d } \frac{1}{|I|} \int_{\rr R} \one_{\tilde E_3} \cdot \ci_I^M dx \lesssim 2^{-Md}$.

We want to apply our local estimates, but they cannot be directly implemented to $\Pi_{\ii I\left( I_0 \right)}$ nor to $\Pi_{\ii I_d\left( I_0 \right)}$. Instead, we will partition again the collection $\ii I_d(I_0)$ as 
\begin{equation}
\label{eq:partition-stopping-time}
\ii I (I_0):=\bigcup_{d \geq 0} \ii I_d(I_0):=\bigcup_{d} \bigcup_{n_1, n_2, n_3} \bigcup_{K \in \ii I^{n_1, n_2, n_3}} \ii I_d(K). 
\end{equation}
We will construct collections $\ii I_{n_1}, \ii I_{n_2}, \ii I_{n_3}$ of dyadic intervals, and for every $I_j$ contained in some $\ii I_{n_j} $, we will select $\ii I_d(I_j) \subseteq \ii I_d(I_0)$ a subcollection of our initial $\ii I_d(I_0)$. Then we say $K \in \ii I^{n_1, n_2, n_3}$ if $K=I_1\cap I_2 \cap I_3$ for $I_j \in \ii I_{n_j}$, and 
\[
\ii I_d(K):=\ii I_d(I_1) \cap \ii I_d(I_2) \cap \ii I_d(I_3).
\]

In effect, we carry out the local estimates on $\Pi_{\ii I_d(K)}$, where $K \in \ii I^{n_1, n_2, n_3}$. For any two such intervals $K\in \ii I^{n_1, n_2, n_3}$ and $K' \in \ii I^{n_1', n_2', n_3'}$, we could have $K \cap K' \neq \emptyset$; but the collections $\ii I_d(K)$ and $\ii I_{d'}(K)$ are going to be disjoint.

The families of dyadic intervals $\ii I_{n_1}$ will have the following properties:
\begin{enumerate}[label=(\arabic*), ref=\arabic*, leftmargin=1cm]
\item \label{item-1} the intervals $I' \in \ii I_{n_1}$ are all mutually disjoint
\item \label{item-2} moreover, they satisfy $\ds \sum_{I' \in  \ii I_{n_1}}|I'| \lesssim 2^{n_1} \| \one_{E_1} \cdot \ci_{I_0} \|_1$
\item \label{item-3}whenever $I_1 \in \ii I_{n_1}$ and  $\ic I \subseteq \ii I_d(I_1)$ is a subset of the selected $\ii I_d(I_1)$, we have 
\[
\sssize_{ \ic I (I_1)^+} \one_{E_1} \lesssim 2^{-n_1} \lesssim 2^d  \frac{\| \one_{E_1} \cdot \ci_{I_0} \|_1}{|E_3|}.
\]
\item \label{item-4} as a consequence of \eqref{item-3}, $\ds \sssize_{ \ic I_{n_1} (I_0)^+} \one_{E_1} \lesssim 2^{-n_1}$, where $\ds \ic I_{n_1} (I_0):=\bigcup_{I_1 \in \ii I_{n_1}} \ii I(I_1)$.
\end{enumerate}
Since the construction argument is similar for the collections $\ii I_{n_2}$ and $\ii I_{n_3}$, we will only describe it for $\ii I_{n_1}$. We start by setting $\ii I_{Stock}:=\ii I_d$, the collection of intervals in $\ii I$ having the property that $\dist(I, \tilde \Omega^c) \sim \left(2^d -1\right) |I|$. 

Assuming the collections $\ii I_{n_1}$ were selected for all $n_1 <\bar n_1$, we construct $\ii I_{\bar n_1}$ in the following way:
\begin{enumerate}[label=(\roman*), ref=\roman*, leftmargin=.7 cm]
\item \label{item-i}if $\ds \sssize_{\ii I_{Stock}} \one_{E_1}< 2^{-\bar n_1}$, then nothing happens; restart the procedure with $\bar n_1:=\bar n_1 +1$.
\item \label{item-ii} admitting that we are in the situation where $\ds \sssize_{\ii I_{Stock}} \one_{E_1}=2^{-\bar n_1}$,  look for $I \in \ii I_{Stock}$ so that 
\[
\frac{1}{|I|} \int_{\rr R} \one_{E_1} \cdot \ci_I dx \sim 2^{- \bar n_1}.
\]
\item \label{item-iii} then $\ii I_{\bar n_1}$ will consist of maximal dyadic intervals $I_1 \subseteq I_0$ with the property that they contain at least one interval $I \in \ii I_{Stock}$ as in \eqref{item-ii}, and so that 
\[
2^{-\bar n_1} \leq \frac{1}{|I_1|}\int_{\rr R} \one_{E_1} \cdot \ci_{I_1} dx \leq 2^{-\bar n_1}.
\]
\item \label{item-iv} for every $I_1 \in \ii I_{\bar n_1}$, the collection $\ii I_d(I_1)$ is defined as 
\[
\ii I_d(I_1):=\lbrace I \in \ii I_{Stock}: I \subseteq I_1  \rbrace.
\]
\item before we restart the procedure from step \eqref{item-i} by increasing $\bar n_1$, we update $\ds \ii I_{Stock}:=\ii I_{Stock} \setminus \bigcup_{I_1 \in \ii I_{\bar n_1}}  \ii I(I_1)$.
\end{enumerate}

It's not difficult to check that conditions \eqref{item-1}-\eqref{item-4} are verified. 

The last step consist in putting everything together, in order to deduce \eqref{eq:imp-est}. Here we use the subadditivity of $\| \cdot \|_\tau^\tau$ as follows:
\begin{align*}
\big\|  \Pi_{\ii I \left( I_0 \right)}^{F, G, \tilde E}(\vec f, \vec  g) \cdot \one_{\tilde E_3} \big\|_\tau^\tau \lesssim \sum_{d \geq 0} \sum_{n_1, n_2, n_3} \sum_{K \in \ii I^{n_1, n_2, n_3}} \big\|  \Pi_{\ii I_d \left( K \right)}^{F, G, \tilde E}(\vec f, \vec  g) \cdot \one_{\tilde E_3}  \big\|_\tau^\tau.
\end{align*}
The remaining part follows from the local estimates, as it will be detailed later on. In order to simplify the notation, sometimes we forget about the $d$ parameter.

\section{Quasi-Banach Valued Inequalities}
\label{sec:quasi-banach-Pi}

In the present section, we develop the ideas from Sections \ref{subsec-method-proof} and \ref{subsection-stopping_times}. In fact, we prove that $\ds \Pi : L^p\left( \ell^{r_1} \right) \times L^q\left( \ell^{r_2} \right) \to L^s\left( \ell^{r} \right)$,  when $\frac{1}{2}<r<1$.

\vspace{.3cm}

\subsection{A Localization Lemma for quasi-Banach spaces}~\\
Fix $I_0$ a dyadic interval, and the sets $F, G$ and $\tilde{E}$, as in \eqref{eq-def-loc-paraprod}. Then we have the following result:

\begin{proposition}
\label{prop:localization-lemma}
For any $1< r_1, r_2 \leq \infty$ and $\frac{1}{2}<r <\infty$ such that $\ds \frac{1}{r_1}+\frac{1}{r_2}=\frac{1}{r}$, the localized paraproduct $\Pi_{I_0}^{F, G, \tilde{E}}$ satisfies the estimate
\begin{equation}
\label{eq:localized-paraproduct}
\big\| \Pi_{I_0}^{F, G, \tilde{E}}(f, g) \big\|_r \lesssim \big(\sssize_{I_0} \one_F \big)^{1/{r_1'}-\epsilon} \cdot \big(\sssize_{I_0} \one_G \big)^{1/{r_2'}-\epsilon} \cdot \big(\sssize_{I_0} \one_{\tilde{E}} \big)^{1/{r}-\epsilon} \cdot \big\| f \cdot \ci_{I_0} \big\|_{r_1} \cdot \big\| g \cdot \ci_{I_0} \big\|_{r_2}.
\end{equation}
\begin{proof}
The case $r \geq 1$ was considered in \cite{vv_BHT}, so here we focus on the situation when $r<1$. We will prove the inequality $\eqref{eq:localized-paraproduct}$ by using restricted weak-type interpolation. Hence we start by fixing $E_1$ and $E_2$, sets of finite measure, and let $f$ and $g$ be functions satisfying $\ds \lft f \rg \leq \one_{E_1}$ and  $\ds \lft g \rg \leq \one_{E_2}$. Following standard interpolation theory, it will be enough to prove
\begin{align}
\label{eq:weak-norms-local-paraprod}
\big\| \Pi_{I_0}^{F, G, \tilde{E}}(f, g) \big\|_{s, \infty} &\lesssim \big(\sssize_{I_0} \one_F \big)^{1/{s_1'}-\epsilon} \cdot \big(\sssize_{I_0} \one_G \big)^{1/{s_2'}-\epsilon} \cdot \big(\sssize_{I_0} \one_{\tilde{E}} \big)^{1/{s}-\epsilon} \\
&\qquad \cdot \|\one_{E_1} \cdot \ci_{I_0}\|_{1}^{\frac{1}{s_1}} \cdot \|\one_{E_2} \cdot \ci_{I_0}\|_{1}^{\frac{1}{s_2}}, \nonumber
\end{align}
for all admissible tuples $( s_1, s_2, s)$ in a neighborhood of $( r_1, r_2, r)$ that satisfy the usual H\"{o}lder scaling condition. 

To this end, we use $L^\tau$-dualization, as described in Proposition \ref{prop:Lr-dualization}. The estimate \eqref{eq:localized-paraproduct} will be independent on the choice of $\tau$, if we pick $\tau\leq r$. If we dualize through an $L^\tau $ space with $\tau>r$, we obtain on the RHS of \eqref{eq:localized-paraproduct} the term
\[
\big( \sssize_{I_0} \one_{\tilde E}  \big)^{\frac{1}{\tau}-\epsilon}.
\]
Since the sizes are subunitary, such an estimate is less sharp than \eqref{eq:localized-paraproduct}. This is in part why, for $r<1$, we don't obtain an ``optimal bound" by dualizing through $L^1$, as in Proposition \ref{prop:Localization for local L^1 paraproducts}.

Thus, given a set $E_3$ of finite measure, we define an exceptional set $\tilde{\Omega}$ by
\begin{equation}
\label{eq:def-ex-set-local}
\tilde{\Omega}:=\Big\lbrace x: \ic{M}\left( \one_{E_1}\cdot \ci_{I_0} \right)(x)>C \frac{\| \one_{E_1}\cdot \ci_{I_0}\|_1}{\lft E_3 \rg}  \Big\rbrace \cup \Big\lbrace x: \ic{M}\left( \one_{E_2}\cdot \ci_{I_0} \right)(x)>C \frac{\| \one_{E_2}\cdot \ci_{I_0}\|_1}{\lft E_3 \rg}  \Big\rbrace, 
\end{equation}
and set $\tilde E_3:=E_3\setminus \tilde{\Omega}$. Then $\tilde E_3$ is a major subset of $E_3$, and we are left with proving
\begin{equation*}
\big\| \Pi_{I_0}(f \cdot \one_{F}, g \cdot \one_{G})\cdot \one_{\tilde{E}}\cdot \one_{\tilde E_3} \big\|_\tau \lesssim \big\| \Pi_{I_0}^{F, G, \tilde{E}} \big\| \cdot \|\one_{E_1} \cdot \ci_{I_0}\|_{1}^{\frac{1}{s_1}} \cdot \|\one_{E_2} \cdot \ci_{I_0}\|_{1}^{\frac{1}{s_2}} \cdot \lft E_3 \rg^{\frac{1}{\tau}-\frac{1}{s}},
\end{equation*}
where the operatorial norm $\ds \big\| \Pi_{I_0}^{F, G, \tilde{E}} \big\|$ is given by
\[
 \big\| \Pi_{I_0}^{F, G, \tilde{E}} \big\|:= \big(\sssize_{I_0} \one_F \big)^{1/{r_1'}-\epsilon} \cdot \big(\sssize_{I_0} \one_G \big)^{1/{r_2'}-\epsilon} \cdot \big(\sssize_{I_0} \one_{\tilde{E}} \big)^{1/{r}-\epsilon}.
\]

As in \cite{vv_BHT}, we will have a triple stopping time which is dictated by the ``concentration" of the sets $E_1, E_2$ and $\tilde E_3$. This is explained in detail in Section \ref{subsection-stopping_times}. More exactly, we have three collections of intervals:
\[
I \in \ii{I}^{n_1}: \qquad 2^{-n_1-1} \leq \frac{1}{\lft I \rg} \int_{\rr{R}} \one_{E_1} \cdot \ci_I dx \leq 2^{-n_1} \lesssim 2^{d} \frac{\| \one_{E_1} \cdot \ci_{I_0} \|_1}{\lft E_3 \rg}, 
\]
\[
I \in \ii{I}^{n_2}: \qquad 2^{-n_2-1} \leq \frac{1}{\lft I \rg} \int_{\rr{R}} \one_{E_2} \cdot \ci_I dx \leq 2^{-n_2} \lesssim 2^{d} \frac{\| \one_{E_2} \cdot \ci_{I_0} \|_1}{\lft E_3\rg}, 
\]
\[
I \in \ii{I}^{n_3}: \qquad 2^{-n_3-1} \leq \frac{1}{\lft I \rg} \int_{\rr{R}} \one_{E_3'} \cdot \ci_I dx \leq 2^{-n_3} \lesssim 2^{-Md}.
\]

Since $\tau<1$, $\| \cdot \|_\tau$ doesn't satisfy the triangle inequality, but $\| \cdot \|_\tau^\tau$ does, and we proceed to estimate
\[
\big\| \Pi_{I_0}\left( f \cdot \one_F, g \cdot \one_G \right) \cdot \one_{\tilde E_3} \cdot \one_{\tilde{E}} \big\|_\tau^\tau \lesssim \sum_{n_1, n_2, n_3} \sum_{I \in \ii{I}^{n_1, n_2, n_3}} \big\|  \Pi_{I}\left( f \cdot \one_F, g \cdot \one_G \right) \cdot \one_{\tilde E_3} \cdot \one_{\tilde{E}} \big\|_\tau^\tau.
\]

We recall that for each $I \in \ii{I}^{n_1, n_2, n_3}$, $\Pi_I$ should be understood as $\Pi_{\ii I_d\left( I \right)}$, as explained in Section \ref{subsection-stopping_times}.

For each interval $I$, we can estimate the localized paraproduct $\ds \Pi_I^{F, G, \tilde E \cap \tilde E_3}$ using sizes only. Lemma \ref{lemma:sub-unit-local-est-paraprod} yields
\begin{align}
\label{eq:trick_1}
&\big\| \Pi_{I}\big( f \cdot \one_F, g \cdot \one_G \big) \cdot \one_{\tilde E_3} \cdot \one_{\tilde{E}}  \big\|_\tau \\
&\lesssim \big( \sssize_{I} f \cdot \one_{F} \big) \cdot \big( \sssize_{I} g \cdot \one_{G} \big) \cdot \big( \sssize_{I} \one_{\tilde E} \cdot \one_{\tilde E_3} \big)^{\frac{1}{\tau} -\epsilon} | I |^{\frac{1}{\tau}} \nonumber \\
&\lesssim \big( \sssize_{I} \one_{F} \big)^{\frac{1}{s_1'}-\epsilon} \cdot \big( \sssize_{I} \one_{G} \big)^{\frac{1}{s_2'}-\epsilon} \cdot \big( \ssize_{I} \one_{\tilde{E}} \big)^{\frac{1}{s}-\epsilon} \nonumber\\
& \quad \cdot \big( \sssize_{I} \one_{E_1} \big)^{\frac{1}{s_1}} \cdot \big( \sssize_{I} \one_{E_2}\big)^{\frac{1}{s_2}} \cdot \big( \ssize_{I} \one_{\tilde E_3}\big)^{\frac{1}{\tau}-\frac{1}{s}+\epsilon} \cdot \lft I \rg^{\frac{1}{\tau}}. \nonumber
\end{align}
Hence, if we denote $\ds K:=\big( \sssize_{I_0} \one_{F} \big)^{\frac{1}{s_1'}-\epsilon} \cdot \big( \sssize_{I_0} \one_{G} \big)^{\frac{1}{s_2'}-\epsilon} \cdot \big( \ssize_{I_0} \one_{\tilde{E}} \big)^{\frac{1}{s}-\epsilon}$, we actually obtain 

\begin{align*}
&\big\| \Pi_{I_0}\left( f \cdot \one_F, g \cdot \one_G \right) \cdot \one_{\tilde E_3} \cdot \one_{\tilde{E}} \big\|_\tau^\tau \lesssim \sum_{n_1, n_2, n_3} \sum_{I \in \ii{I}^{n_1, n_2, n_3}} K^\tau 2^{-\frac{n_1 \tau}{s_1} } 2^{-\frac{n_2 \tau}{s_2} } 2^{- n_3 \left(  1-\frac{\tau}{s}+\epsilon \right)} \cdot \lft I \rg.
\end{align*}

Given the selection process for the collections of intervals $\ii{I}^{n_1, n_2, n_3}$, we have that 
\[
\sum_{I \in \ii{I}^{n_1, n_2, n_3}} \lft I \rg\lesssim \left( 2^{n_1} \big\| \one_{E_1} \cdot \ci_{I_0} \big\|_1 \right)^{\gamma_1}\cdot \left( 2^{n_2} \big\| \one_{E_2} \cdot \ci_{I_0} \big\|_1 \right)^{\gamma_2}\cdot  \lft E_3 \rg^{\gamma_3},
\]
where $\gamma_1, \gamma_2, \gamma_3$ are numbers between $0$ and $1$, with $\gamma_1+\gamma_2+\gamma_3=1$. This implies
\begin{align*}
\big\| \Pi_{I_0}\left( f \cdot \one_F, g \cdot \one_G \right) \cdot \one_{\tilde E_3} \cdot \one_{\tilde{E}} \big\|_\tau^\tau & \lesssim \sum_{n_1, n_2, n_3}  K^\tau 2^{-n_1 \left(\frac{ \tau}{s_1}-\gamma_1\right) }
\cdot  2^{-n_2\left( \frac{\tau}{s_2}-\gamma_2 \right) } 2^{- n_3 \left(  1-\frac{\tau}{s}+\epsilon-\gamma_3 \right)} \\
&  \qquad\cdot \big\| \one_{E_1} \cdot \ci_{I_0} \big\|_1^{\gamma_1} \cdot \big\| \one_{E_2} \cdot \ci_{I_0} \big\|_1^{\gamma_2} \cdot \lft E_3 \rg^{\gamma_3} \\
&\lesssim K^\tau 2^{-\tilde{M} d} \big\| \one_{E_1} \cdot \ci_{I_0} \big\|_1^{\frac{\tau}{s_1}} \big\| \one_{E_2} \cdot \ci_{I_0} \big\|_1^{\frac{\tau}{s_2}} \lft E_3 \rg^{1-\frac{\tau}{s}}.
\end{align*}
Above we used that $2^{-n_1} \leq 2^d\|  \one_{E_1} \cdot \ci_{I_0} \|_1/{|E_3|}$. The only constraint we have is that 
\[
\frac{\tau}{s_1}+\frac{\tau}{s_2}+ 1-\frac{\tau}{s}+\epsilon >1=\gamma_1+\gamma_2+\gamma_3,
\] 
which is definitely true for $\epsilon>0$. We obtained in this way the desired inequality \eqref{eq:weak-norms-local-paraprod}.
\end{proof}
\end{proposition}

\subsection{Proof of Theorem \ref{thm:one-quasiBanach-Paraprod}}

Now we are ready to prove a quasi-Banach valued inequality for paraproducts, as well as its localized version, corresponding to the scalar Proposition \ref{prop:localization-lemma}.

\begin{theorem}
\label{thm:localized-one-quasiBanach-Paraprod}
For  $r_1, r_2, r$ as in Theorem \ref{thm:one-quasiBanach-Paraprod}, the localized paraproduct $\Pi_{I_0}^{F, G, \tilde E}$ defined in \eqref{eq-def-loc-paraprod} satisfies the estimate
\begin{align*}
\big\| \big( \sum_{k} \lft \Pi_{I_0}^{F, G, \tilde E} \left(f_k, g_k  \right) \rg^r \big)^{1/r} \big\|_s &\lesssim 
\big(\sssize_{I_0} \one_F \big)^{1/{s_1'}-\epsilon} \cdot \big(\sssize_{I_0} \one_G \big)^{1/{s_2'}-\epsilon} \cdot \big(\sssize_{I_0} \one_{\tilde{E}} \big)^{1/{s}-\epsilon} \\
& \cdot \big\| \big( \sum_{k} \lft f_k \rg^{r_1} \big)^{1/{r_1}} \cdot \ci_{I_0}\big\|_{s_1} \cdot \big\| \big( \sum_{k} \lft g_k\rg^{r_2} \big)^{1/{r_2}} \cdot \ci_{I_0} \big\|_{s_2},
\end{align*}
for any $1< s_1, s_2 \leq \infty$, $\frac{1}{2}<s <\infty$ so that $\ds \frac{1}{s_1}+\frac{1}{s_2}=\frac{1}{s}$.
\end{theorem}

Even though Theorem \ref{thm:one-quasiBanach-Paraprod} can be regarded as a limiting case of Theorem \ref{thm:localized-one-quasiBanach-Paraprod}, we first give a proof of the former. It will be clear why this result is not necessarily sharp, and how it can be improved.

\begin{proof}[Proof of Theorem \ref{thm:one-quasiBanach-Paraprod}]
The proof is similar to the Banach-space case from \cite{vv_BHT}, with the exception that now we use $L^r$ dualization, as presented in Proposition \ref{prop:Lr-dualization}.

Fix $F, G, E$ sets of finite measure, and let $\lbrace f_k\rbrace_k, \lbrace g_k\rbrace_k$ be so that
\begin{equation}
\label{eq:cond-resticted-type-qB}
\big( \sum_k \lft f_k \rg^{r_1}  \big)^{1/{r_1}} \leq \one_F \quad \text{a. e.    and} \quad \big( \sum_k \lft g_k \rg^{r_2}  \big)^{1/{r_2}} \leq \one_G \qquad \text{a. e.}
\end{equation}

For simplicity, we assume that $\lft E \rg=1$, and we construct the set $\tilde{E}$ by removing the parts where $\ic{M}\left(\one_F \right)$ and $\ic{M}\left(\one_G \right)$ are big:
\begin{equation}
\label{def:exceptional-set}
\Omega:=\left\lbrace x : \ic{M}\left(\one_F \right)(x)> C \lft F \rg  \right\rbrace \cap \left\lbrace x : \ic{M}\left(\one_G \right)(x)> C \lft G \rg  \right\rbrace,
\end{equation}
and $\tilde{E}:= E \setminus \Omega$.

We assume that all the intervals $I \in \ii{I}$ have the property that 
\[
1+\frac{\dist(I , \Omega^c)}{\lft I \rg} \sim 2^d.
\]
This will allow us to have a better control over the maximal functions of $\one_F$ and $\one_G$, and in consequence over their sizes. It will be enough to prove
\begin{equation}
\label{eq:proof-quasi-banach-wanted}
\big\|\big( \sum_k \lft \Pi_{\ii{i}}\left( f_k, g_k \right) \rg^r \big)^{1/r} \cdot \one_{\tilde{E}} \big\|_{r} \lesssim 2^{-100d} \lft F \rg^{1/{\tilde s_1}} \lft G \rg^{1/{\tilde s_2}},
\end{equation}
for $\frac{1}{\tilde s_1}, \frac{1}{\tilde s_2}$ arbitrarily close to $ \frac{1}{r_1}$ and $\frac{1}{r_2}$ respectively.

As usual, we perform a triple stopping time through which we select the collections of intervals $\ds \ii{I}_{n_1},  \ii{I}_{n_2}$ and $ \ii{I}_{n_3}$
\begin{equation}
\label{eq:stopping-time_1}
I \in \ii{I}_{n_1}: \qquad 2^{-n_1-1} \leq \frac{1}{\lft I \rg} \int_{\rr{R}} \one_F(x) \cdot \ci_{I}^M dx \leq 2^{-n_1} \lesssim 2^d \lft F \rg. 
\end{equation} 
Similarly, for $\one_G$ and $\one_{\tilde{E}}$ we have
\begin{equation}
\label{eq:stopping-time_2}
I \in \ii{I}_{n_2}: \qquad 2^{-n_2-1} \leq \frac{1}{\lft I \rg} \int_{\rr{R}} \one_G(x) \cdot \ci_{I}^M dx \leq 2^{-n_1} \lesssim 2^d \lft G \rg, 
\end{equation} 
\begin{equation}
\label{eq:stopping-time_3}
I \in \ii{I}_{n_3}: \qquad 2^{-n_3-1} \leq \frac{1}{\lft I \rg} \int_{\rr{R}} \one_{\tilde{E}}(x) \cdot \ci_{I}^M dx \leq 2^{-n_3} \lesssim 2^{-Md}.
\end{equation} 

We denote $\ds \ii{I}^{n_1, n_2, n_3}:=\ii{I}_{n_1}\cap \ii{I}_{n_2} \cap \ii{I}_{n_3}$. Using the sub-additivity of $ \| \cdot \|_r^r$, we get that 
\begin{align*}
& \big\|\big( \sum_k \lft \Pi_{\ii{i}}\left( f_k, g_k \right) \rg^r \big)^{1/r} \cdot \one_{\tilde{E}} \big\|_{r}^r =\sum_k \big\|  \Pi\left( f_k, g_k \right) \cdot \one_{\tilde{E}} \big\|_r^r \\
&=\sum_{n_1, n_2, n_3} \sum_{I_0 \in \ii{I}^{n_1, n_2, n_3}} \big\|  \Pi_{\ii{I}\left( I_0 \right)}\left( f_k, g_k \right) \cdot \one_{\tilde{E}} \big\|_r^r.
\end{align*}

In fact, since the functions $f_k$ and $g_k$ are supported on $F$ and $G$ respectively, we have that $\ds \Pi_{\ii{I}\left( I_0 \right)}\left( f_k, g_k \right) \cdot \one_{\tilde{E}} =\Pi_{\ii{I}\left( I_0 \right)}\left( f_k \cdot \one_F, g_k \cdot \one_G \right) \cdot \one_{\tilde{E}}$, and Proposition \ref{prop:localization-lemma} implies
\begin{align*}
& \big\|\big( \sum_k \lft \Pi_{\ii{i}}\left( f_k, g_k \right) \rg^r \big)^{1/r} \cdot \one_{\tilde{E}} \big\|_{r}^r \\
& \lesssim \sum_{n_1, n_2, n_3} \sum_{I_0 \in \ii{I}^{n_1, n_2, n_3}} 2^{-n_1\left( \frac{r}{r_1'}-\epsilon \right)} 2^{-n_2\left( \frac{r}{r_2'} -\epsilon\right)} 2^{-n_3\left( 1-\epsilon \right)}  \sum_k \big\| f_k \cdot \ci_{I_0} \big\|_{r_1}^{r} \cdot \big\| g_k \cdot \ci_{I_0} \big\|_{r_2}^{r} \\
&\lesssim \sum_{n_1, n_2, n_3} \sum_{I_0 \in \ii{I}^{n_1, n_2, n_3}} 2^{-n_1\left( \frac{r}{r_1'}-\epsilon \right)} 2^{-n_2\left( \frac{r}{r_2'} -\epsilon\right)} 2^{-n_3\left( 1-\epsilon \right)} \big\| \one_F \cdot \ci_{I_0} \big\|_{r_1}^{r} \big\| \one_G \cdot \ci_{I_0} \big\|_{r_2}^{r}.
\end{align*}
Above, we used H\"{o}lder's inequality, together with property \eqref{eq:cond-resticted-type-qB}. Because of the stopping time properties \eqref{item-1} - \eqref{item-4}, we can estimate the norms of $\one_F \cdot \ci_{I_0}$ and $\one_G \cdot \ci_{I_0}$ by:
\[
\big\| \one_F \cdot \ci_{I_0} \big\|_{r_1}^r \lesssim \left( 2^{-n_1} \lft I_0 \rg \right)^{\frac{r}{r_1}}, \quad \text{and }   \quad      \big\| \one_G \cdot \ci_{I_0} \big\|_{r_2}^r \lesssim \left(  2^{-n_2}\lft I_0 \rg \right)^{\frac{r}{r_2}}.
\]
In this way, the estimate we obtain for the $\ell^r$-valued paraproduct is
\begin{align*}
& \big\|\big( \sum_k \lft \Pi_{\ii{i}}\left( f_k, g_k \right) \rg^r \big)^{1/r} \cdot \one_{\tilde{E}} \big\|_{r}^r \lesssim \sum_{n_1, n_2, n_3} \sum_{I_0 \in \ii{I}^{n_1, n_2, n_3}} 2^{-n_1\left( r-\epsilon \right)} 2^{-n_2\left( r -\epsilon\right)} 2^{-n_3\left( 1-\epsilon \right)} \lft I_0 \rg.
\end{align*}

The sum $\ds \sum_{I_0 \in \ii{I}^{n_1, n_2, n_3}} \lft I_0 \rg$ can be bounded above by
\[
2^{n_1} \lft F \rg, \qquad 2^{n_2} \lft G \rg, \qquad \text{ and     } 2^{n_3}.
\]
Hence, if $\gamma_1, \gamma_2, \gamma_3$ are so that $0 \leq \gamma_1, \gamma_2, \gamma_3 \leq 1$ and $\gamma_1+\gamma_2+\gamma_3=1$, 
\[
\sum_{I_0 \in \ii{I}^{n_1, n_2, n_3}} \lft I_0 \rg \lesssim \left( 2^{n_1} \lft F \rg  \right)^{\gamma_1} \left( 2^{n_2} \lft G \rg  \right)^{\gamma_2} \left( 2^{n_3} \right)^{\gamma_3}.
\]


Since the sizes are all sub-unitary, the equations \eqref{eq:stopping-time_1}, \eqref{eq:stopping-time_2} and \eqref{eq:stopping-time_3} imply that 
\begin{align}
\label{eq:qB-one-loss}
& \big\|\big( \sum_k \lft \Pi_{\ii{i}}\left( f_k, g_k \right) \rg^r \big)^{1/r} \cdot \one_{\tilde{E}} \big\|_{r}^r \lesssim  \sum_{n_1, n_2, n_3}  2^{-n_1\left( \frac{r}{\tilde s_1} -\gamma_1 \right)} 2^{-n_2\left( \frac{r}{\tilde s_2} -\gamma_2 \right)} 2^{-n_3\left(1-\epsilon -\gamma_3\right)} \lft F \rg^{\gamma_1}  \lft G \rg^{\gamma_2}.
\end{align}
The series above converges provided $\ds \frac{1}{\tilde s_1}+\frac{1}{\tilde s_2}+\frac{1-\epsilon}{r} >\frac{1}{r}$, which is true for $( \tilde s_1, \tilde s_2, \tilde s)$ in a neighborhood of $\ds ( r_1, r_2, r)$.

We obtain that
\begin{align*}
 \big\|\big( \sum_k \lft \Pi_{\ii{i}}\left( f_k, g_k \right) \rg^r \big)^{1/r} \cdot \one_{\tilde{E}} \big\|_{r}^r & \lesssim \left( 2^d \lft F \rg  \right)^{\left( \frac{r}{\tilde s_1} -\gamma_1 \right)} \cdot \left( 2^d \lft G \rg  \right)^{\left( \frac{r}{\tilde s_2} -\gamma_2 \right)} 2^{-Md \left(\epsilon -\gamma_3\right)} \lft F \rg^{\gamma_1}  \lft G \rg^{\gamma_2},
\end{align*}
which is exactly \eqref{eq:proof-quasi-banach-wanted}.
\end{proof}

Now we proceed with the localized version:

\begin{proof}[Proof of Theorem \ref{thm:localized-one-quasiBanach-Paraprod}]~\\
The proof consists of two steps:
\begin{itemize}
\item[i)] First prove Theorem \ref{thm:localized-one-quasiBanach-Paraprod} in the case $s \geq r$.
\item[ii)]For $\frac{1}{2}<s <r$, we make use of the result corresponding to $s=r$, which was proved in the previous step.
\end{itemize}

\textbf{\emph{i)The case $s \geq r$:}}~\\
The result in this case follows from a careful examination of the proof of Theorem \ref{thm:one-quasiBanach-Paraprod}. We noticed in \eqref{eq:qB-one-loss} that we lose some information by changing the exponents of the subunitary sizes from $r$ to $\frac{r}{\tilde s_j}$, where $1 < \tilde s_j< \infty$. It is this point that we will modify in order to obtain the sharper estimate in Theorem \ref{thm:localized-one-quasiBanach-Paraprod}.

This time, the sets $F, G, \tilde E$ are fixed, and we have two sequences of functions $\lbrace f_k \rbrace_k$ and $\lbrace g_k \rbrace_k$ satisfying 
\begin{equation}
\label{eq:rest-case_s}
\big( \sum_k \lft f_k \rg^{r_1}  \big)^{1/{r_1}} \leq \one_{E_1}, \qquad \big( \sum_k \lft g_k \rg^{r_2}  \big)^{1/{r_2}} \leq \one_{E_2},
\end{equation}
where $E_1$ and $E_2$ are sets of finite measure.

Following a variant of Proposition \ref{prop:interpolation} for general measures, it will be enough to prove
\begin{equation}
\label{eq:restr-one-vv-loc}
\big\| \big( \sum_k \lft \Pi_{I_0}^{F, G, \tilde E}( f_k, g_k) \rg^{r} \big)^{1/{r}}  \big\|_{\tilde s , \infty} \lesssim \big\| \Pi_{I_0}^{F, G, \tilde E} \big\| \cdot \big\| \one_{E_1} \cdot \ci_{I_0} \big\|_{\tilde s_1} \big\|  \one_{E_2} \cdot \ci_{I_0} \big\|_{\tilde s_2},
\end{equation}
where the operatorial norm is 
\[
\big\| \Pi_{I_0}^{F, G, \tilde E} \big\| \sim \big(  \sssize_{I_0} \one_F  \big)^{\frac{1}{\tilde s_1'}-\epsilon} \cdot \big(  \sssize_{I_0} \one_G  \big)^{\frac{1}{\tilde s_2'}-\epsilon} \cdot \big(  \sssize_{I_0} \one_{\tilde{E}}  \big)^{\frac{1}{r}-\epsilon}
\]
and $( \tilde s_1, \tilde s_2, \tilde s)$ is an admissible tuple lying in an arbitrarily small neighborhood of $( r_1, r_2, r)$. We note that the conditions in \eqref{eq:rest-case_s} imply that 
\[
\Pi_{I_0}^{F, G, \tilde E}( f_k, g_k)=\Pi_{I_0}^{F \cap E_1, G\cap E_2, \tilde E}( f_k, g_k).
\]

We dualize the $\| \cdot \|_{\tilde s, \infty}$ norm through $L^r$, just like in the proof of Theorem \ref{thm:one-quasiBanach-Paraprod}: given $E_3$ a set of finite measure, we define the exceptional set 
\[
\tilde \Omega:=\left\lbrace x:  \ic{M}\left( \one_{E_1} \cdot \ci_{I_0}  \right)(x) > C \frac{\| \one_{E_1} \cdot \ci_{I_0} \|_1}{\lft  E_3\rg}  \right\rbrace \cup \left\lbrace x:  \ic{M}\left( \one_{E_2} \cdot \ci_{I_0}  \right)(x) > C \frac{\| \one_{E_2} \cdot \ci_{I_0} \|_1}{\lft  E_3\rg}  \right\rbrace,
\]
and set $\tilde E_3:= E_3 \setminus \tilde \Omega$. Then $\tilde E_3$ is a major subset of $E_3$ and
\[
\big\| \big( \sum_k \lft \Pi_{I_0}^{F, G, \tilde E}( f_k, g_k) \rg^{r} \big)^{1/{r}}  \big\|_{\tilde s , \infty} \sim \big\| \big( \sum_k \lft \Pi_{I_0}^{F, G, \tilde E}( f_k, g_k) \rg^{r} \big)^{1/{r}} \cdot \one_{\tilde E_3}  \big\|_r \cdot \lft E_3 \rg^{\frac{1}{\tilde s}-\frac{1}{r}}.
\]
Since we are in the case $r<1$, $\| \cdot \|_r^r$ is subadditive and we have
\begin{align*}
&\big\| \big( \sum_k \lft \Pi_{I_0}^{F, G, \tilde E}( f_k, g_k) \rg^{r} \big)^{1/{r}} \cdot \one_{\tilde E_3}  \big\|_r^r \lesssim  \sum_{\bar n_1, \bar n_2, \bar n_3} \sum_{I \in \ii{I}^{\bar n_1, \bar n_2, \bar n_3}} \big( \sssize_I \one_{F} \cdot \one_{E_1} \big)^{\frac{r}{r_1'}-\epsilon} \\
& \cdot \big( \sssize_I \one_{G} \cdot \one_{E_2} \big)^{\frac{r}{r_2'}-\epsilon} \big( \sssize_I \one_{\tilde E} \cdot \one_{\tilde E_3} \big)^{1-\epsilon} \sum_k \| f_k \cdot \ci_{I}  \|_{r_1}^{r} \| g_k \cdot \ci_{I}  \|_{r_2}^{r}\\
&\lesssim \sum_{\bar n_1, \bar n_2, \bar n_3} \sum_{I \in \ii{I}^{\bar n_1, \bar n_2, \bar n_3}} \big( \sssize_I \one_{F} \cdot \one_{E_1} \big)^{r-\epsilon} \cdot \big( \sssize_I \one_{G} \cdot \one_{E_2} \big)^{r-\epsilon} \big( \sssize_I \one_{\tilde E} \cdot \one_{\tilde E_3} \big)^{1-\epsilon} \cdot \lft I \rg.
\end{align*}
where we used \eqref{eq:localized-paraproduct} from Proposition \ref{prop:localization-lemma}, together with H\"older's inequality and the trivial estimate $\| \one_{F\cap E_1} \cdot \ci_I \|_1 \leq (\sssize_{I} \one_{F\cap E_1}) \cdot | I |$. Here the collections $\ii I^{\bar n_1, \bar n_2, \bar n_3}$ of intervals are defined in the same way as in Section \ref{subsection-stopping_times}.

From the expression $\ds \big( \sssize_I \one_{F} \cdot \one_{E_1} \big)^{r-\epsilon}$, a part will the used to rebuild the norms of $\one_{E_1}$ and $\one_{E_2}$, while the rest becomes part of the sharp operatorial norm:
\[
\big( \sssize_I \one_{F} \cdot \one_{E_1} \big)^{r-\epsilon} \lesssim \big( \sssize_{I_0} \one_{F} \big)^{\frac{r}{\tilde s_1'}-\epsilon} \cdot \big( \sssize_I \one_{E_1} \big)^{\frac{r}{\tilde s_1} }.
\]
Similar estimates are used for $\one_G \cdot \one_{E_2}$ and $\one_{\tilde E} \cdot \one_{\tilde E_3}$:
\begin{align*}
\big\| \big( \sum_k \lft \Pi_{I_0}^{F, G, \tilde E}( f_k, g_k) \rg^{r} \big)^{1/{r}} \cdot \one_{\tilde E_3}  \big\|_r^r &\lesssim 
\big( \sssize_{I_0} \one_{F} \big)^{\frac{r}{\tilde s_1'}-\epsilon} \cdot  \big( \sssize_{I_0} \one_{G} \big)^{\frac{r}{\tilde s_2'}-\epsilon} \cdot  \big( \sssize_{I_0} \one_{\tilde E} \big)^{\frac{r}{\tilde s}-\epsilon} \\
& \cdot \sum_{\bar n_1, \bar n_2, \bar n_3} \sum_{I \in \ii{I}^{\bar n_1, \bar n_2, \bar n_3}} 2^{- \frac{ n_1r}{\tilde s_1}} \cdot 2^{-\frac{n_2 r}{\tilde s_2}} 2^{-n_3 \left(1 -\frac{r}{\tilde s} +\epsilon   \right)} \lft I \rg.
\end{align*}
The last part adds up exactly to $\ds \| \one_{E_1} \cdot \ci_{I_0} \|_1^\frac{r}{\tilde s_1}  \| \one_{E_2} \cdot \ci_{I_0} \|_1^{\frac{r}{\tilde s_2}} \lft E_3 \rg^{1- \frac{r}{ \tilde{s}}}$, which proves \eqref{eq:restr-one-vv-loc}.

\textbf{\emph{ii) The case $s < r$:}}~\\
This case uses an extra step. The interval $I_0$ and the sets $F, G, \tilde E, E_1, E_2, E_3$ and $\tilde E_3$ are define as before; the only difference is that now we dualize through $L^\tau$, where $\tau <s$:
\[
\big\| \big( \sum_k \lft \Pi_{I_0}^{F, G, \tilde E}( f_k, g_k) \rg^{r} \big)^{1/{r}}  \big\|_{\tilde s , \infty} \sim \big\| \big( \sum_k \lft \Pi_{I_0}^{F, G, \tilde E}( f_k, g_k) \rg^{r} \big)^{1/{r}} \cdot \one_{\tilde E_3} \big\|_\tau \cdot \lft E_3 \rg^{\frac{1}{\tilde{s}}-\frac{1}{\tau}}.
\]
We use the monotonicity and the  subadditivity of $\| \cdot \|_\tau^\tau$:
\begin{align*}
\big\| \big( \sum_k \lft \Pi_{I_0}^{F, G, \tilde E}( f_k, g_k) \rg^{r} \big)^{1/{r}} \cdot \one_{\tilde E_3} \big\|_\tau^\tau \lesssim \sum_{\bar n_1, \bar n_2, \bar n_3} \sum_{I \in \ii{I}^{\bar n_1, \bar n_2, \bar n_3}} \big\| \big( \sum_k \lft \Pi_{I}^{F, G, \tilde E}( f_k, g_k) \rg^{r} \big)^{1/{r}} \cdot \one_{\tilde E_3} \big\|_\tau^\tau ,
\end{align*}
where the collection $\ii{I}_{\bar n_1}, \ii{I}_{\bar n_2}, \ii{I}_{\bar n_3}$ are defined by the usual triple stopping time. Now we use a trick similar to the one in \eqref{eq:trick_1}: since $\tau<s<r$, we can write $\ds \frac{1}{\tau} =\frac{1}{r}+\frac{1}{\tau_r}$, where $\tau_r >0$. We have, due to H\"older, that 
\[
\big\| \big( \sum_k \lft \Pi_{I}^{F, G, \tilde E \cap \tilde E_3}( f_k, g_k) \rg^{r} \big)^{1/{r}} \big\|_\tau \lesssim \big\| \big( \sum_k \lft \Pi_{I}^{F, G, \tilde E \cap \tilde E_3}( f_k, g_k) \rg^{r} \big)^{1/{r}} \big\|_r \cdot \|  \one_{\tilde E} \cdot \one_{\tilde E_3} \cdot \ci_{I}  \|_{\tau_r}.
\]

On the right hand side, we initially have $\lft \tilde E \cap \tilde E_3 \rg^{1/{r_\tau}}$, but we will soon see that $\sssize_{I}(\one_{\tilde E} \cdot \one_{\tilde E_3})$ appears, accounting for the decay when $\tilde E \cap \tilde E_3$ is supported away from $I$. This step was explained in more detail in Proposition \ref{prop:Localization for local L^1 paraproducts}.

Using the \emph{sharp} estimate from Theorem \ref{thm:localized-one-quasiBanach-Paraprod} in the case $s=r$, we have
{\fontsize{10}{10}
\begin{align*}
&\big\| \big( \sum_k \lft \Pi_{I}^{F, G, \tilde E \cap \tilde E_3}( f_k, g_k) \rg^{r} \big)^{1/{r}} \big\|_\tau \lesssim \big( \sssize_{I} \one_{F \cap E_1} \big)^{1-\epsilon} \big( \sssize_{I} \one_{G \cap E_2} \big)^{1-\epsilon} \big( \sssize_{I} \one_{\tilde E \cap \tilde E_3}  \big)^{\frac{1}{r}+\frac{1}{\tau_r}-\epsilon} \lft I \rg^{\frac{1}{r}+\frac{1}{\tau_r}}.
\end{align*}}

Here we actually use the result from the previous step for the operator  $\ds \Pi_{I}^{F \cap E_1, G \cap E_2, \tilde E \cap \tilde E_3}$, which coincides with $\Pi_{I}^{F, G, \tilde E \cap \tilde E_3}$ when applied to functions $f$ supported on $E_1$ and functions $g$ supported on $E_2$.

From here on, the proof is identical to the previous case $s \geq r$, and we have
\begin{align*}
&\big\| \big( \sum_k \lft \Pi_{I_0}^{F, G, \tilde E \cap \tilde E_3}( f_k, g_k) \rg^{r} \big)^{1/{r}}\big\|_\tau^\tau \lesssim \big( \sssize_{I_0} \one_F \big)^{\frac{\tau}{\tilde s_1'} -\epsilon} \big( \sssize_{I_0} \one_G \big)^{\frac{\tau}{\tilde s_2'} -\epsilon} \big( \sssize_{I_0} \one_{\tilde E} \big)^{\frac{\tau}{\tilde s} -\epsilon} \\
& \qquad \cdot \sum_{\bar n_1, \bar n_2, \bar n_3} \sum_{I \in \ii{I}^{\bar n_1, \bar n_2, \bar n_3}} 2^{-\frac{n_1 \tau}{\tilde s_1}} \cdot 2^{\frac{-n_2 \tau}{\tilde s_2} } 2^{-n_3 \left(1 -\frac{\tau}{\tilde s} +\epsilon   \right)} \lft I \rg \\
&\lesssim \big( \sssize_{I_0} \one_F \big)^{\frac{\tau}{\tilde s_1'} -\epsilon} \big( \sssize_{I_0} \one_G \big)^{\frac{\tau}{\tilde s_2'} -\epsilon} \big( \sssize_{I_0} \one_{\tilde E} \big)^{\frac{\tau}{\tilde s} -\epsilon} \| \one_{E_1} \cdot \ci_{I_0}  \|_1^{\frac{\tau}{\tilde s_1}} \| \one_{E_2} \cdot \ci_{I_0}  \|_1^{\frac{\tau}{\tilde s_2} } \lft E_3\rg^{1-\frac{\tau}{\tilde{s}}},
\end{align*}
for any admissible tuple $( \tilde s_1, \tilde s_2, \tilde s)$ in a neighborhood of $( s_1, s_2, s)$. We don't include all the details because they are identical to the previous case $s \geq r$.
\end{proof}
\subsection{The case $p=\infty$ or $q=\infty$}~\\
We need to treat separately the cases when $p=\infty$ or $q=\infty$. Commonly, these are handled with the help of the two adjoint operators, but here we work with quasi-Banach valued bilinear operators, and cannot consider the associated trilinear form. We will only prove the case $q=\infty$ for Theorem \ref{thm:localized-one-quasiBanach-Paraprod}, since Theorem \ref{thm:one-quasiBanach-Paraprod} can be seen as a limiting case of the former. Also, the case $p=\infty$ or or $q=\infty$ for Theorems \ref{thm:main-thm-paraprod} or \ref{thm:main-thm-BHT} can be treated similarly and we will not elaborate on the details.

\begin{proof}[Proof of Theorem \ref{thm:localized-one-quasiBanach-Paraprod} for $q=\infty$]

Here we want to prove 
\begin{align*}
\big\| \big( \sum_k \big| \Pi_{I_0}^{F,G, \tilde E} (f_k, g_k) \big|^r   \big)^{1/r}   \big\|_s  &\lesssim \big( \sssize_{I_0} \one_F \big)^{\frac{1}{s'}-\epsilon} \big( \sssize_{I_0} \one_G \big)^{1-\epsilon} \big( \sssize_{I_0} \one_{\tilde E} \big)^{\frac{1}{s}-\epsilon}  \\
&\cdot \big\| \big(  \sum_k \lft  f_k \rg^{r_1}  \big)^{1/{r_1}} \cdot \ci_{I_0} \big\|_s \cdot \big\| \big(  \sum_k \lft  g_k \rg^{r_2}  \big)^{1/{r_2}} \cdot \ci_{I_0} \big\|_\infty.
\end{align*}

In order to achieve this, we will use a linear version of Proposition \ref{prop:interpolation}. That is, we consider $\lbrace g_k \rbrace_k$ to be a fixed sequence of functions so that $\ds  \big(  \sum_k \lft  g_k \rg^{r_2}  \big)^{1/{r_2}}  \in L^\infty$, and it will be enough to prove
\begin{equation}
\label{eq:restr-lin-infty}
\big\| \big( \sum_k \big| \Pi_{I_0}^{F,G, \tilde E} (f_k, g_k) \big|^r   \big)^{1/r}   \big\|_{\tilde s, \infty} \lesssim \| \Pi_{I_0}^{F, G, \tilde E} \| \| \one_{E_1} \cdot \ci_{I_0} \|_{\tilde s} \cdot \big\| \big(  \sum_k \lft  g_k \rg^{r_2}  \big)^{1/{r_2}} \cdot \ci_{I_0} \big\|_\infty,
\end{equation}
where, this time, $\big(  \sum_k \lft  f_k \rg^{r_1}  \big)^{1/{r_1}} \leq \one_{E_1}$, and 
\[
\| \Pi_{I_0}^{F, G, \tilde E} \|  := \big( \sssize_{I_0} \one_F \big)^{\frac{1}{\tilde s'}-\epsilon} \big( \sssize_{I_0} \one_G \big)^{1-\epsilon} \big( \sssize_{I_0} \one_{\tilde E} \big)^{\frac{1}{\tilde s}-\epsilon}.
\]
The idea in this case is to isolate the $L^\infty$ norm of $\ds \big(  \sum_k \lft  g_k \rg^{r_2}  \big)^{1/{r_2}}$; in other words, we will not use the functions $\lbrace g_k \rbrace$ in the stopping time. The exceptional set is defined as
\[
\tilde \Omega=\left\lbrace x: \ic M \left( \one_{E_1} \cdot \ci_{I_0} \right) >C \frac{\| \one_{E_1} \cdot \ci_{I_0}  \|_1}{|E_3|}  \right\rbrace, 
\] 
and $\tilde E_3:=E_3 \setminus \tilde \Omega$. (As usual, we will dualize through $L^r$, and for any given $E_3$, we construct $\tilde E_3$ a major subset...)

The stopping time will differ from the general one described in Section \ref{subsection-stopping_times} in the sense that it will be a double stopping time involving only level sets of $\one_{F \cap E_1}$ and $\one_{\tilde E \cap \tilde E_3}$. The rest will be analogous to the proof of Theorem \ref{thm:localized-one-quasiBanach-Paraprod}.
\begin{align*}
&\big \|  \big( \sum_k \lft \Pi_{I_0}^{F, G, \tilde E}(f_k, g_k)\rg^r   \big)^{1/r} \cdot \one_{\tilde E_3} \big\|_r^r=\sum_k \big \|  \Pi_{I_0}^{F, G, \tilde E}(f_k, g_k) \cdot \one_{\tilde E_3} \big\|_r^r \\
& \lesssim \sum_{n_1, n_3} \sum_{I \in \ii I^{n_1, n_3}} \sum_k \big \|  \Pi_{I}^{F, G, \tilde E}(f_k, g_k) \cdot \one_{\tilde E_3} \big\|_r^r \\
& \lesssim \sum_{n_1, n_3} \sum_{I \in \ii I^{n_1, n_3}} \big( \sssize_{I} \one_{F \cap E_1} \big)^{\frac{r}{r_1'}-\epsilon} \big( \sssize_{I} \one_G \big)^{\frac{r}{r_2'}-\epsilon} \big( \sssize_{I} \one_{\tilde E \cap \tilde E_3} \big)^{1-\epsilon} \sum_k \| f_k \cdot \one_F\cdot \ci_I\|_{r_1}^r \cdot \| g_k  \cdot \one_G\cdot \ci_I\|_{r_2}^r.
\end{align*}
Above, we used the subadditivity of $\|  \cdot\|_r^r$, as well as the local estimate proved in Proposition \ref{prop:localization-lemma}. For the term on the most right, we use H\"older to get
\begin{align*}
&\sum_k \| f_k \cdot \one_F \cdot \ci_I\|_{r_1}^r \cdot \| g_k \cdot \one_G \cdot \ci_I\|_{r_2}^r \lesssim \big\| \big( \sum_k \lft f_k \cdot \one_F \rg^{r_1} \big)^{1/{r_1}} \cdot \ci_{I}\big\|_{r_1}^r \cdot \big\| \big( \sum_k \lft g_k \cdot \one_{G}\rg^{r_2} \big)^{1/{r_2}} \cdot \ci_{I}\big\|_{r_2}^r \\
&\qquad \lesssim \| \one_{F \cap E_1} \cdot \ci_I \|_{r_1}^r \cdot \big\| \big(  \sum_k \lft  g_k \rg^{r_2}  \big)^{1/{r_2}} \cdot \ci_{I_0} \big\|_\infty^r \cdot  \| \one_{G} \cdot \ci_I \|_{r_2}^r.
\end{align*}

The estimate \eqref{eq:restr-lin-infty} reduces to proving 
\[
\sum_{n_1, n_3} \sum_{I \in \ii I^{n_1, n_3}} \big( \sssize_I \one_{E_1}  \big)^{\frac{r}{\tilde s}} \big( \sssize_I \one_{\tilde E_3}  \big)^{1-\frac{r}{\tilde s} +\epsilon} |I| \lesssim \| \one_{E_1} \cdot \ci_{I_0}  \|_1^{\frac{r}{\tilde s}} |E_3|^{1-\frac{r}{\tilde s}}.
\]
Here we need to remember the properties achieved through the double stopping time: if $I \in \ii I ^{n_1, n_3}$
\[
 \sssize_I \one_{E_1} \leq 2^{-n_1} \lesssim 2^d \frac{ \| \one_{E_1} \cdot \ci_{I_0}  \|_1}{|E_3|}, \quad \sssize_I \one_{\tilde E_3} \leq 2^{-n_3} \lesssim 2^{-Md}
\]
and also
\[
\sum_{I \in \ii I^{n_1, n_3}} |I| \lesssim \left( 2^{n_1}  \| \one_{E_1} \cdot \ci_{I_0}  \|_1 \right)^{\gamma_1} |E_3|^{ 1- \gamma_1},
\]
for any $0 < \gamma_1<1$.

All of the above imply
\begin{align*}
&\sum_{n_1, n_3} \sum_{I \in \ii I^{n_1, n_3}} \big( \sssize_I \one_{E_1}  \big)^{\frac{r}{\tilde s}} \big( \sssize_I \one_{\tilde E_3}  \big)^{1-\frac{r}{\tilde s} +\epsilon} |I| \\
&\lesssim \sum_{n_1, n_3} 2^{-n_1\left(\frac{r}{\tilde s} -\gamma_1 \right)} 2^{-n_3 \left( 1-\frac{r}{\tilde s} +\epsilon  -1+ \gamma_1 \right)} \cdot  \| \one_{E_1} \cdot \ci_{I_0}  \|_1 ^{\gamma_1} \cdot |E_3|^{ 1- \gamma_1} \\
&\lesssim \left( 2^d \frac{ \| \one_{E_1} \cdot \ci_{I_0}  \|_1}{|E_3|}  \right)^{\frac{r}{\tilde s} -\gamma_1} 2^{-\tilde M d}  \| \one_{E_1} \cdot \ci_{I_0}  \|_1^{\gamma_1} \cdot |E_3|^{ 1- \gamma_1} \\
&\lesssim \| \one_{E_1} \cdot \ci_{I_0}  \|_1^{\frac{r}{\tilde s} } |E_3|^{1-\frac{r}{\tilde s}},
\end{align*}
which is precisely what we wanted.

This ends the proof for the case when one of $p$ or $q$ is $\infty$, which is also necessary for the proof of multiple vector-valued estimates when $L^\infty$(or $\ell^\infty$) spaces are involved (i.e. some $r^j_1=\infty$).
\end{proof}

\section{Multiple Vector-Valued Inequalities}
\label{sec:multiple vector-valued inequalities}

We prove the general Theorem \ref{thm:main-thm-paraprod} by induction. In fact, since the sizes are subunitary, it will be enough to prove its localized version:

\begin{theorem}
\label{thm:genral_case}
Let $I_0$ be a fixed dyadic interval, and $F, G, \tilde E$ sets of finite measure. $R^n=\left( r^1, \ldots r^n \right)$ is an $n$-tuple containing at least one index $r^j<1$, and $R^n_1, R^n_2$ are $n$-tuples satisfying component-wise $1< r^j_i \leq \infty, \frac{1}{2}<r^j<\infty$, $\ds \frac{1}{r^j_1}+\frac{1}{r^j_2}=\frac{1}{r^j}$. Then the localized paraproduct $\Pi_{I_0}^{F, G, \tilde E}$ defined in \eqref{eq-def-loc-paraprod}, satisfies the following estimates:
{\fontsize{10}{10}
\begin{equation}
\label{eq:general-localized-ind} \tag*{$\ii{P}\left( n\right):$}
\big\|   \big\|  \Pi_{I_0}^{F, G, \tilde E}(\vec f, \vec g)   \big\|_{L^{R^n}}   \big\|_{s} \lesssim \big( \sssize_{I_0} \one_F \big)^{\frac{1}{p'}-\epsilon} \big( \sssize_{I_0} \one_G \big)^{\frac{1}{q'}-\epsilon} \big( \sssize_{I_0} \one_{\tilde E} \big)^{\frac{1}{s}-\epsilon} 
\big\| \| \vec f  \|_{L^{R^n_1}} \cdot \ci_{I_0} \big\|_p \big\| \big\| \vec g  \big\|_{L^{R^n_2}} \cdot \ci_{I_0} \big\|_q,
\end{equation}}
for any $p, q, s$ such that $\ds \frac{1}{p}+\frac{1}{q}=\frac{1}{s}$, $1<p, q \leq \infty$ and $\dfrac{1}{2}<s < \infty$.
\end{theorem}

\begin{remark}
\label{rmk:P*}
A related estimate that proves useful for establishing Theorem \ref{thm:genral_case} is the following:
\begin{equation}
\label{eq:general-localized-ind-sizes} \tag*{$\ii{P}^*\left( n\right):$}
\big\|   \big\|  \Pi_{I_0}^{F, G, \tilde E}(\vec f, \vec g)   \big\|_{L^{R^n}}   \big\|_{s} \lesssim \big( \sssize_{I_0} \one_F \big)^{1-\epsilon} \big( \sssize_{I_0} \one_G \big)^{1-\epsilon} \big( \sssize_{I_0} \one_{\tilde E} \big)^{\frac{1}{s}-\epsilon} \lft I_0 \rg^{\frac{1}{s}},
\end{equation}
whenever $\vec f, \vec g$ are so that $\big\| \vec f (x)\big\|_{L^{R^n_1}} \leq \one_F(x),   \big\|  \vec g(x)  \big\|_{L^{R^n_2}} \leq \one_G(x)$.
\end{remark}

\begin{proof}[Proof of Theorem \ref{thm:genral_case}]
Since $\ii P (0)$ represents the scalar case, and $\ii{P}(1)$ was proved in Theorem \ref{thm:localized-one-quasiBanach-Paraprod}, it will be enough to show $\ii{P}(n-1) \Rightarrow \ii{P}(n)$. Although Theorem \ref{thm:localized-one-quasiBanach-Paraprod} deals with discrete $\ell^r$ spaces, the result easily extends to general $L^r$ spaces. From Proposition \ref{prop:reorder}, we know that $\ds \| \cdot \|_{L^{R^n}}^{r^{j_0}}$ is subadditive, where $\ds r^{j_0}=\min_{1 \leq j \leq n} r^j$. Our iterative argument will depend on the ratio $\ds r^{j_0}/{r^1}$; more exactly, we will treat the cases $r^1 = r^{j_0}$ and  $r^1 >r^{j_0}$ differently. In both situations, we are re-organizing the (quasi-)norms, in order to obtain subadditivity.

We denote the $\left( n-1 \right)$-tuple $\tilde{R}^{n-1}:=\left( r^2, \ldots r^n\right)$, obtaining in this way $\ds R^n=\left( r^1, \tilde R^{n-1}  \right)$ and $\ds \big\|  \cdot \big\|_{L^{R^n}}=\big\|  \| \cdot \|_{L^{\tilde R^{n-1}}}  \big\|_{r^1}$.

\vspace{.2 cm}
\subsection*{I. The case $\ds r^1=\min_{1 \leq j \leq n} r^j <1$}~\\

In this case, $\ds \|  \cdot \|_{L^{R^n}}^{r^1}$ is subadditive, and this will play an important role. We will be proving, for $\left( s_1, s_2 , \tilde s\right)$ in a neighborhood of $\left( p, q, s \right)$, that
\begin{align}
\label{eq:weak-type-mul-r-1}
&\big\|  \big\|  \Pi_{I_0}^{F, G, \tilde E}(\vec f, \vec g)   \big\|_{L^{R^n}} \big\|_{\tilde{s}, \infty}\\
& \qquad \lesssim \big( \sssize_{I_0} \one_F \big)^{\frac{1}{s_1'}-\epsilon} \big( \sssize_{I_0} \one_G \big)^{\frac{1}{s_2'}-\epsilon} \big( \sssize_{I_0} \one_{\tilde E} \big)^{\frac{1}{s}-\epsilon} \big\|  \one_{E_1} \cdot \ci_{I_0}\big\|_1^{\frac{1}{s_1}} \cdot \big\|  \one_{E_2} \cdot \ci_{I_0}\big\|_1^{\frac{1}{s_2}}, \nonumber
\end{align}
whenever $\vec f, \vec g$ are so that $\ds \big\| \vec f (x)\big\|_{L^{R^n_1}} \leq \one_{E_1}(x)$ and $\ds \big\| \vec g (x)\big\|_{L^{R^n_2}} \leq \one_{E_2}(x)$, and $E_1, E_2$ are sets of finite measure.

Here again we need to treat two separate cases:
\begin{enumerate}[label=(\alph*), ref=(\alph*)]
\item \label{case:simple} $\tilde{s} \geq r^1$
 \item \label{case:complex} $\tilde{s} <r^1$.
\end{enumerate}

In the case \ref{case:simple}, it is easier to obtain the exponent $\frac{1}{s}-\epsilon$ for $\sssize_{I_0} \one_{\tilde E}$; the case \ref{case:complex} relies on case \ref{case:simple} and an intermediate step.

\subsection*{Case $I. (a): \tilde{s} \geq r^1$}~\\
We dualize the quasinorm $\| \cdot \|_{\tilde s, \infty}$ through $L^{r_1}$, as in Proposition \ref{prop:Lr-dualization}. Given $E_3$, a set of finite measure, we define the exceptional set
\begin{equation}
\label{eq:def-exceptional-set}
\tilde \Omega :=\Big \lbrace x: \ic{M}\left( \one_{E_1} \cdot \ci_{I_0} \right)(x) > C \cdot  \frac{\big\| \one_{E_1} \cdot \ci_{I_0} \big\|_1}{\lft E_3 \rg} \Big \rbrace \cup \Big \lbrace x: \ic{M}\left( \one_{E_2} \cdot \ci_{I_0} \right)(x) > C \cdot \frac{\big\| \one_{E_2} \cdot \ci_{I_0} \big\|_1}{\lft E_3 \rg} \Big \rbrace,
\end{equation}
and set $\tilde E_3:= E_3 \setminus \tilde \Omega$. Then we have
\begin{align*}
\big\|  \big\|  \Pi_{I_0}^{F, G, \tilde E}(\vec f, \vec g)   \big\|_{L^{R^n}} \big\|_{\tilde{s}, \infty}^{r_1} \cdot \lft E_3 \rg^{1-\frac{r^1}{\tilde s}} \sim \big\|  \big\|  \Pi_{I_0}^{F, G, \tilde E}(\vec f, \vec g) \big\|_{L^{R^n}}  \cdot \one_{\tilde E_3} \big\|_{r^1}^{r^1}.
\end{align*}

Now we unfold the Lebesgue norms on the RHS of the above expression:
\begin{align*}
& \int_{\rr R} \big\| \Pi_{I_0}^{F, G, \tilde E \cap \tilde E_3} \left( \vec f(x), \vec g(x)  \right) \big\|_{R^n}^{r^1} dx \lesssim \sum_{\bar n_1, \bar n_2, \bar n_3} \sum_{I \in \ii{I}^{\bar n_1, \bar n_2, \bar n_3}}  \int_{\rr R} \big\| \Pi_{I}^{F, G, \tilde E \cap \tilde E_3} \left( \vec f(x), \vec g(x)  \right) \big\|_{R^n}^{r^1} dx\\
&= \sum_{\bar n_1, \bar n_2, \bar n_3} \sum_{I \in \ii{I}^{\bar n_1, \bar n_2, \bar n_3}}  \int_{\rr R} \int_{\mathcal W_1} \big \| \Pi_{I}^{F, G, \tilde E \cap \tilde E_3} \left( \vec f_{w_1}(x), \vec g_{w_1}(x)  \right)  \big\|_{L^{\tilde R^{n-1}}}^{r^1} d w_1 dx \\
&=\sum_{\bar n_1, \bar n_2, \bar n_3} \sum_{I \in \ii{I}^{\bar n_1, \bar n_2, \bar n_3}}  \int_{\mathcal W_1} \int_{\rr R} \big \| \Pi_{I}^{F, G, \tilde E \cap \tilde E_3} \left( \vec f_{w_1}(x), \vec g_{w_1}(x)  \right)  \big\|_{L^{\tilde R^{n-1}}}^{r^1} dx dw_1.
\end{align*}

The collection of intervals $\ii{I}^{\bar n_1, \bar n_2, \bar n_3}$ is defined by a triple time, as in Section \ref{subsection-stopping_times}. We used here the subadditivity of $\|  \cdot \|_{R^n}^{r_1}$, and afterwards Fubini. This allows us to use the induction step in order to estimate $\ds \big\| \big \| \Pi_{I}^{F, G, \tilde E \cap \tilde E_3} \left( \vec f_{w_1}(x), \vec g_{w_1}(x)  \right)  \big\|_{L^{\tilde R^{n-1}}} \big\|_{L^{r^1}_x}$. First we note that $\vec f_{w_1}$ is supported on $E_1$ and $\vec g_{w_1}$ is supported on $E_2$, and hence $\ii P(n-1)$ implies
\begin{align*}
&\int_{\rr R} \big \| \Pi_{I}^{F, G, \tilde E \cap \tilde E_3} \left( \vec f_{w_1}(x), \vec g_{w_1}(x)  \right)  \big\|_{L^{\tilde R^{n-1}}}^{r^1} dx=\int_{\rr R} \big \| \Pi_{I}^{F\cap E_1, G \cap E_2, \tilde E \cap \tilde E_3} \left( \vec f_{w_1}(x), \vec g_{w_1}(x)  \right)  \big\|_{L^{\tilde R^{n-1}}}^{r^1} dx\\ &\lesssim \big( \sssize_{I} \one_{F \cap E_1} \big)^{\frac{r^1}{\left(r^1_1\right)'}-\epsilon} \cdot \big( \sssize_{I} \one_{G \cap E_2} \big)^{\frac{r^1}{\left(r^1_2\right)'}-\epsilon} \cdot \big( \sssize_{I} \one_{\tilde E \cap \tilde E_3} \big)^{1-\epsilon} \\
& \qquad \cdot  \big\|  \big\| \vec f_{w_1} \big\|_{\tilde R^{n-1}_1} \cdot \ci_I \big\|_{r_1^1}^{r^1} \big\|  \big\| \vec g_{w_1} \big\|_{\tilde R^{n-1}_2} \cdot \ci_I\big\|_{r_2^1}^{r^1}.
\end{align*}

Upon integrating in $w_1$, we use H\"older's inequality with conjugate Lebesgue exponents $\dfrac{r_1^1}{r^1}$ and $\dfrac{r_2^1}{r^1}$:
\begin{align*}
&\int_{\mathcal W_1} \int_{\rr R} \big \| \Pi_{I}^{F, G, \tilde E \cap \tilde E_3} \left( \vec f_{w_1}(x), \vec g_{w_1}(x)  \right)  \big\|_{L^{\tilde R^{n-1}}}^{r^1} dx dw_1 \\
&\lesssim \big( \sssize_{I} \one_{F \cap E_1} \big)^{\frac{r^1}{\left(r^1_1\right)'}-\epsilon}  \big( \sssize_{I} \one_{G \cap E_2} \big)^{\frac{r^1}{\left(r^1_2\right)'}-\epsilon} \cdot \big( \sssize_{I} \one_{\tilde E \cap \tilde E_3} \big)^{1-\epsilon} \\
& \qquad \cdot \big\| \big\|  \big\| \vec f_{w_1} \big\|_{\tilde R^{n-1}_1} \cdot \ci_I \big\|_{L_x^{r_1^1}}  \big\| _{L_{w_1}^{r_1^1}}^{r^1} \big\| \big\|  \big\| \vec g_{w_1} \big\|_{\tilde R^{n-1}_2} \cdot \ci_I \big\|_{L_x^{r_2^1}}  \big\| _{L_{w_1}^{r_2^1}}^{r^1}.
\end{align*}
Fubini allows us to switch the order of integration in $x$ and $w_1$, and thus the expression above is equal to
\begin{align*}
&\big( \sssize_{I} \one_{F \cap E_1} \big)^{\frac{r^1}{\left(r^1_1\right)'}-\epsilon} \big( \sssize_{I} \one_{G \cap E_2} \big)^{\frac{r^1}{\left(r^1_2\right)'}-\epsilon} \cdot \big( \sssize_{I} \one_{\tilde E \cap \tilde E_3} \big)^{1-\epsilon} \\
& \qquad \cdot \big\| \big\|  \big\| \vec f_{w_1} \big\|_{\tilde R^{n-1}_1} \big\|_{L_{w_1}^{r_1^1}}  \cdot \ci_I \big\| _{L_x^{r_1^1}}^{r^1} \big\| \big\|  \big\| \vec g_{w_1} \big\|_{\tilde R^{n-1}_2} \big\|_{L_{w_1}^{r_2^1}}  \cdot \ci_I \big\| _{L_x^{r_2^1}}^{r^1}.
\end{align*}
We recall that the functions $\vec f$ and $\vec g$ in fact satisfy
\[
\big\| \vec f (x)\big\|_{L^{R^n_1}} \leq \one_{F \cap E_1}(x) \quad \text{and} \quad \big\| \vec g (x)\big\|_{L^{R^n_2}} \leq \one_{G \cap E_2}(x)
\]
because the operator involves the projections onto the sets $F$ and $G$ respectively. So we have proved that 
{\fontsize{10}{10}
\begin{equation}
\label{eq:loc-rec-r-1}
\big\| \big\| \big \| \Pi_{I}^{F, G, \tilde E \cap \tilde E_3} \big( \vec f_{w_1}(x), \vec g_{w_1}(x)  \big)  \big\|_{L^{\tilde R^{n-1}}}^{r^1} \|_{L_x^1}\|_{L_{w_1}^1} \lesssim \big( \sssize_{I} \one_{F \cap E_1} \big)^{r^1-\epsilon} \big( \sssize_{I} \one_{G \cap E_2} \big)^{r^1-\epsilon} \cdot \big( \sssize_{I} \one_{\tilde E \cap \tilde E_3} \big)^{1-\epsilon} \cdot \lft  I \rg.
\end{equation}}
We note that this estimate is similar to $\ii P^*(n)$ from Remark \ref{rmk:P*}. Turning back to the initial estimate, we have:
\begin{align*}
& \int_{\rr R} \big\| \Pi_{I_0}^{F, G, \tilde E \cap \tilde E_3} \left( \vec f_{w_1}(x), \vec g_{w_1}(x)  \right) \big\|_{R^n}^{r^1} dx \\
&\lesssim \sum_{\bar n_1, \bar n_2, \bar n_3} \sum_{I \in \ii I^{\bar n_1, \bar n_2, \bar n_3}} \big( \sssize_{I} \one_{F \cap E_1} \big)^{r^1-\epsilon} \big( \sssize_{I} \one_{G \cap E_2} \big)^{r^1-\epsilon} \cdot \big( \sssize_{I} \one_{\tilde E \cap \tilde E_3} \big)^{1-\epsilon} \cdot \lft  I \rg.
\end{align*}
With the purpose of recovering the norms of $\one_{E_1}$ and $\one_{E_2}$, we separate $\sssize \one_F \cdot \one_{E_1}$ as
\[
\big( \sssize_I \one_F \cdot \one_{E_1} \big)^{r^1-\epsilon} \lesssim \big( \sssize_{I_0} \one_F\big)^{\frac{r^1}{\left(s_1\right)'}-\epsilon} \cdot \big(\sssize_I \one_{E_1} \big)^{\frac{r^1}{s_1}},
\]
and similarly for $\sssize \one_G \cdot \one_{E_2}$ and $\sssize \one_{\tilde E} \cdot \one_{\tilde E_3}$.

Due to the stopping times, we have
\begin{align*}
& \int_{\rr R} \big\| \Pi_{I_0}^{F, G, \tilde E \cap \tilde E_3} \left( \vec f_{w_1}(x), \vec g_{w_1}(x)  \right) \big\|_{R^n}^{r^1} dx \\
&\lesssim \big( \sssize_{I_0} \one_F\big)^{\frac{r^1}{\left(s_1\right)'}-\epsilon} \cdot \big( \sssize_{I_0} \one_G \big)^{\frac{r^1}{\left(s_2\right)'}-\epsilon} \cdot \big( \sssize_{I_0} \one_{\tilde E}\big)^{\frac{r^1}{s}-\epsilon} \\
&\qquad \cdot \sum_{\bar n_1, \bar n_2, \bar n_3} \sum_{I \in \ii I^{\bar n_1, \bar n_2, \bar n_3}} 2^{-\frac{n_1 r^1}{s_1}} 2^{-\frac{n_2 r^1}{s_2}} 2^{-n_3 \left( 1 -\frac{r^1}{\tilde s} +\epsilon  \right)} \lft I \rg
\end{align*}

From Section \ref{subsection-stopping_times}, we know that whenever we are performing the stopping times we have
\[
\sum_{I \in \ii I^{\bar n_1, \bar n_2, \bar n_3}} \lft I \rg \lesssim \left( 2^{n_1} \big\| \one_{E_1} \cdot  \ci_{I_0} \big\|_1 \right)^{\gamma_1} \cdot \left( 2^{n_2} \big\| \one_{E_2} \cdot  \ci_{I_0} \big\|_1 \right)^{\gamma_2} \cdot \left( 2^{n_3} \lft E_3 \rg \right)^{\gamma_3},
\]
where $0 \leq \gamma_j \leq 1$, $\gamma_1+\gamma_2+\gamma_3=1$.

In the end, we obtain 
\begin{align*}
\big\| \big\|  \Pi_{I_0}^{F, G, \tilde E}(\vec f, \vec g) \big\|_{L^{R^n}}  \cdot \one_{\tilde E_3} \big\|_{r^1}^{r^1} & \lesssim 
\big( \sssize_{I_0} \one_F\big)^{\frac{r^1}{\left(s_1\right)'}-\epsilon} \cdot \big( \sssize_{I_0} \one_G \big)^{\frac{r^1}{\left(s_2\right)'}-\epsilon} \cdot \big( \sssize_{I_0} \one_{\tilde E}\big)^{\frac{r^1}{s}-\epsilon} \\
&\qquad   \cdot  \big\| \one_{E_1} \cdot  \ci_{I_0} \big\|_1^\frac{r^1}{s_1} \cdot  \big\| \one_{E_1} \cdot  \ci_{I_0} \big\|_1^\frac{r^1}{s_1} \lft E_3 \rg^{1-\frac{r^1}{\tilde s}},
\end{align*}
which, together with \eqref{eq:weak-type-mul-r-1} and interpolation, concludes the induction statement $\ii P(n-1) \Rightarrow \ii P(n)$ for the case \ref{case:simple}.

\vspace{.2cm}
\subsection*{Case $I. (b): \tilde{s} < r^1$}~\\
In this case, we will dualize through an $L^\tau$ space, for some $\tau \leq \tilde s < r^1$. Given $E_3$, we define $\tilde E_3$ in the same way as before, and hence we have
\[
\big\|  \big\|  \Pi_{I_0}^{F, G, \tilde E}(\vec f, \vec g)   \big\|_{L^{R^n}}  \big\|_{\tilde s, \infty} \sim \big\|  \big\|  \Pi_{I_0}^{F, G, \tilde E}(\vec f, \vec g)   \big\|_{L^{R^n}} \cdot \one_{\tilde E_3} \big\|_{\tau} \cdot \lft E_3 \rg^{\frac{1}{\tilde s}-\frac{1}{\tau}}.
\]

Since $\ds \tau<r^1=\min_{1 \leq j \leq n}r^j$, following Proposition \ref{prop:reorder}, $\ds \big\| \|\cdot \|_{L^{R^n}} \big\|_\tau^\tau$ is subadditive. Hence we have
\begin{align*}
&\big\|  \big\|  \Pi_{I_0}^{F, G, \tilde E}(\vec f, \vec g)   \big\|_{L^{R^n}} \cdot \one_{\tilde E_3} \big\|_{\tau}^\tau \leq \sum_{\bar n_1, \bar n_2, \bar n_3} \sum_{I \in \ii I^{\bar n_1, \bar n_2, \bar n_3}} \big\|  \big\|  \Pi_{I}^{F, G, \tilde E}(\vec f, \vec g)   \big\|_{L^{R^n}} \cdot \one_{\tilde E_3} \big\|_{\tau}^\tau.
\end{align*}

We denote by $\tau_r$ the positive exponent for which $\ds \frac{1}{r^1}+\frac{1}{\tau_r}=\frac{1}{\tau}$. This, together with the previous case $s \geq r^1$ will allow us to estimate the term on the RHS of the above expression. Through an intermediate step consisting of a decomposition similar to that appearing in Lemma \ref{lemma:sub-unit-local-est-paraprod}, we eventually obtain
\begin{align*}
& \big\|  \big\|  \Pi_{I}^{F, G, \tilde E \cap \tilde E_3}(\vec f, \vec g)   \big\|_{L^{R^n}} \big\|_{\tau} \lesssim \big\|  \big\|  \Pi_{I}^{F, G, \tilde E \cap \tilde E_3}(\vec f, \vec g)   \big\|_{L^{R^n}} \big\|_{r^1} \cdot \big\| \one_{\tilde E} \cdot \one_{\tilde E_3} \cdot \ci_I  \big\|_{\tau_r}\\
& \lesssim \big(\sssize_{I} \one_F \cdot \one_{E_1}  \big)^{1-\epsilon} \big(\sssize_{I} \one_G \cdot \one_{E_2}  \big)^{1-\epsilon} \big(\sssize_{I} \one_{\tilde E} \cdot \one_{\tilde E_3}  \big)^{\frac{1}{r^1}+\frac{1}{\tau_r}-\epsilon} \lft I \rg^{\frac{1}{r^1}+\frac{1}{\tau_r}} .
\end{align*}

Here in fact we used $\ii P^*(n)$ form Remark \ref{rmk:P*}, applied to the operator $\ds \Pi_I^{F \cap E_1, G \cap E_2, \tilde E \cap \tilde E_3}$. We then obtain a familiar estimate, that will allow us to recover \eqref{eq:weak-type-mul-r-1}:
{\fontsize{9}{10}
\begin{align*}
\big \|   \big \| \Pi_{I_0}^{F, G, \tilde E}(\vec f, \vec g)\big\|_{L^{R^n}} \cdot \one_{\tilde E_3}\big \|_\tau^\tau \lesssim \sum_{\bar n_1, \bar n_2, \bar n_3} \sum_{I \in \ii I^{\bar n_1, \bar n_2, \bar n_3}} \big(\sssize_{I} \one_F \cdot \one_{E_1}  \big)^{\tau-\epsilon} \big(\sssize_{I} \one_G \cdot \one_{E_2}  \big)^{\tau-\epsilon} \big(\sssize_{I} \one_{\tilde E} \cdot \one_{\tilde E_3}  \big)^{1-\epsilon} \lft I \rg.
\end{align*}}

\subsection*{II. The case $\ds r^1> r^{j^0}=\min_{1 \leq j \leq n} r^j <1$}~\\
Let $\sigma:= r^1/r^{j^0} >1$, and $\tilde R^{n-1}=\left(r^2, \ldots, r^n \right)$ so that $R^n=\left( r^1, \tilde R ^{n-1}  \right)$. We note that in this case $\| \cdot \|_{L^{R^n}}^{r^{j_0}}$ and $\| \cdot \|_{L^{\tilde R^{n-1}}}^{r^{j_0}}$ are both subadditive.

What we aim for is an inequality similar to \eqref{eq:weak-type-mul-r-1}, for any admissible tuple $(s_1, s_2, \tilde s)$ in a neighborhood of $( p, q, s)$. For this, we need to rearrange the quasinorms:
\begin{align*}
&\big\| \big\| \Pi_{I_0}^{F, G, \tilde E}(\vec f, \vec g) \big\|_{L^{R^n}} \big\|_{\tilde s, \infty}=
\big\| \big(\big\| \big\| \Pi_{I_0}^{F, G, \tilde E}(\vec f_{w_1}, \vec g_{w_1}) \big\|_{L^{\tilde R^{n-1}}}^{r^{j_0}} \big\|_{L_{w_1}^{r^1/r^{j_0}}} \big)^{\frac{1}{r^{j_0}}}\big\|_{\tilde s, \infty}\\
&=\big\| \big\| \big\| \Pi_{I_0}^{F, G, \tilde E}(\vec f_{w_1}, \vec g_{w_1}) \big\|_{L^{\tilde R^{n-1}}}^{r^{j_0}} \big\|_{L_{w_1}^{\sigma}} \big\|_{\frac{\tilde s}{r^{j_0}}, \infty}^{\frac{1}{r^{j_0}}}.
\end{align*}
\subsection*{Case $II. (a): s \geq r^{j_0}$}
Now we will want to ``dualize" the $L_x^{\frac{\tilde s}{r^{j_0}}, \infty}$ norm; we regard it as an $L^\sigma_{w_1}$-valued function, and we will use the fact that $\sigma>1$ and hence $L^{\sigma}_{w_1}$ is a Banach space. For a set $E_3$ of finite measure, we set $E_3'=E_3 \setminus \tilde \Omega$, where $\tilde \Omega$ is defined in \eqref{eq:def-exceptional-set}, and we note that 
{\fontsize{10}{10}
\begin{align*}
\big\| \big\| \big\| \Pi_{I_0}^{F, G, \tilde E}(\vec f_{w_1}, \vec g_{w_1}) \big\|_{L^{\tilde R^{n-1}}}^{r^{j_0}} \big\|_{L_{w_1}^{\sigma}} \big\|_{\frac{\tilde s}{r^{j_0}}, \infty}
\sim \int_{\rr R} \int_{\ic W_1} \big\| \Pi_{I_0}^{F, G, \tilde E}(\vec f_{w_1}, \vec g_{w_1}) \big\|_{L^{\tilde R^{n-1}}}^{r^{j_0}} h(w_1, x) d w_1 dx \cdot \lft E_3 \rg^{\frac{r^{j_0}}{\tilde s}-1},
\end{align*}}
where $h(w_1, x)$ is so that $\ds \big\| h(\cdot, x) \big\|_{L^{\sigma '}_{w_1}} \lesssim \one_{E_3'}(x)$. That is, we are using a Banach-valued version of Proposition \ref{prop:Lr-dualization}, which can be found in \cite{vv_BHT}.

Now we can finally make use of the subadditivity of $\| \cdot \|_{L^{\tilde R^{n-1}}}^{r^{j_0}}$. With the collections of intervals $\ii I^{\bar n_1, \bar n_2, \bar n_3}$ as in Section \ref{subsection-stopping_times}, we have
\begin{align*}
&\int_{\rr R} \int_{\ic W_1} \big\| \Pi_{I_0}^{F, G, \tilde E \cap  E_3'}(\vec f_{w_1}, \vec g_{w_1}) \big\|_{L^{\tilde R^{n-1}}}^{r^{j_0}} h(w_1, x) d w_1 dx \\
& \lesssim \sum_{\bar n_1, \bar n_2, \bar n_3} \sum_{I \in \ii I ^{\bar n_1, \bar n_2, \bar n_3}}\int_{\rr R} \int_{\ic W_1} \big\| \Pi_{I}^{F, G, \tilde E \cap E_3'}(\vec f_{w_1}, \vec g_{w_1}) \big\|_{L^{\tilde R^{n-1}}}^{r^{j_0}} h(w_1, x) d w_1 dx \\
&=\sum_{\bar n_1, \bar n_2, \bar n_3} \sum_{I \in \ii I ^{\bar n_1, \bar n_2, \bar n_3}}\int_{\ic W_1} \int_{\rr R} \big\| \Pi_{I}^{F, G, \tilde E \cap E_3'}(\vec f_{w_1}, \vec g_{w_1}) \big\|_{L^{\tilde R^{n-1}}}^{r^{j_0}} h(w_1, x) dx d w_1 \\
&\lesssim \sum_{\bar n_1, \bar n_2, \bar n_3} \sum_{I \in \ii I ^{\bar n_1, \bar n_2, \bar n_3}} \big\| \big\|  \big\| \Pi_{I}^{F, G, \tilde E \cap E_3'}(\vec f_{w_1}, \vec g_{w_1}) \big\|_{L^{\tilde R^{n-1}}}^{r^{j_0}}  \big\|_{L^\sigma_x} \big\|_{L^\sigma_{w_1}} \cdot \big\| \big\| h(w_1, \cdot) \cdot \one_{\tilde E \cap E_3'} \cdot \ci_I\big\|_{L^{\sigma'}_x} \big\|_{L^{\sigma '}_{w_1}}.
\end{align*}

Above, we used Fubini and H\"older's inequality several times. The decaying factor $\ci_I$ appearing in the last line is motivated by the same line of ideas as those appearing in Lemma \ref{lemma:sub-unit-local-est-paraprod}. Following the induction step $\ii P(n-1)$, we have 
\begin{align*}
&\big\|  \big\| \Pi_{I}^{F, G, \tilde E \cap E_3'}(\vec f_{w_1}, \vec g_{w_1}) \big\|_{L^{\tilde R^{n-1}}}^{r^{j_0}}  \big\|_{L^\sigma_x} = \big\| \big\| \Pi_{I}^{F\cap E_1, G\cap E_2, \tilde E \cap E_3'}(\vec f_{w_1}, \vec g_{w_1}) \big\|_{L^{\tilde R^{n-1}}}  \big\|_{L_x^{r^1}}^{r_{j_0}}\\
&\lesssim \big(\sssize_I \one_{F} \cdot \one_{E_1} \big)^{\frac{r^{j_0}}{\left(r_1^1\right)'}-\epsilon} \cdot \big(\sssize_I \one_{G} \cdot \one_{E_2} \big)^{\frac{r^{j_0}}{\left(r_2^1\right)'}-\epsilon}\cdot \big(\sssize_I \one_{\tilde E} \cdot \one_{E'_3} \big)^{\frac{r^{j_0}}{r^1}-\epsilon}\\
&\qquad \cdot \big\| \big\|\vec f_{w_1}\big\|_{L^{\tilde R_1^{n-1}}} \cdot \ci_{I} \big\|_{L^{r^1_1}_x}^{r^{j_0}} \cdot \big\| \big\|\vec g_{w_1}\big\|_{L^{\tilde R_2^{n-1}}} \cdot \ci_{I} \big\|_{L^{r^1_2}_x}^{r^{j_0}}.
\end{align*}

We integrate the last line in $w_1$ and we use H\"older's inequality for the Lebesgue exponents $\left(\frac{r_1^1}{r^{j_0}}, \frac{r_2^1}{r^{j_0}}, \frac{r^1}{r^{j_0} } \right)$, since $\ds \frac{1}{\sigma}=\frac{r^{j_0}}{r_1^1}+\frac{r^{j_0}}{r_2^1}$. One of the terms we obtain in this way is $\ds \big\| \big\| \big\|\vec f_{w_1}\big\|_{L^{\tilde R_1^{n-1}}} \cdot \ci_{I} \big\|_{L^{r^1_1}_x}^{r^{j_0}} \big\|_{L_{w_1}^{\frac{r^1_1}{r^{j_0}}}}$, which can be rewritten, using Fubini, as 
\[
 \big\| \big\| \big\|\vec f_{w_1}\big\|_{L^{\tilde R_1^{n-1}}} \cdot \ci_{I} \big\|_{L^{r^1_1}_x}^{r^{j_0}} \big\|_{L_{w_1}^{\frac{r^1_1}{r^{j_0}}}}=\big( \int_{\ic W_1} \int_{\rr R} \big| \big\|\vec f_{w_1}\big\|_{L^{\tilde R_1^{n-1}}}  \cdot \ci_I \big|^{r_1^1} dx d w_1\big)^{\frac{r^{j^0}}{r_1^1}} =\big\|  \big\|  \vec f \big\|_{L^{R^n_1}} \cdot \ci_{I} \big\|^{r^{j_0}}_{L_x^{r_1^1}}.
\]
The function $\vec f$ is supported on $F \cap E_1$ in the $x$ coordinate, and the above expression is bounded by
\[
 \big\| \one_{F \cap E_1} \cdot \ci_I \big\|_1^{\frac{r^{j_0}}{r_1^1}} \lesssim \big( \sssize_I \one_F \cdot \one_{E_1}\big)^{\frac{r^{j_0}}{r_1^1}} \cdot \lft I \rg^{\frac{r^{j_0}}{r_1^1}}.
 \] 

Similarly, the term corresponding to $\vec g$ will be bounded above by 
\[
 \big\| \one_{G \cap E_2} \cdot \ci_I \big\|_1^{\frac{r^{j_0}}{r_2^1}} \lesssim \big( \sssize_I \one_G \cdot \one_{E_2}\big)^{\frac{r^{j_0}}{r_2^1}} \cdot \lft I \rg^{\frac{r^{j_0}}{r_2^1}}.
 \] 
For the function $h(w_1, x)$, we have the estimate 
\[
\big\| \big\| h(w_1, \cdot) \cdot \one_{\tilde E \cap E_3'} \cdot \ci_I\big\|_{L^{\sigma'}_x} \big\|_{L^{\sigma '}_{w_1}} \lesssim \big\|  \one_{E_3'} \cdot \one_{\tilde E} \cdot \ci_I \big\|_{L^{\sigma'}_x} \lesssim  \big( \sssize_I \one_{\tilde E} \cdot \one_{E_3'}\big)^{\frac{1}{\sigma'}} \cdot \lft I \rg^{\frac{1}{\sigma'}}.
\]
Returning to the stopping time, our initial estimate becomes
\begin{align*}
&\int_{\rr R} \int_{\ic W_1} \big\| \Pi_{I_0}^{F, G, \tilde E}(\vec f_{w_1}, \vec g_{w_1}) \big\|_{L^{\tilde R^{n-1}}}^{r^{j_0}} h(w_1, x) d w_1 dx \\
&\lesssim \sum_{\bar n_1, \bar n_2, \bar n_3} \sum_{I \in \ii I^{\bar n_1, \bar n_2, \bar n_3}} \big( \sssize_I \one_F \cdot \one_{E_1}\big)^{r^{j_0}-\epsilon} \cdot \big( \sssize_I \one_G \cdot \one_{E_2}\big)^{r^{j_0}-\epsilon} \cdot \big( \sssize_I \one_{\tilde E} \cdot \one_{E_3'}\big)^{1-\epsilon} \cdot \lft I \rg \\
&\lesssim \big( \sssize_{I_0} \one_F\big)^{\frac{r^{j_0}}{s_1'}-\epsilon} \cdot \big( \sssize_{I_0} \one_G \big)^{\frac{r^{j_0}}{s_2'}-\epsilon} \cdot \big( \sssize_{I_0} \one_{\tilde E}\big)^{\frac{r^{j_0}}{\tilde s}-\epsilon} \\
&\qquad \cdot   \sum_{\bar n_1, \bar n_2, \bar n_3} 2^{-\frac{n_1 r^{j_0}}{s_1}} 2^{-\frac{n_2 r^{j_0}}{s_2}} 2^{-n_3\left( 1-\frac{r^{j_0}}{\tilde s} +\epsilon \right)} \left(  2^{n_1} \big\| \one_{E_1} \cdot \ci_I  \big\|_1\right)^{\gamma_1} \left(  2^{n_2} \big\| \one_{E_2} \cdot \ci_I  \big\|_1\right)^{\gamma_2} \lft E_3 \rg^{\gamma_3},
\end{align*}
where, as usual, $\gamma_1+\gamma_2+\gamma_3=1$. The series will converge if $s \geq r^{j_0}$, and we obtain \eqref{eq:weak-type-mul-r-1}, with both terms in the inequality raised to the $r^{j_0}$ power.

\subsection*{Case $II. (b): s< r^{j_0}$}

Here we will make use of the previous result from $II.(b)$ corresponding to $s \geq r^{j_0}$, after rewriting the localized paraproduct. First, we want to use the subadditivity of $\|  \cdot \|_{L^{R^n}}^{r^{j_0}}$, and hence we write
\[
\big\|  \big\|  \Pi_{I_0}^{F, G, \tilde E} (\vec f, \vec g) \big\|_{L^{R^n}}  \big\|_{\tilde s, \infty}= \big\|  \big\|  \Pi_{I_0}^{F, G, \tilde E} (\vec f, \vec g) \big\|_{L^{R^n}}^{r^{j_0}}  \big\|_{\frac{\tilde s}{r^{j^0}}, \infty}^{\frac{1}{r^{j_0}}}.
\]  

The $L^{\frac{\tilde s}{r^{j^0}}, \infty}$ quasinorm will be dualized as in Proposition \ref{prop:Lr-dualization}, through $L^{\frac{\tau}{r^{j_0}}}$, for some $\tau \leq s$. In particular, $\frac{\tau}{r^{j_0}} <1$, hence $\ds \| \cdot  \|_{\frac{\tau}{r^{j_0}}}^{\frac{\tau}{r^{j_0}}}$ is subadditive.

We have $E_3$ a set of finite measure, and $\tilde E_3$ a major subset constructed as before, for which 
\[
\big\|  \big\|  \Pi_{I_0}^{F, G, \tilde E} (\vec f, \vec g) \big\|_{L^{R^n}}^{r^{j_0}}  \big\|_{\frac{\tilde s}{r^{j^0}}, \infty} \sim \big\|  \big\|  \Pi_{I_0}^{F, G, \tilde E} (\vec f, \vec g) \big\|_{L^{R^n}}^{r^{j_0}}  \cdot \one_{\tilde E_3} \big\|_{\frac{\tau}{r^{j^0}}} \lft E_3 \rg^{\frac{r^{j_0}}{\tilde s}-\frac{r^{j_0}}{\tau}}.
\]

Employing the subadditivity of $\|  \cdot \|_{L^{R^n}}^{r^{j_0}}$ and $\ds \| \cdot  \|_{\frac{\tau}{r^{j_0}}}^{\frac{\tau}{r^{j_0}}}$, we have, for the collections of intervals $\ii I^{\bar n_1, \bar n_2, \bar n_3}$ as in Section \ref{subsection-stopping_times}
\begin{align}
\label{eq:ineq-smthg}
&\big\|  \big\|  \Pi_{I_0}^{F,G, \tilde E} (\vec f, \vec g) \big\|_{L^{R^n}}^{r^{j_0}}  \cdot \one_{\tilde E_3} \big\|_{\frac{\tau}{r^{j^0}}}^{\frac{\tau}{r^{j^0}}} \\
&\lesssim \sum_{\bar n_1, \bar n_2, \bar n_3} \sum_{I \in \ii I ^{\bar n_1, \bar n_2, \bar n_3}} \big\|  \big\|  \Pi_{I}^{F, G, \tilde E} (\vec f, \vec g) \big\|_{L^{R^n}}^{r^{j_0}}  \cdot \one_{\tilde E_3} \big\|_{\frac{\tau}{r^{j^0}}}^{\frac{\tau}{r^{j^0}}}=\sum_{\bar n_1, \bar n_2, \bar n_3} \sum_{I \in \ii I ^{\bar n_1, \bar n_2, \bar n_3}} \big\|  \big\|  \Pi_{I}^{F, G, \tilde E \cap \tilde E_3} (\vec f, \vec g) \big\|_{L^{R^n}} \big\|_{L^\tau}^{\tau}. \nonumber
\end{align}

Since $\tau \leq s< r^{j_0}$, there exists $\tau_r>0$ so that $\ds \frac{1}{r^{j_0}}+\frac{1}{\tau_r}=\frac{1}{\tau}$. An analysis similar to the one in Lemma \ref{lemma:sub-unit-local-est-paraprod}, together with the case $II. (a)$ of the present theorem, imply that 
\begin{align*}
& \big\|  \big\|  \Pi_{I}^{F, G, \tilde E \cap \tilde E_3} (\vec f, \vec g) \big\|_{L^{R^n}} \big\|_{L^\tau} \lesssim 
\big\|  \big\|  \Pi_{I}^{F, G, \tilde E \cap \tilde E_3} (\vec f, \vec g) \big\|_{L^{R^n}} \big\|_{L^{r^{j_0}}} \cdot \big\| \one_{\tilde E \cap \tilde E_3} \cdot \ci_I \big\|_{\tau_r}\\
& \lesssim \big( \sssize_I \one_F \cdot \one_{E_1} \big)^{1 -\epsilon}  \big( \sssize_I \one_G \cdot \one_{E_2} \big)^{1 -\epsilon} \big( \sssize_I \one_{\tilde E} \cdot \one_{\tilde E_3} \big)^{\frac{1}{r^{j_0}} -\epsilon} \cdot  \lft I \rg^{\frac{1}{r^{j_0}}} \cdot \big( \ssize_I \one_{\tilde E} \cdot \one_{\tilde E_3} \big)^{\frac{1}{\tau_r}} \cdot \lft I \rg^{\frac{1}{\tau_r}} \\
&=\big( \sssize_I \one_F \cdot \one_{E_1} \big)^{1 -\epsilon}  \big( \sssize_I \one_G \cdot \one_{E_2} \big)^{1 -\epsilon} \big( \sssize_I \one_{\tilde E} \cdot \one_{\tilde E_3} \big)^{\frac{1}{\tau} -\epsilon} \cdot  \lft I \rg^{\frac{1}{\tau}}.
\end{align*}

Returning to inequality \eqref{eq:ineq-smthg}, we have 
\begin{align*}
&\big\|  \big\|  \Pi_{I_0}^{F,G, \tilde E} (\vec f, \vec g) \big\|_{L^{R^n}}^{r^{j_0}}  \cdot \one_{\tilde E_3} \big\|_{\frac{\tau}{r^{j^0}}}^{\frac{\tau}{r^{j^0}}} \\
&\lesssim \sum_{\bar n_1, \bar n_2, \bar n_3} \sum_{I \in \ii I ^{\bar n_1, \bar n_2, \bar n_3}} \big( \sssize_I \one_F \cdot \one_{E_1} \big)^{\tau -\epsilon}  \big( \sssize_I \one_G \cdot \one_{E_2} \big)^{\tau -\epsilon} \big( \sssize_I \one_{\tilde E} \cdot \one_{\tilde E_3} \big)^{1 -\epsilon} \cdot  \lft I \rg \\
& \lesssim \big( \sssize_{I_0} \one_F \big)^{\frac{\tau}{p'} -\epsilon}  \big( \sssize_{I_0} \one_G \big)^{\frac{\tau}{q'} -\epsilon} \big( \sssize_{I_0} \one_{\tilde E}\big)^{\frac{\tau}{s} -\epsilon} \\
&\qquad \cdot  \sum_{\bar n_1, \bar n_2, \bar n_3} 2^{-n_1 \frac{\tau}{s_1}} 2^{-n_2 \frac{\tau}{s_2}} 2^{-n_3 \left( 1-\frac{\tau}{\tilde s} +\epsilon \right)} \left( 2^{n_1} \big\| \one_{E_1} \cdot \ci_{I_0} \big\|_1 \right)^{\gamma_1} \left( 2^{n_2} \big\| \one_{E_2} \cdot \ci_{I_0} \big\|_1 \right)^{\gamma_2} \lft E_3 \rg^{\gamma_3}.
\end{align*}

The series above converge because we are under the assumption that $\tau \leq s< r^{j^0}$, and eventually we obtain inequality \eqref{eq:weak-type-mul-r-1}. This ends the proof of Theorem \ref{thm:genral_case}.
\end{proof}

\section{Similar results for $BHT$}
\label{sec:BHT}
The bilinear Hilbert transform, $BHT$ in short, is a bilinear operator whose Fourier multiplier is singular along a line. Its study reduces to that of the model operator
\[
BHT_{\rr P}(f, g)(x):=\sum_{P \in \rr P} \frac{1}{\lft I_P \rg^{1/2}} \langle f, \phi_P^1 \rangle  \langle g, \phi_P^2 \rangle \phi_P^3(x).
\]

Instead of families $\phi_I$ indexed after a collection of intervals (the paraproduct case), we have as index set $\rr P$, a collection of tritiles. A \emph{tile} is a product $I \times \omega$ of an interval $I$ in space and an interval $\omega$ in frequency. A tritile $P$ is a set of three tiles sharing the spacial interval:
\[
P=\left( P_1, P_2, P_3 \right), \qquad P_j=I_P \times \omega_{P_j}, \qquad |I_P| \cdot |\omega_{P_j}|\sim 1.
\]
The functions $\phi_P^j$ are called ``wave packets" associated to the tritiles: $\widehat \phi_P^j$ is supported inside the frequency interval $.9 \omega_{P_j}$, and is $L^2$ adapted to $I_P$, in the sense that 
\[
\vert \partial^{\alpha} \phi_{P}^j (x) \vert \leq C_{\alpha, M} |I_P|^{-1/2-\vert \alpha \vert} \ci_{I_P}^M(x).
\] 
The collection $\rr P$ of tritiles associated to the model operator $BHT_{\rr P}$ is of \emph{rank one}, reflecting the dimension of the singularity. That is, the tiles can be located anywhere in frequency, but there is only one degree of freedom.

For more properties of $BHT$, reduction to the model operator, as well as a self-contained proof, we refer the interested reader to \cite{multilinear_harmonic}. We recall a few results and definitions, that can be found in \cite{vv_BHT}.

\begin{definition}Traditionally, the size is defined to be a supremum over suitable trees of discretized square functions. Instead, we will use this term for expressions that bound the `classical sizes':
\[
\ssize_{\rr P}f :=\sup_{P \in \rr P}\frac{1}{|I_P|} \int_{\rr R} |f(x)| \ci_{I_P}^M dx.
\]
If $I_0$ is a fixed interval, then $\rr P(I_0)$ denotes the collection of tritiles in $\rr P$ whose spatial interval is contained inside $I_0$:
\[
\rr P(I_0):=\lbrace I \in \rr P: I_P \subseteq I_0  \rbrace.
\] 
In this case, we define a new size:
\[
\sssize_{\rr P\left( I_0 \right)}f :=\max\left( \sup_{P \in \rr P}\frac{1}{|I_P|} \int_{\rr R} |f(x)| \ci_{I_P}^M dx, \frac{1}{|I_0|} \int_{\rr R} |f(x)| \ci_{I_0}^M dx   \right).
\]
\end{definition}

In our approach from \cite{vv_BHT}, a very important role is played by localized results. Here $\Lambda_{BHT; \rr P}$ is the trilinear form associated to the model operator $BHT_{\rr P}$, and in general, the collection $\rr P$ will be understood from the context (sometimes we write $BHT_{I_0}$ for $BHT_{\rr P \left( I_0 \right)}$).

\begin{proposition}[Lemma 5  from \cite{vv_BHT}]
\label{prop:loc-tril-form}
If $I_0$ is a fixed dyadic interval, and $\rr P$ is a rank 1 collection of intervals, then
{\fontsize{10}{10}
\[
\lft\Lambda_{BHT; \rr P \left(I_0\right)}(f, g, h) \rg \lesssim \big( \ssize_{\rr P \left( I_0 \right)} f   \big)^{\theta_1}  \big( \ssize_{\rr P \left( I_0 \right)} g   \big)^{\theta_2}  \big( \ssize_{\rr P \left( I_0 \right)} h   \big)^{\theta_3} \big\| f \cdot \ci_{I_0}  \big\|_2^{1-\theta_1}  \big\| g \cdot \ci_{I_0}  \big\|_2^{1-\theta_2}  \big\| h \cdot \ci_{I_0}  \big\|_2^{1-\theta_3},
\]}
for any $0 \leq \theta_1, \theta_2, \theta_3<1$ with $\theta_1+\theta_2+\theta_3=1$.
\end{proposition}

As an immediate consequence, we have:
\begin{corollary}
If $F, G$ and $\tilde E$ are sets of finite measure, and $f, g, h$ are so that $\lft f \rg \leq \one_F(x), \lft g\rg \leq \one_G(x)$ and $\lft h\rg \leq \one_{\tilde E}(x)$, then 
\[
\lft \Lambda_{\rr P \left( I_0 \right)} (f, g, h) \rg \lesssim \big( \sssize_{\rr P \left( I_0 \right)} \one_F  \big)^{\frac{1+\theta_1}{2}} \big( \sssize_{\rr P \left( I_0 \right)} \one_G  \big)^{\frac{1+\theta_2}{2}} \big( \sssize_{\rr P \left( I_0 \right)} \one_{\tilde E}  \big)^{\frac{1+\theta_3}{2}} \cdot \lft I_0 \rg.
\]
\end{corollary}

\begin{corollary}
\label{cor:L1-bht-local}
If $F, G$ and $\tilde E$ are sets of finite measure, and $f, g$ are so that $\lft f \rg \leq \one_F(x), \lft g\rg \leq \one_G(x)$, then 
\[
\big \| BHT_{I_0}(f, g) \cdot \one_{\tilde E} \big\|_1\lesssim \big( \sssize_{\rr P \left( I_0 \right)} \one_F  \big)^{\frac{1+\theta_1}{2}} \big( \sssize_{\rr P \left( I_0 \right)} \one_G  \big)^{\frac{1+\theta_2}{2}} \big( \sssize_{\rr P \left( I_0 \right)} \one_{\tilde E}  \big)^{\frac{1+\theta_3}{2}} \cdot \lft I_0 \rg.
\]
\end{corollary}

Our approach to proving vector-valued estimates for the bilinear operators involves localizations. Just like in the paraproduct case, we define
\[
BHT_{I_0}^{F, G, \tilde E}(f, g)(x):=BHT_{I_0}(f \cdot \one_F, g \cdot \one_G)(x) \cdot \one_{\tilde E}(x).
\]
 
For the trilinear form associated to this localized operator, we have proved in \cite{vv_BHT} the following inequality:

\begin{proposition}[Proposition 8 of \cite{vv_BHT}]
\label{prop:local-bht-banach}
If $1<r_1, r_2 \leq \infty$, and $1 \leq r <\infty$, then
\begin{align*}
\Big | \Lambda_{BHT_{I_0}^{F, G, \tilde E}}(f, g, h)  \Big | &\lesssim \big( \sssize_{I_0} \one_F \big)^{\frac{1+\theta_1}{2}-\frac{1}{r_1}-\epsilon}  \big( \sssize_{I_0} \one_G \big)^{\frac{1+\theta_2}{2}-\frac{1}{r_2}-\epsilon}  \big( \sssize_{I_0} \one_{\tilde E} \big)^{\frac{1+\theta_3}{2}-\frac{1}{r'}-\epsilon}  \\
&\cdot \big \| f \cdot \ci_{I_0} \big\|_{r_1} \big \| g \cdot \ci_{I_0} \big\|_{r_2} \big \| h \cdot \ci_{I_0} \big\|_{r'},
\end{align*}
provided the exponents appearing above are all strictly positive.
\end{proposition} 

The $BHT$ operator is bounded on $L^s$, for $\frac{2}{3} < s <\infty$, so it is natural to look for bounds within the same range for the localization $BHT_{I_0}^{F, G, \tilde E}$. Proposition \ref{prop:local-bht-banach} provides an answer for $1 \leq s < \infty$. In the quasi-Banach case, we have the following:
\begin{lemma}
\label{lemma:sub-unit-BHT}
If $\frac{2}{3}<\tau <1$, then 
\begin{equation}
\label{eq:sub-unit-BHT}
\big\|  BHT_{I_0}^{F, G}(f, g) \cdot \one_{\tilde E} \big\|_\tau \lesssim \big( \sssize_{I_0} \one_{F}  \big)^{\frac{1+\theta_1}{2}-\epsilon} \big( \sssize_{I_0} \one_{G}  \big)^{\frac{1+\theta_2}{2}-\epsilon} \big( \sssize_{I_0} \one_{\tilde E}  \big)^{\frac{1+\theta_3}{2}-\frac{1}{\tau'}-\epsilon} \lft I_0 \rg^{\frac{1}{\tau}},
\end{equation}
provided $\lft f(x) \rg \leq \one_F(x)$, $\lft g(x) \rg \leq \one_G(x)$.
\end{lemma}

This actually follows from the following:

\begin{proposition}
\label{lemma-subuni-BHT-functions}
For any $\frac{2}{3}<r <1$, and $1< r_1, r_2 \leq \infty$ so that $\frac{1}{r_1}+\frac{1}{r_2}=\frac{1}{r}$, we have
\begin{align}
\label{eq:gen-BHT-loc-fns}
\big \| BHT_{I_0}^{F, G, \tilde E}(f, g) \big\|_r &\lesssim \big( \sssize_{I_0} \one_F  \big)^{\frac{1+\theta_1}{2}-\frac{1}{r_1}-\epsilon} \big( \sssize_{I_0} \one_G  \big)^{\frac{1+\theta_1}{2}-\frac{1}{r_2}-\epsilon} \big( \sssize_{I_0} \one_{\tilde E}  \big)^{\frac{1+\theta_3}{2}-\frac{1}{r'}-\epsilon} \\&\cdot \quad \big\| f \cdot \ci_{I_0} \big\|_{r_1} \big\| g \cdot \ci_{I_0} \big\|_{r_2}, \nonumber
\end{align}
provided the exponents above are all positive. That is, there exist $0 \leq \theta_1, \theta_2, \theta_3 <1$ so that $ \theta_1+ \theta_2+ \theta_3 =1$ and
\begin{equation}
\label{eq:cond-vv-bht}
\tag*{$\ic C (r_1, r_2, r')$}
\frac{1}{r_1}<\frac{1+\theta_1}{2}, \quad \frac{1}{r_2}<\frac{1+\theta_2}{2}, \quad \frac{1}{r'}<\frac{1+\theta_3}{2}.
\end{equation}
\begin{proof}
It will be enough to prove
\begin{align}
\label{eq:enough-weak-type}
\big \| BHT_{\rr P \left( I_0 \right)}^{F, G, \tilde E}(f, g)  \big\|_{\tilde s, \infty}& \lesssim \big( \sssize_{I_0} \one_F \big)^{\frac{1+\theta_1}{2}-\frac{1}{s_1}-\epsilon}  \big( \sssize_{I_0} \one_G \big)^{\frac{1+\theta_2}{2}-\frac{1}{s_2}-\epsilon}  \big( \sssize_{I_0} \one_{\tilde E} \big)^{\frac{1+\theta_3}{2}-\frac{1}{\tilde{s}'}-\epsilon} \\
& \quad \cdot \big\| \one_{E_1} \cdot \ci_{I_0} \big\|_{s_1} \big\| \one_{E_2} \cdot \ci_{I_0} \big\|_{s_2}
\end{align}
whenever $\lft f(x) \rg \leq \one_{E_1}(x), \lft g(x) \rg \leq \one_{E_2}(x)$, and for $(s_1, s_2, \tilde s)$ admissible tuple in a neighborhood of $( r_1, r_2, r)$.
We will dualize the weak-$L^{\tilde s}$ norm through an $L^\tau$ space, with $\tau<\tilde s$. Given $E_3$ a set of finite measure, we set $\tilde E_3:=E_3 \setminus \tilde \Omega$. The exceptional set $\tilde \Omega$ is defined by the same formula \eqref{eq:def-ex-set-local}. We write $\ds\rr P :=\bigcup_{d \geq 0} \rr P_d$, where all the tiles in $\rr P_d$ have the property that 
\[
1+\frac{\dist (I_P, \tilde \Omega^c)}{\lft I_P \rg} \sim 2^d.
\]

For every $n_1$ with $2^{-n_1} \leq 2^d \frac{\| \one_{E_1} \cdot \ci_{I_0}  \|}{\lft E_3 \rg}$, we perform a stopping time similar to the one described in Section \ref{subsection-stopping_times}. The stopping time will yield a collection $\ii I^{n_1}$ of mutually disjoint intervals, and for every $I \in \ii I^{n_1}$, also a collection $\rr P(I) \subseteq \rr P_d$ of tri-tiles. For each interval $I \in \ii I^{n_1}$, we have 
\[
2^{-n_1-1} \leq \frac{1}{\lft I \rg} \int_{\rr R} \one_{E_1} \cdot \ci_{I} dx \leq 2^{-n_1} \sim \sssize_{\rr P \left( I \right)} \one_{E_1}.
\]
As a consequence, $\sum_{I \in \ii I^{n_1}}\lft I \rg \lesssim 2^{n_1} \| \one_{E_1} \cdot \ci_{I_0}\|_1$.

Moreover, whenever $\rr P' \subseteq \rr P\left( I \right)$, $\ssize_{\rr P'} \one_{E_1} \lesssim 2^{-n_1}$. The collections of intervals $\ii I^{n_2}, \ii I^{n_3}$ associated to $\one_{E_2}$ and $\one_{\tilde E_3}$ will have similar properties.

We choose a $\tau< \tilde s <1$, and it will be enough to estimate
\[
\big\| BHT_{\rr P \left( I_0 \right)}^{F, G, \tilde E} (f, g)  \cdot \one_{\tilde E_3}\big\|_\tau=\big\| BHT_{\rr P \left( I_0 \right)}^{F \cap E_1, G \cap E_2, \tilde E \cap \tilde E_3} (f, g)\big\|_\tau.
\]

The subadditivity and monotonicity of $\| \cdot \|_\tau^\tau$ implies that 
\[
\big\| BHT_{\rr P \left( I_0 \right)}^{F \cap E_1, G \cap E_2, \tilde E \cap \tilde E_3} (f, g)\big\|_\tau^\tau \lesssim \sum_{n_1, n_2, n_3} \sum_{\substack{I=I_1 \cap I_2 \cap I_3 \\ I_j \in \ii I^{n_j}}} \big\| BHT_{\rr P \left(I \right)}^{F \cap E_1, G \cap E_2, \tilde E \cap \tilde E_3} (f, g)\big\|_\tau^\tau.
\]
An estimate similar to Lemma \ref{lemma:sub-unit-local-est-paraprod} is needed; informally, this reduces to
\begin{equation}
\label{eq:sub-unit-bht}
\big \| BHT_{\rr P \left( I \right)}^{F \cap E_1, G \cap E_2, \tilde E \cap \tilde E_3} (f, g) \big \|_\tau \lesssim \big\| BHT_{\rr P \left( I \right)}^{F \cap E_1, G \cap E_2, \tilde E \cap \tilde E_3} (f, g)\|_1 \cdot \big\|  \one_{\tilde E \cap \tilde E_3}  \cdot \ci_I \big\|_{\tau_0},
\end{equation}
where $\tau_0$ is so that $\frac{1}{\tau_0}+1=\frac{1}{\tau}$. Even though we cannot expect to prove such an estimate, we will show that
{\fontsize{10}{10}
\begin{align*}
\big \| BHT_{\rr P \left( I \right)}^{F \cap E_1, G \cap E_2, \tilde E \cap \tilde E_3} (f, g) \big \|_\tau &\lesssim  \big( \sssize_{\rr P \left( I \right)} \one_{F \cap E_1}  \big)^{\frac{1+\theta_1}{2}} \big( \sssize_{\rr P \left( I \right)} \one_{G \cap E_2}  \big)^{\frac{1+\theta_2}{2}} \big( \sssize_{\rr P \left( I_0 \right)} \one_{\tilde E \cap \tilde E_3}  \big)^{\frac{1+\theta_3}{2}-\epsilon} \\
 &\quad \cdot \big\|  \one_{\tilde E \cap \tilde E_3}  \cdot \ci_I \big\|_{\tau_0} \lft I \rg .
\end{align*}}
Compared to the estimate in Corollary \ref{cor:L1-bht-local}, we loose an $\epsilon$ in the exponent of $\sssize_{\rr P \left( I \right)} \one_{\tilde E \cap \tilde E_3}$. The proof is very similar to the estimate in Lemma \ref{lemma:sub-unit-local-est-paraprod}: first we write $\one_{\tilde E \cap \tilde E_3}$ as
\[
\one_{\tilde E \cap \tilde E_3}=\sum_{k_3 \geq 0} \one_{\tilde E_{k_3}},  \text{   where    } \one_{\tilde E_{k_3}}(x)=\one_{\tilde E \cap \tilde E_3} \cdot \one_{\left\lbrace x: \dist(x, I) \sim \left( 2^{k_3}-1\right) \vert I \vert \right\rbrace}.
\]
Using again the subadditivity of $\| \cdot \|_\tau^\tau$, together with the estimate in Corollary \ref{cor:L1-bht-local}, we have that 
\begin{align*}
&\big \| BHT_{\rr P \left( I \right)}^{F \cap E_1, G \cap E_2, \tilde E \cap \tilde E_3} (f, g) \big \|_\tau^\tau \lesssim \sum_{k_3 \geq 0} \big \| BHT_{\rr P \left( I \right)}^{F \cap E_1, G \cap E_2, \tilde E_{k_3}} (f, g) \big \|_1^\tau \cdot \lft \tilde E_{k_3} \rg^{\tau/{\tau_0}} \\
&\lesssim \big( \sssize_{\rr P \left( I \right)} \one_{F \cap E_1}  \big)^{\tau \cdot \frac{1+\theta_1}{2}} \big( \sssize_{\rr P \left( I \right)} \one_{G \cap E_2}  \big)^{\tau \cdot  \frac{1+\theta_2}{2}} \big( \sssize_{\rr P \left( I \right)} \one_{\tilde E_{k_3}}  \big)^{\tau \cdot\frac{1+\theta_3}{2}}\cdot \lft I \rg^\tau \cdot \lft \tilde E_{k_3} \rg^{\tau/{\tau_0}} .
\end{align*}
Noticing that $$\sssize_{\rr P \left( I \right)} \one_{\tilde E_{k_3}} \lesssim \big( \sssize_{\rr P \left( I \right)} \one_{ \tilde E \cap \tilde E_{3}}\big)^{1-\epsilon} \cdot 2^{-k_3 M \epsilon}$$
and that 
\[
2^{-k_3 \tilde M} \cdot \lft \tilde E_{k_3} \rg^{\tau/{\tau_0}} \lesssim \|  \one_{\tilde E_{k_3}} \cdot \ci_I \|_1^{\frac{\tau}{\tau_0}},
\]
H\"older's inequality eventually implies 
\begin{align*}
&\big \| BHT_{\rr P \left( I \right)}^{F \cap E_1, G \cap E_2, \tilde E \cap \tilde E_3} (f, g) \big \|_\tau^\tau \\
&\lesssim \big( \sssize_{\rr P \left( I \right)} \one_{F \cap E_1}  \big)^{\tau \cdot \frac{1+\theta_1}{2}} \big( \sssize_{\rr P \left( I \right)} \one_{G \cap E_2}  \big)^{\tau \cdot  \frac{1+\theta_2}{2}} \big( \sssize_{\rr P \left( I \right)} \one_{\tilde E \cap \tilde E_{3}}  \big)^{\tau \cdot\frac{1+\theta_3}{2}+\frac{\tau}{\tau_0}-\epsilon} \cdot \lft I \rg.
\end{align*}

Now we are ready to prove inequality \eqref{eq:gen-BHT-loc-fns}; indeed, we have 
\begin{align*}
&\big\| BHT_{\rr P \left( I_0 \right)}^{F \cap E_1, G \cap E_2, \tilde E \cap \tilde E_3} (f, g)\big\|_\tau^\tau \lesssim \sum_{n_1, n_2, n_3} \sum_{\substack{I=I_1 \cap I_2 \cap I_3 \\ I_j \in \ii I^{n_j}}} \big\| BHT_{\rr P \left(I \right)}^{F \cap E_1, G \cap E_2, \tilde E \cap \tilde E_3} (f, g)\big\|_\tau^\tau \\
& \lesssim  \sum_{n_1, n_2, n_3} \sum_{\substack{I=I_1 \cap I_2 \cap I_3 \\ I_j \in \ii I^{n_j}}}  \big( \sssize_{\rr P \left( I \right)} \one_{F \cap E_1}  \big)^{\tau \cdot \frac{1+\theta_1}{2}} \big( \sssize_{\rr P \left( I \right)} \one_{G \cap E_2}  \big)^{\tau \cdot  \frac{1+\theta_2}{2}} \big( \sssize_{\rr P \left( I \right)} \one_{\tilde E \cap \tilde E_{3}}  \big)^{\tau \cdot\frac{1+\theta_3}{2}+\frac{\tau}{\tau_0}-\epsilon} \cdot \lft I \rg \\
&\lesssim \big( \sssize_{\rr P \left( I_0 \right)} \one_{F }  \big)^{\tau \cdot \left( \frac{1+\theta_1}{2} -\frac{1}{ s_1} \right)} \big( \sssize_{\rr P \left( I_0 \right)} \one_{G \cap E_2}  \big)^{\tau \cdot \left(  \frac{1+\theta_2}{2}-\frac{1}{s_2} \right)} \big( \sssize_{\rr P \left( I _0\right)} \one_{\tilde E \cap \tilde E_{3}}  \big)^{\tau \cdot\left( \frac{1+\theta_3}{2}-\frac{1}{\tilde s'}-\epsilon \right)} \\
&\cdot  \sum_{n_1, n_2, n_3} \sum_{\substack{I=I_1 \cap I_2 \cap I_3 \\ I_j \in \ii I^{n_j}}} 2^{-\frac{n_1 \tau }{s_1}} 2^{-\frac{n_2 \tau }{s_2}} 2^{-n_3 \left( 1 -\frac{\tau}{\tilde s}+\epsilon\right)}  \cdot \lft I \rg.
\end{align*}
The last line can eventually be bounded above by 
\[
2^{-\tilde M d} \|  \one_{E_1} \cdot \ci_{I_0} \|_1^{\frac{\tau}{s_1}} \|  \one_{E_2} \cdot \ci_{I_0} \|_1^{\frac{\tau}{s_2}} \lft E_3 \rg^{1-\frac{\tau}{\tilde s}}, 
\]
proving, upon summation in $d \geq 0$, the estimate in \eqref{eq:enough-weak-type}.
\end{proof}
\end{proposition}

For the general case of Theorem \ref{thm:main-thm-BHT}, we would need to  prove inductively the following statements:
{\fontsize{9}{10}
\begin{align}
\label{Pn-bht}
\tag*{$\ii P(n):$} 
\big\|  \big\|   BHT_{\rr P \left( I_0 \right)}^{F, G, \tilde E} \left(\vec f, \vec g  \right)\big \|_{L^{R^n}} \big\|_s & \lesssim \big( \sssize_{\rr P\left( I_0 \right)} \one_F  \big)^{\frac{1+\theta_1}{2}-\frac{1}{p}-\epsilon} \big( \sssize_{\rr P\left( I_0\right)} \one_G  \big)^{\frac{1+\theta_2}{2}-\frac{1}{q}-\epsilon} \big( \sssize_{\rr P\left( I_0 \right)} \one_{\tilde E}  \big)^{\frac{1+\theta_3}{2}-\frac{1}{s'}-\epsilon} \\
&\cdot \quad \big \|  \big\|  \vec f \big\|_{L^{R^n_1}} \cdot \ci_{I_0} \big\|_p \big \|  \big\|  \vec g\big\|_{L^{R^n_2}} \cdot \ci_{I_0} \big\|_q. \nonumber
\end{align}}
Also, whenever $\big\|  \vec f (x)\big\|_{L^{R^n_1}} \leq \one_{F}(x)$ and $\big\|  \vec g(x)\big\|_{L^{R^n_2}} \leq \one_G(x)$, we have
{\fontsize{9}{10}\begin{align}
\label{Pn-star-bht}
\tag*{$\ii P^*(n):$} 
\big\|  \big\|   BHT_{\rr P \left( I_0 \right)}^{F, G, \tilde E} \left(\vec f, \vec g  \right)\big \|_{L^{R^n}} \big\|_s & \lesssim \big( \sssize_{\rr P\left( I_0 \right)} \one_F  \big)^{\frac{1+\theta_1}{2}-\epsilon} \big( \sssize_{\rr P\left( I_0\right)} \one_G  \big)^{\frac{1+\theta_2}{2}-\epsilon} \big( \sssize_{\rr P\left( I_0 \right)} \one_{\tilde E}  \big)^{\frac{1+\theta_3}{2}-\frac{1}{s'}-\epsilon} \cdot \lft I_0 \rg^{1/s}.
\end{align}}

We note that $\ii P(0)$ is precisely Lemma \ref{lemma:sub-unit-BHT}, and as a consequence we also obtain $\ii P^*(0)$. $\ii P(1)$ follows through $L^r$ dualization, and the $\ii P^*(0)$ statement is needed. More generally, $\ii P (n)$ follows through $L^{r^1}$ dualization, as a consequence of $\ii P ^*(n-1)$. The proof separates in two cases: $s \geq r^{j_0}$ and $s< r^{j_0}$, just like in the paraproduct case. In fact, the proof follows the same principle, with the difference that now the exponents of the sizes are $\frac{1+\theta_j}{2}$ and not $1$. The details are left to the interested reader.

\section{Mixed norm estimates for $\Pi \otimes \Pi$ and the Leibniz Rule}
\label{sec:mixed_norm-est}

We present the proof of Theorem \ref{thm:kenig-two-param}, in the case when $s_2<1$ (and  as a consequence, $1<p_2, q_2<\infty$). The other situations were considered in \cite{vv_BHT}, with the case $s_1<1$ and $p_2=\infty$ or $q_2=\infty$ being the most difficult.

\begin{proof}[Proof of Theorem \ref{thm:kenig-two-param}]
Since the other cases are very similar, we can assume that $\Pi_y$, the paraproduct acting on the variable $y$ is of the form 
\[
\Pi_y\left( \cdot , \cdot \right)= \sum_k Q_k \left( P_k \left( \cdot  \right) , Q_k \left( \cdot \right) \right).
\]  
Then we can write $\Pi \otimes \Pi$ as
\[
\Pi \otimes \Pi(f, g)(x, y)=\sum_k Q_k^2 \Pi \left( P_k^y, Q_k^y \right)(x).
\]
Using the inequality $\ds \big \| \sum_k Q_k \Phi \big\|_p \leq \big\| \big( \sum_k | Q_k \Phi  |^2  \big)^{1/2} \big\|_p$, which is true for any $0<p<\infty$, we have
\begin{align*}
&\big\|    \big\|    \sum_k Q_k^2 \Pi \big( P_k^y, Q_k^y \big)(x)    \big\|_{L_y^{s_2}}   \big\|_{L_x^{s_1}} \lesssim 
\big\|    \big\|  \big(  \sum_k  \lft \Pi \left( P_k^y, Q_k^y \right)(x)\rg^2 \big)^{1/2} \big\|_{L_y^{s_2}}   \big\|_{L_x^{s_1}}\\
& \lesssim  \big\| \big\|  \sup_k \lft  P_k^y f(x) \rg \big\|_{L_y^{p_2}} \big\|_{L_x^{p_1}} \cdot  \big\|    \big\|    \big( \sum_k \lft Q_k^y g(x)  \rg^2  \big)^{1/2}    \big\|_{L_y^{q_2}}   \big\|_{L_x^{q_1}}.
\end{align*}

In the estimate above we used the multiple vector-valued inequality
\[
\Pi_x : L_x^{p_1}\big(  L_y^{p_2} \big(  \ell^\infty \big) \big) \times  L_x^{q_1}\big(  L_y^{q_2} \big(  \ell^2 \big) \big) \to L_x^{s_1}\big(  L_y^{s_2} \big(  \ell^2 \big) \big),
\]
which is a consequence of Theorem \ref{thm:genral_case}.

Together with the result in \cite{vv_BHT}, we obtain the boundedness of $\Pi \otimes \Pi$ in the whole possible range of Lebesgue exponents. 
\end{proof}

\begin{remark}
In a similar way, mixed-norm $L^p$ estimates for $BHT \otimes \Pi$ can be deduced, using this time the multiple vector-valued estimates for $BHT$ from Theorem \ref{thm:main-thm-BHT}. 
\end{remark}

Now we provide a proof for Theorem \ref{thm:Leibniz}, which will also clarify the necessity of the conditions imposed on the Lebesgue coefficients $s_1$ and $s_2$.

\begin{proof}[Proof of Theorem \ref{thm:Leibniz}]

As usual, the derivatives $D_1^\alpha$ and $D_2^\beta$ will not act directly on the product $f \cdot g$, but on the paraproducts. In the bi-parameter case, the product can be written as a sum of nine paraproducts:
{\fontsize{9}{10}
\[
f\cdot g(x, y)= \underbrace{\sum_{k,l} \left( f \ast \varphi_k \otimes \varphi_l \cdot g \ast \psi_k \otimes \psi_l  \right)\ast \psi_k \otimes \psi_l(x, y) +\ldots +  \left( f \ast \psi_k \otimes \psi_l \cdot g \ast \psi_k \otimes \psi_l  \right)\ast \varphi_k \otimes \varphi_l(x, y).}_{ 9 \text{  terms   }}
\]}

We now claim that the derivative of a paraproduct becomes a paraproduct of certain derivatives of $f$ and of $g$ : $D_1^\alpha D_2^\beta \big( \Pi \otimes \Pi (f, g) \big) =\tilde\Pi \otimes \tilde\Pi (D_1^\alpha  f, D_2^\beta g) $ or a like term.

The derivatives initially are placed on the outer-most terms of the paraproduct, giving rise to expressions of the form $D_1^\alpha \psi_k \otimes D_2^\beta \psi_l, D_1^\alpha \psi_k \otimes D_2^\beta \varphi_l$ or $D_1^\alpha \varphi_k \otimes D_2^\beta \varphi_l$. On the dyadic frequency shell $\vert  \xi\vert \sim 2^k$, the $D_1^\alpha$ derivative acts as multiplication by $2^{k \alpha}$:
\[
D_1^\alpha \psi_k(x)=2^{k\alpha} \tilde \psi_k(x), \quad \text{where}\quad \widehat{\tilde{\psi}}_k=\frac{\vert \xi \vert^\alpha}{2^{k \alpha}} \widehat{\psi}_k(\xi).
\]

For a paraproduct of the type $\sum_k Q_k \left(P_k f \cdot Q_k g \right)$, we have 
\begin{align*}
&D_1^\alpha\big(\sum_k (f \ast \varphi_k \cdot g \ast \psi_k)\ast \psi_k(x)\big)= \sum_k 2^{k \alpha} \left( f \ast \varphi_k \cdot g \ast \psi_k  \right)\ast \tilde \psi_k(x),
\end{align*}
and now the idea is to transform the multiplication by $2^{k \alpha}$ again into a derivative. Note that $$2^{k \alpha} g \ast \psi_k (x) = g \ast D_1^\alpha \tilde{ \tilde{\psi}}_k (x) = D_1^\alpha (g \ast \tilde{\tilde{ \psi}}_k)(x)=(D_1^\alpha g)\ast \tilde{\tilde{ \psi}}_k (x),$$
where $\tilde{\tilde {\psi}}_k$ is defined by $\ds \widehat{\tilde{\tilde{\psi}}}_k(\xi):=\frac{2^{k\alpha}}{\vert \xi \vert^\alpha} \widehat{\psi}_k(\xi)$. In addition, it becomes evident that we couldn't have placed the derivative on $ f \ast \varphi_k$ because $0$ is contained in its Fourier support.

Consequently, in this case, $$ D_1^\alpha( \Pi (f, g))(x)=\sum_k \big( f \ast \varphi_k \cdot \big( D_1^\alpha g \ast \tilde{\tilde \psi}_k\big)    \big)\ast \tilde \psi_k(x):=\tilde\Pi(f, D_1^\alpha g)(x).$$

Similarly, inside the ball $\vert \xi \vert \leq 2^k$ we have 
\[
D_1^\alpha \varphi_k(x)=2^{k\alpha} \tilde{\tilde{  \varphi}}_k(x), \quad \text{where}\quad \widehat{\tilde{\tilde{  \varphi}}}_k=\frac{\vert \xi \vert^\alpha}{2^{k \alpha}} \widehat{\varphi}_k(\xi).
\]
The difference now is that $\widehat{\tilde{\tilde{  \varphi}}}_k$ is not smooth at the origin (unlike $\widehat{\psi}_k$, the support of $\widehat{\varphi_k}$ contains the origin), and $\tilde{\tilde{  \varphi}}_k$ has only finite decay: every $\tilde{\tilde{  \varphi}}_k(x)=2^k \tilde{\tilde{  \varphi}}(2^k x)$, where
\begin{equation}
\label{eq:decay-alpha}
\big| \tilde{\tilde{ \varphi}}(x) \big| \lesssim \frac{1}{\left( 1+\big| x \big|  \right)^{1+\alpha}}.
\end{equation}

A paraproduct associated to a function of fixed decay as in \eqref{eq:decay-alpha} will be denoted $\Pi^\alpha:$
\begin{equation}
\label{eq:def-Pi-alpha}
\Pi^\alpha(f, g)(x):=\sum_k \left( f \ast \psi_k \cdot g \ast \psi_k  \right) \ast \tilde{\tilde \varphi}_k(x).
\end{equation}

To deal with the finite decay in \eqref{eq:decay-alpha}, we split each $\widehat{\tilde{\tilde{ \varphi}}}_k$ into Fourier series onto the set $\vert \xi \vert \leq 2^k$:
\[
\widehat{\tilde{\tilde{ \varphi}}}_k(\xi):=\sum_{n} c_{n,k} e^{\frac{2 \pi i n \xi}{2^k}}=\sum_{n} c_{n,k} e^{\frac{2 \pi i n \xi}{2^k}} \widehat{\tilde \varphi}_k(\xi):=\sum_{n} c_{n,k}  \widehat{\tilde \varphi}_{k, n}(\xi),
\] 
where $\widehat{\tilde \varphi}_k$ is similar to $\widehat{\varphi}_k$, but it is going to be constantly equal to $1$ on $\supp \widehat{\psi}_k +\supp \widehat{\psi}_k$. Moreover, for any function $\Phi \in \ic{S}$, we use the notation $\ds \Phi_{k, n}(x):= 2^k \Phi\left(2^k x +n  \right)=\Phi_k(x+\frac{n}{2^k})$. As a consequence of \eqref{eq:decay-alpha}, the Fourier coefficients satisfy uniformly in $k$
\[
\vert c_{n, k} \vert \lesssim \frac{1}{\left( 1+\vert n \vert  \right)^{1+\alpha}}.
\]

Now we can see how the derivative in the first variable acts on the paraproduct  $\sum_k P_k(Q_k f \cdot Q_k g)$:
\begin{align*}
&D_1^\alpha\big(\sum_k (f \ast \psi_k \cdot g \ast \psi_k)\ast \varphi_k(x)\big)=\sum_n \sum_k c_{n, k}  2^{k \alpha}  \left( f \ast\psi_{k, n} \cdot g \ast\psi_{k, n}  \right) \ast \tilde \varphi_k(x) \\
&=\sum_n \sum_k c_{n, k}  \big(  \big(D_1^\alpha f\big) \ast \tilde{ \tilde{\psi}}_{k, n} \cdot g \ast\psi_{k, n}  \big) \ast \tilde \varphi_k(x).
\end{align*}

We denote $P_{k, n}f(x):=f \ast \varphi_{k, n}(x), Q_{k, n}f(x):=f \ast \psi_{k, n}(x)$. In frequency, these correspond to $\ds\widehat{P_{k, n}f} (\xi):=\hat{f}(\xi) \widehat{\varphi}_k(\xi) e^{\frac{2 \pi i n \xi}{2^k}}$ and $\ds\widehat{Q_{k, n}f} (\xi):=\hat{f}(\xi) \widehat{\psi}_k(\xi) e^{\frac{2 \pi i n \xi}{2^k}}$, respectively. We also used that 
\[
\sum_k c_{n, k} \left( f \ast \psi_k \cdot g \ast \psi_k   \right)\ast \varphi_{k, n}(x)=\sum_k c_{n, k} \left( f \ast \psi_{k,n} \cdot g \ast \psi_{k,n}   \right)\ast \varphi_{k}(x),
\]
which becomes obvious when written on the frequency side:
\begin{align*}
&\int_{\rr R^2}\sum_k c_{n, k} \hat{f}(\xi_1) \widehat{\psi}_{k}(\xi_1) \hat{g}(\xi_2) \widehat{\psi}_{k}(\xi_2) \widehat{\varphi}_k(\xi_1+\xi_2) e^{\frac{2 \pi i n \left( \xi_1+\xi_2 \right)}{2^k}} e^{2 \pi i x \left( \xi_1+\xi_2 \right)}d \xi_1 d \xi_2\\
=&\int_{\rr R^2}\sum_k c_{n, k} \hat{f}(\xi_1) \widehat{\psi}_{k}(\xi_1) e^{\frac{2 \pi i n \xi_1}{2^k}} \hat{g}(\xi_2) \widehat{\psi}_{k}(\xi_2) e^{\frac{2 \pi i n \xi_2}{2^k}} \widehat{\varphi}_k(\xi_1+\xi_2) e^{2 \pi i x \left( \xi_1+\xi_2 \right)} d \xi_1 d \xi_2.
\end{align*}

Hence, the Leibniz rule reduces to the boundedness of the shifted paraproduct
\begin{equation}
\label{def:shifted-paraprod}
\Pi_n(f, g)(x):=\sum_k P_k \left( Q_{k, n}f \cdot Q_{k, n} g  \right)(x)=\sum_k P_{k, n} \left( Q_k  f \cdot Q_k g \right)(x).
\end{equation}

The bilinear operator $\Pi_n$ is very similar to the classical paraproduct $\Pi$ from \eqref{def:classicalParap}, except that we need in this case shifted maximal operators and square functions
\[
\ic M^n (f)(x):= \sup_{I \ni x} \frac{1}{\vert I\vert}\int_{\rr R} \vert  f(y) \vert  \cdot \ci_{I_n}(y)dy \quad {\text{and}} \quad \ic S^n(f)(x):=\big( \sum_{I} \frac{\vert \langle f, \psi_{I_n}   \rangle  \vert^2 }{\vert I \vert}   \cdot \one_{I} (x)\big)^{1/2}.
\]
Above, for a fixed interval $I$, we denote by $I_n:=I +n |I|$, the translation of $I$ $n$ units to the right (or to the left, if $n<0$). It is well known (a complete proof is provided in \cite{multilinear_harmonic}), that these operators are bounded on every $L^p$ space for $1<p<\infty$, with an operatorial norm bounded above by $ \log \left( 1+\vert n \vert\right)$. So in fact we don't loose much by performing this decomposition, and the summability in $n$ is dictated by the decay of the coefficients $c_{n, k}$.

We recall that in proving the boundedness of the paraproduct $\Pi^\alpha$ in one dimension, the more difficult case corresponds to estimates in $L^s$, with $s<1$. More exactly, we have 
\[
\vert \Pi^\alpha (f, g)(x)\vert \leq \sum_{n} \frac{1}{\left( 1 +\vert n \vert \right)^{1+\alpha}} \vert \Pi_n (f, g)(x) \vert \quad \text{and}
\]
\[
\|  \Pi^{\alpha}(f, g) \|_s^s \lesssim \sum_{n} \frac{1}{\left( 1 +\vert n \vert \right)^{\left(1+\alpha\right)s}} \|  \Pi_n(f, g) \|_s^s \lesssim \sum_n \frac{\left(\log\left( 1+\vert n \vert \right)\right)^{2}}{\left( 1 +\vert n \vert \right)^{\left(1+\alpha\right)s}} \|f\|_p \|  g \|_q.
\]
Provided $\left( 1+\alpha \right) s>1$, we obtain the $L^p \times L^q \to L^s$ boundedness of $\Pi^\alpha$.

A similar analysis will yield the general Leibniz of Theorem \ref{thm:Leibniz}; in the end, we will need to study the boundedness of $\big\|\Pi \otimes \Pi \big\|_{L_x^{s_1}L_y^{s_2}},  \big\|\Pi^\alpha \otimes \Pi \big\|_{L_x^{s_1}L_y^{s_2}}, \big\|\Pi \otimes \Pi^\beta \big\|_{L_x^{s_1}L_y^{s_2}}$ or $\big\|\Pi^\alpha \otimes \Pi^\beta \big\|_{L_x^{s_1}L_y^{s_2}}$. Ultimately, the range of $L^p$ estimates in Theorem \ref{thm:Leibniz} will be determined by the ``worst" term, which is $\Pi^\alpha \otimes \Pi^\beta$. If $s_0$ denotes the minimum between $s_1$ and $s_2$, Proposition \ref{prop:reorder} implies that $\| \cdot \|_{L_x^{s_1}L_y^{s_2}}^{s_0}$ is subadditive. Following the arguments presented earlier, we have 
\begin{equation}
\label{eq:bdd-Pi-alpha-Pi-beta}
\big\| \Pi^\alpha \otimes \Pi^\beta \big\|_{L_x^{s_1}L_y^{s_2}}^{s_0} \lesssim \sum_n \frac{1}{\left( 1 +\vert n \vert \right)^{\left(1+\alpha\right)s_0}} \big\| \Pi_n \otimes \Pi^\beta \big\|_{L_x^{s_1}L_y^{s_2}}^{s_0}.
\end{equation}
Provided 
\begin{equation}
\label{eq:cond-alpha-s_1}
\min(s_1, s_2) >\frac{1}{1+\alpha},
\end{equation}
the boundedness of $\Pi^\alpha \otimes \Pi^\beta$ reduces to that of $\Pi_n \otimes \Pi^\beta$, with an operatorial norm that depends at most logarithmically on $n$. 


If $n=0$, we need to prove that $\Pi \otimes \Pi^\beta : L_x^{p_1}L_y^{p_2} \times L_x^{q_1}L_y^{q_2} \to L_x^{s_1}L_y^{s_2}$, whenever
\begin{equation}
\label{eq:cond-s_2}
s_2>\frac{1}{1+\beta}
\end{equation}
The conditions \eqref{eq:cond-alpha-s_1} and \eqref{eq:cond-s_2} are equivalent to the constraints of $s_1$ and $s_2$ from the hypotheses of Theorem \ref{thm:Leibniz}.

In order to establish that $\Pi \otimes \Pi^\beta : L_x^{p_1}L_y^{p_2} \times L_x^{q_1}L_y^{q_2} \to L_x^{s_1}L_y^{s_2}$, we use restricted type interpolation, as in Proposition \ref{prop:interpolation}: it will be enough to prove 
\begin{equation}
\label{eq:restricted-bi-par}
\|  \|  \Pi \otimes \Pi^\beta (f, g)   \|_{L_y^{s_2}} \cdot \one_{\tilde E} \|_{L_x^{s_2}} \lesssim \vert F \vert^{\frac{1}{p_1}}  \cdot \vert G \vert^{\frac{1}{q_1}} \cdot \vert E \vert^{\frac{1}{s_1}-\frac{1}{s_2}},
\end{equation}
where $F, G$ and $E$ are sets of finite measure, $\tilde E$ is a major subset of $E$ to be constructed (it is defined by \eqref{eq:def-exceptional-set}), while $f$ and $g$ are functions satisfying $\| f(x, \cdot ) \|_{L_y^{p_2}} \leq \one_F(x)$ and $\| g(x, \cdot ) \|_{L_y^{q_2}} \leq \one_G(x)$, respectively.

The cases $s_2<1$ and $s_2 \geq 1$ need to be treated separately. We first deal with the case $s_2<1$. In fact, we will be proving sharp estimates for $\big\| \Pi_{I_0}^{F, G, \tilde E} \otimes \Pi^\beta\big\|_{L_{xy}^{s_2}}$, where $\Pi_{I_0}^{F, G, \tilde E}$ is the same discretized, localized paraproduct introduced in \eqref{eq-def-loc-paraprod}. Before, we were using the localized paraproducts in order to deduce multiple vector-valued inequalities for $\Pi$, and from there, mixed norm $L^p$ estimates for $\Pi \otimes \Pi$. Now we work directly with $\Pi_{\ii I \left( I_0 \right)} \otimes \Pi^\beta$, and we want to prove that
\begin{equation*}
\|  \|  \Pi_{\ii I\left(I_0\right)} \otimes \Pi^\beta (f, g)   \|_{L_y^{s_2}} \cdot \one_{\tilde E} \|_{L_x^{s_2}}^{s_2} \lesssim \big( \sssize_{I_0} \one_F \big)^{s_2-\epsilon} \big( \sssize_{I_0} \one_G \big)^{s_2-\epsilon} \big( \sssize_{I_0} \one_{\tilde E} \big)^{1-\epsilon} \cdot \vert I_0 \vert,
\end{equation*}
for any $f$ and $g$ as above.

Formerly, we decomposed $\Pi^\alpha$ using Fourier series in frequency, and now we are going to do the same for $\Pi^\beta$. In this way, we can write it as
\begin{equation}
\label{eq:dec-Pi-beta}
\Pi^\beta (f, g)(y):=\sum_{m} \sum_l c_{m, l} \left( f \ast \psi_{l, m} \cdot g \ast \psi_{l, m}   \right) \ast {\tilde\varphi}_l(y),
\end{equation}
where the Fourier coefficients satisfy $\vert c_{m, l} \vert \lesssim \frac{1}{\left( 1+\vert m \vert \right)^{\left( 1+\beta \right)}}$, and $\widehat{\tilde{\varphi}}_l \equiv 1$ on $\supp \widehat{\psi}_{l, m} + \supp \widehat{\psi}_{l, m}$. Since $s_2<1$, we have 
\begin{align*}
\big\| \Pi_{\ii I \left( I_0 \right)}^{F, G, \tilde E} \otimes    \Pi^\beta (f, g) \|_{L_{xy}^{s_2}}^{s_2} &\lesssim \sum_m \frac{1}{\left( 1+\vert m \vert \right)^{s_2\left( 1+\beta \right)}} \big\| \sum_l P_l^2 \Pi_{\ii I \left( I_0 \right)}^{F, G, \tilde E}(Q_{l, m}^y f,  Q_{l, m}^y g)(x) \|_{L_{xy}^{s_2}}^{s_2} \\
&\lesssim \sum_m \frac{1}{\left( 1+\vert m \vert \right)^{ s_2\left( 1+\beta \right)}} \big\| \sum_l  \vert \Pi_{\ii I \left( I_0 \right)}^{F, G, \tilde E}(Q_{l, m}^y f,  Q_{l, m}^y g)(x)\vert \|_{L_{xy}^{s_2}}^{s_2}.
\end{align*}

To deduce the last inequality, we used that $\widehat{\tilde{ \varphi}}_l \equiv 1$ on $\supp \widehat{\psi}_{l, m} + \supp \widehat{\psi}_{l, m}$, which further indicates that $P_l (f \ast \psi_{l, m} \cdot g \ast \psi_{l, m})(x)=f \ast \psi_{l, m} \cdot g \ast \psi_{l, m}(x)$.

The vector-valued estimates for $\Pi_{I_0}^{F, G, \tilde E}$ from Theorem \ref{thm:genral_case} imply that 
\begin{align*}
&\big\| \sum_l  \vert \Pi_{\ii I \left( I_0 \right)}^{F, G, \tilde E}(Q_{l, m}^y f,  Q_{l, m}^y g)(x)\vert \|_{L_{xy}^{s_2}} \lesssim \big( \sssize_{I_0} \one_F   \big)^{\frac{1}{p_2'}-\epsilon} \cdot \big( \sssize_{I_0} \one_G   \big)^{\frac{1}{q_2'}-\epsilon} \cdot \big( \sssize_{I_0} \one_{\tilde E}   \big)^{\frac{1}{s_2}-\epsilon} \\
&\qquad \quad \cdot \big \|  \big(\sum_l \vert Q_{l, m}^y f  \vert^2   \big)^{1/2} \cdot \ci_{I_0}(x) \big\|_{L_{xy}^{p_2}} \cdot \big \|  \big(\sum_l \vert Q_{l, m}^y g  \vert^2  \cdot \ci_{I_0}(x) \big)^{1/2} \big\|_{L_{xy}^{q_2}}.
 \end{align*}

Then, because $1<p_2, q_2<\infty$, the shifted square function is bounded and we have 
{\fontsize{9}{10}
\begin{align*}
\big\| \sum_l  \vert \Pi_{\ii I \left( I_0 \right)}^{F, G, \tilde E}(Q_{l, m}^y f,  Q_{l, m}^y g)(x)\vert \|_{L_{xy}^{s_2}}^{ s_2} \lesssim \left( 1 +\log \vert m \vert \right)^{2} \big( \sssize_{I_0} \one_F   \big)^{s_2-\epsilon} \cdot \big( \sssize_{I_0} \one_G   \big)^{s_2-\epsilon} \cdot \big( \sssize_{I_0} \one_{\tilde E}   \big)^{1-\epsilon} \cdot \vert I_0 \vert.
\end{align*}}

With the above estimate and the usual stopping times from Section \ref{subsection-stopping_times}, for each $d \geq 0$, we have collections $\ii I_d^{n_1, n_2, n_3}$ of dyadic intervals for which 
\begin{enumerate}
\item if $I \in \ii I^{n_1}$, then $\ds 2^{-n_1} \sim \frac{1}{\vert I \vert}\int_{\rr R} \one_{F} \cdot \ci_{I} dx \lesssim 2^d \frac{\vert F \vert}{|E|}$\\
\item if $I \in \ii I^{n_2}$, then $\ds 2^{-n_2} \sim \frac{1}{\vert I \vert}\int_{\rr R} \one_{G} \cdot \ci_{I} dx \lesssim 2^d \frac{\vert G \vert}{|E|}$\\
\item if $I \in \ii I^{n_3}$, then $\ds 2^{-n_3} \sim \frac{1}{\vert I \vert}\int_{\rr R} \one_{\tilde E} \cdot \ci_I   dx \lesssim 2^{-M d} $.
\end{enumerate}
Moreover, for every $I_0 \in \ii I_d^{n_1, n_2, n_3}$, there exists a certain collection $\ii I(I_0)$ associated to $I_0$, which is selected through the stopping time.  This yields a partition of $\ii I$ as $\ds \ii I := \bigcup_{d \geq 0} \bigcup_{n_1, n_2, n_3} \bigcup_{I_0 \in \ii I_d^{n_1, n_2, n_3}} \ii I(I_0)$, which we use in order to estimate 
\begin{align*}
&\big\| \Pi \otimes \Pi^\beta \big\|_{L_{xy}^{s_2}}^{s_2} \lesssim \sum_{d \geq 0} \sum_{n_1, n_2, n_3} \sum_{I_0 \in \ii I^{n_1, n_2, n_3}_d} \big\| \Pi_{\ii I\left( I_0 \right)}^{F, G, \tilde E} \otimes \Pi^\beta \big\|_{L_{xy}^{s_2}}^{s_2} \\
&\qquad\lesssim \sum_{d \geq 0} \sum_{n_1, n_2, n_3} \sum_{I_0 \in \ii I^{n_1, n_2, n_3}_d} \big( \sssize_{I_0} \one_F   \big)^{s_2-\epsilon} \cdot \big( \sssize_{I_0} \one_G   \big)^{s_2-\epsilon} \cdot \big( \sssize_{I_0} \one_{\tilde E}   \big)^{1-\epsilon} \cdot \vert I_0 \vert \\
&\qquad \lesssim \sum_{d \geq 0} \sum_{n_1, n_2, n_3} \sum_{I_0 \in \ii I^{n_1, n_2, n_3}_d} 2^{-n_1 \frac{s_2}{p_1}} 2^{-n_2 \frac{s_2}{p_2}} 2^{-n_3\left(1-\epsilon\right)} \vert I_0\vert.
\end{align*}

The sum $\ds\sum_{I_0 \in \ii I^{n_1, n_2, n_3}_d} \vert I_0 \vert$ is bounded above by $\left( 2^{n_1} \vert F \vert\right)^{\gamma_1} \cdot\left( 2^{n_2} \vert G \vert\right)^{\gamma_2} \left( 2^{n_3} \vert E \vert\right)^{\gamma_3}$, where $0 \leq\gamma_1, \gamma_2, \gamma_3 \leq 1$ and $\gamma_1+ \gamma_2+\gamma_3 = 1$. Hence we obtain 
\begin{align*}
&\big\| \Pi \otimes \Pi^\beta \big\|_{L_{xy}^{s_2}}^{s_2} \lesssim \sum_{d \geq 0} \sum_{n_1, n_2, n_3} 2^{-n_1\left( \frac{s_2}{p_1}-\gamma_1 \right)} \cdot 2^{-n_2\left( \frac{s_2}{q_1}-\gamma_2 \right)} \cdot 2^{-n_3\left( 1-\epsilon-\gamma_3 \right)} \vert F \vert^{\gamma_1} \cdot \vert G \vert^{\gamma_2} \cdot \vert E \vert^{\gamma_3},
\end{align*} 
and the series above converge (we have the freedom to choose $\gamma_1, \gamma_2$ and $\gamma_3$) provided $\frac{s_2}{p_1}+\frac{s_2}{q_1}+ 1-\epsilon>\gamma_1+\gamma_2+\gamma_3=1$. The condition is satisfied thanks to the contribution of $\sssize_{I_0} \one_{\tilde E}$ which comes with an exponent arbitrarily close to $1$.

Finally, it is not difficult to see that all of the above imply exactly \eqref{eq:restricted-bi-par}.

We still have to treat the case $s_2 \geq 1$: that is, we want to prove that $\Pi \otimes \Pi^\beta$ is bounded in the space $\|   \cdot \|_{L_x^{s_1}\left( L_y^{s_2} \right)}$. Since $s_2 \geq 1$, we can dualize the inner norm, and using generalized restricted type interpolation, it is enough to prove
\begin{equation}
\label{eq:dual-s_2geq1}
\Big\vert \int_{\rr R^2} \Pi \otimes \Pi^\beta (f, g)(x, y) \cdot h(x, y) dx dy \Big\vert \lesssim \vert F \vert^{\frac{1}{p_1}}  \cdot \vert G \vert^{\frac{1}{q_1}} \cdot \vert E \vert^{1-\frac{1}{s_1}},
\end{equation} 
whenever  $F, G$ and $E$ are sets of finite measure, $E'$ is a major subset of $E$ to be constructed (it is also defined by \eqref{def:exceptional-set}), while $f, g$ and $h$ are functions satisfying $\| f(x, \cdot ) \|_{L_y^{p_2}} \leq \one_F(x),  \| g(x, \cdot ) \|_{L_y^{q_2}} \leq \one_G(x)$, and $\| h(x, \cdot ) \|_{L_y^{s_2'}} \leq \one_{E'}(x)$ respectively.

From here on, everything follows the same pattern:
\begin{align*}
\big\vert \int_{\rr R^2} \Pi \otimes \Pi^\beta (f, g)(x, y) \cdot h(x, y) dx dy \big\vert\lesssim \sum_m \frac{1}{\left(  1+\vert m\vert \right)^{1+\beta}} 
\big\vert \int_{\rr R^2} \Pi \otimes \Pi_m (f, g)(x, y) \cdot h(x, y) dx dy \big\vert,
\end{align*}
and in fact we will need to estimate $\Pi_{I_0}^{F, G, E'} \otimes \Pi_m$. We don't repeat the argument because it's identical to the situation $s_2<1$. 

The case $p_2=\infty$ or $q_2=\infty$ (which is acceptable since now $s_2 \geq 1$) needs an additional justification, but the proof reduces to the boundedness of $\Pi \otimes \Pi: L_x^{p_1}L_y^\infty \times L_x^{q_1}L_y^{q_2} \to L_x^{s_1}L_y^{q_2}$. The latter was proved in \cite{vv_BHT}, using a similar strategy: due to restricted-type interpolation, it is enough to prove a sharp estimate for the adjoint $\big(\Pi_{I_0}^{F, G, H'} \otimes \Pi \big)^{1, *}$ which is defined by the relation
\[
\int_{\rr R^2} \Pi_{I_0}^{F, G, H'} \otimes \Pi (f, g)(x, y) \cdot h(x, y) dx dy = \int_{\rr R^2} \big(\Pi_{I_0}^{F, G, H'} \otimes \Pi \big)^{1, *} (h, g)(x, y) \cdot f(x, y) dx dy.
\]
The sharp estimate concerns the operatorial norm:
\begin{equation*}
 \big\| \big(\Pi_{I_0}^{F, G, H'} \otimes \Pi \big)^{*, 1}\big\|_{L_x^{q'} L_y^{q'}\times L_x^q L_y^q \to L_x^1 L_y^1} \lesssim \ \big( \sssize_{I_0} \one_{H'}\big)^{\frac{1}{q}-\epsilon}\big( \sssize_{I_0} \one_{G}\big)^{\frac{1}{q'}-\epsilon} \big( \sssize_{I_0} \one_{F}\big)^{1-\epsilon}.
\end{equation*}

This ends the proof in the case $n=0$, when the paraproduct $\Pi_n$ is a classical paraproduct. We are left with proving, for any $\vert n \vert \geq 1$, that 
\begin{equation}
\label{eq:Pi_notimesPibeta}
\big\| \Pi_n \otimes \Pi^\beta (f,g) \big\|_{L_x^{s_1}L_y^{s_2}} \lesssim  \log (1+ \vert n \vert)^{2} \|f\|_{L_x^{p_1}L_y^{p_2}} \cdot \|g\|_{L_x^{q_1}L_y^{q_2}}.
\end{equation}
Together with \eqref{eq:bdd-Pi-alpha-Pi-beta}, the above inequality implies the boundedness of $\Pi^\alpha \otimes \Pi^\beta$.  Similarly to the case $n=0$, we use vector-valued restricted type interpolation, and the equivalent of \eqref{eq:restricted-bi-par} in this case is
\begin{equation}
\label{eq:restricted-bi-par-n}
\big\|  \big\| \Pi_n \otimes \Pi^\beta (f, g) \big\|_{L_y^{s_2}} \cdot \one_{\tilde E} \big\|_{L_x^{s_2}} \lesssim \log (1 +\vert n \vert )^{2} \vert F \vert^{\frac{1}{p_1}} \cdot \vert G \vert^{\frac{1}{q_1}} \cdot \vert E \vert^{\frac{1}{s_1}-\frac{1}{s_2}}.
\end{equation}

To achieve this, we need to prove local estimates for the discretized version of $\Pi_n$, which is the operator
\begin{equation}
\label{eq:def-discretized-shifted-paraprod}
\Pi_n(f, g)(x):=\sum_{I \in \ii I} \frac{1}{\vert I \vert} \langle f, \psi_{I_n} \rangle \langle g, \psi_{I_n} \rangle \varphi_I(x).
\end{equation} 
If we look at the intervals $I \in \ii I$ so that $I \subseteq I_0$, their translates $I_n$ need not be contained inside $I_0$; in fact, there are approximately $\log \vert n\vert$ translates of $I_0$ that could possibly contain such a $I_n$. This is also the key observation in proving the boundedness of the shifted maximal operator $\ic M^n$ or of the shifted square function $\ic S^n$ with an operatorial norm not larger that $1+\log \vert n\vert$. 

In order to make sure that $\log \vert n \vert\neq 0$, we replace it by the equivalent expression $\log \nn := \log \left( 1+ n^2 \right)^{1/2} $. Then given a fixed dyadic interval $I_0$, we denote by $I_0^\sharp$, with $0 \leq \sharp \leq \log \nn $ the translates of $I_0$ that contain some intervals $I_n$ with $I \in \ii I (I_0)$. These are actually the $2^l$-translates of $I_0$, for $2^l \leq n$.

The local estimate for $\Pi_n$, corresponding to Proposition \ref{prop:localization-lemma}, reads as 
{\fontsize{9}{9}
\begin{align*}
\big\| \Pi_{n, I_0}^{F, G, \tilde E}(f, g)   \big\|_r^{r_0} \lesssim \sum_{0 \leq \sharp_1, \sharp_2 \leq \log \nn } \big( \sssize_{I_0^{\sharp_1}} \one_F  \big)^{\frac{r_0}{r_1'}-\epsilon} \big( \sssize_{I_0^{\sharp_2}} \one_G  \big)^{\frac{r_0}{r_2'}-\epsilon} \big( \sssize_{I_0} \one_{\tilde E}  \big)^{\frac{r_0}{r}-\epsilon}
 \cdot \big\|  f \cdot \ci_{I_0^{\sharp_1}} \big\|_{r_1}^{r_0}  \cdot \big\|  g \cdot \ci_{I_0^{\sharp_2}} \big\|_{r_1}^{r_0}, 
\end{align*}}
where $r_0:=\min \left( r, 1 \right)$. The localized vector-valued paraproduct will satisfy similar estimates.

The stopping time is similar, but it is defined by more parameters; the $\sssize$ for the functions $f$ and $g$ are given by
\[
\sssize^n_{\ii I} f :=\sup_{I \in \ii I} \frac{1}{\vert I \vert} \int_{\rr R } \vert f(x) \vert \ci_{I_n}(x) dx, 
\]
so the collections $\ii I_{n_1}$ from the stopping time described in Section \ref{subsection-stopping_times} will be replaced by collections $\ii I_{n_1, \sharp_1}$, with $0 \leq \sharp_1 \leq \log \nn$. An interval $I_1 \in \ii I_{n_1, \sharp_1}$ if there exists $I \in \ii I_{Stock}, I \subseteq I_1$ so that $I_n \subseteq I_1^{\sharp_1}$ and 
\begin{equation}
\label{eq:cond-shifted-stopping-time}
2^{-n_1-1}\leq \frac{1}{|I|}\int_{\rr R} \one_F \cdot \ci_{I_n} dx \leq 2^{-n_1}.
\end{equation}
In fact, $I_1^{\sharp_1}$ is selected prior to $I_1$, from the set of the intervals containing such a $I_n$ satisfying the above condition. Moreover, we require $I_1^{\sharp_1}$ to satisfy a condition similar to \eqref{eq:cond-shifted-stopping-time}, and to be maximal among the intervals meeting these properties. Then the collection $\ii I_{n_1, \sharp_1}(I_1)$ will consist of 
\[
\ii I_{n_1, \sharp_1}(I_1):=\lbrace I \in \ii I_{Stock}: I \subset I_1,  I_n \subseteq I_1^{\sharp_1}   \rbrace.
\]
We note that an interval $I_1$ can be selected in several collections $\ii I_{n_1, \sharp_1}$, but however in no more than $\log \nn$ of them. We also note that 
\[
\sssize^n_{\ii I_{n_1, \sharp_1(I_1)}} \one_F \lesssim 2^{-n_1} \lesssim \log \nn 2^d \frac{\vert F \vert}{\vert E \vert},
\]
since in this case the exceptional set is defined by:
\[
\Omega:=\Big\lbrace x: \ic M^n \one_F > C \log \nn \frac{|F|}{|E|}  \Big\rbrace \cup \Big\lbrace x: \ic M^n \one_G > C \log \nn \frac{|G|}{|E|}  \Big\rbrace.
\]
The parameter $d \geq 0$ is introduced as before, in order to control $\ic M^n \one_F$ and $\ic M^n \one_G$: we have a partition $\ds \ii I:=\bigcup_{d \geq 0} \ii I _d$, where for all $I \in \ii I_d$ we require that $\ds 1+\frac{\dist (I, \Omega^c)}{\vert I \vert} \sim 2^d$.

The stopping time for $\Pi_n$ is more elaborate because we need to find a way of grouping the intervals in $\ii I$ so that the shifted size $\sssize_{\ii I}^n \one_F \sim 2^{-n_1}$, and at the same time we need to assure some disjointness that will allow us to estimate $\ds \sum_I \vert I \vert$.  Using the disjointness of the intervals $I_1^{\sharp_1}$ as $I_1$ varies in $\ii I_{n_1, \sharp_1}$, we have 
\[
\sum_{I_1 \in \ii I_{n_1, \sharp_1}} \vert I_1 \vert=\sum_{I_1 \in \ii I_{n_1, \sharp_1}} \vert I_1^{\sharp_1} \vert \lesssim 2^{n_1} \vert F \vert, 
\]
since every $I_1^{\sharp_1} \subseteq \lbrace \ic M \one_F > 2^{-n_1-1}  \rbrace$ (it satisfies the condition \eqref{eq:cond-shifted-stopping-time}).

We are now ready to prove the desired estimate for $\| \Pi_n \otimes \Pi^\beta (f, g)(x,y) \cdot \one_{\tilde E}(x)\|_{L_{xy}^{s_2}}$. For simplicity, we illustrate the main ideas in the case $s_2<1$; in the case $s_2 \geq 1$, we only need to rewrite the argument by employing the trilinear form.

Using the above stopping time, we have
\[
\big\| \Pi_n \otimes  \Pi^\beta (f, g) \cdot \one_{\tilde E} \big\|_{L_{xy}^{s_2}}^{s_2} \leq \sum_{d \geq 0} \sum_{n_1, n_2, n_3} \sum_{0 \leq \sharp_1, \sharp_2 \leq \log \nn  }\sum_{I_0 \in \ii I_{\sharp_1, \sharp_2}^{n_1, n_2, n_3}} \big\| \Pi_{n,  \sharp_1, \sharp_2, I_0}^{F, G, \tilde E} \otimes  \Pi^\beta (f, g) \big\|_{L_{xy}^{s_2}}^{s_2}.
\]

Now we decompose $\Pi^\beta$ as in \eqref{eq:dec-Pi-beta}, and provided $(1+\beta)s_2>1$, it will be enough to prove
{\fontsize{10}{10}
\begin{align*}
\big\| \sum_l \big\vert \Pi_{n,  \sharp_1, \sharp_2, I_0}^{F, G, \tilde E} (Q_{l, m}^y f, Q_{l, m}^y g)   \big \vert\big\|_{L_{xy}^{s_2}} &\lesssim \big\|  \Pi_{n,  \sharp_1, \sharp_2, I_0}^{F, G, \tilde E} \big\|  \\
&\cdot \big\| \big( \sum_l \vert  Q_{l, m}^y f \vert^2  \big)^{1/2} \cdot \ci_{I_0^{\sharp_1}}  \big\|_{L_{xy}^{p_2}}  \cdot \big\| \big( \sum_l \vert  Q_{l, m}^y g \vert^2  \big)^{1/2} \cdot \ci_{I_0^{\sharp_2}}  \big\|_{L_{xy}^{p_2}},
\end{align*}}
with an operatorial norm
\[
 \big\| \Pi_{n,  \sharp_1, \sharp_2, I_0}^{F, G, \tilde E} \big\|  \lesssim \big( \sssize_{I_0^{\sharp_1}} \one_F \big)^{\frac{1}{p_2'}-\epsilon} \cdot \big( \sssize_{I_0^{\sharp_2}} \one_G \big)^{\frac{1}{p_2'}-\epsilon} \cdot \big( \sssize_{I_0} \one_{\tilde E} \big)^{\frac{1}{s_2}-\epsilon}.   
\]
This follows from the boundedness of the shifted square function (a certain power of $\log 1+ |m|$ will appear, but it doesn't affect the summation in $m$) and the usual vector-valued estimates for the paraproduct $\Pi_{n,  \sharp_1, \sharp_2, I_0}^{F, G, \tilde E}$, which localizes well.

In the end, we will have 
\begin{align*}
\big\| \Pi_n \otimes  \Pi^\beta (f, g) \cdot \one_{\tilde E} \big\|_{L_{xy}^{s_2}}^{s_2} &\lesssim \sum_{d \geq 0} \sum_{n_1, n_2, n_3} \sum_{0 \leq \sharp_1, \sharp_2 \leq \log \nn  }\sum_{I_0 \in \ii I_{\sharp_1, \sharp_2}^{n_1, n_2, n_3}}  \big( \sssize_{I_0^{\sharp_1}} \one_F \big)^{s_2-\epsilon}  \\
& \cdot \big( \sssize_{I_0^{\sharp_2}} \one_G \big)^{s_2-\epsilon} \cdot \big( \sssize_{I_0} \one_{\tilde E} \big)^{1-\epsilon} \vert I_0 \vert.
\end{align*}

Since we have control over all the sizes and over $\sum_{I_0} \vert I_0 \vert$, we can easily obtain the inequality
\[
\big\| \Pi_n \otimes  \Pi^\beta (f, g) \cdot \one_{\tilde E} \big\|_{L_{xy}^{s_2}}^{s_2} \lesssim \log \nn ^2 \vert F \vert^{\frac{s_2}{p_1}} \cdot \vert G \vert^{\frac{s_2}{p_2}} \cdot \vert E \vert^{\frac{s_2}{s_1}-1},
\]
which ends the proof because this is exactly the estimate \eqref{eq:restricted-bi-par-n}.

We want to emphasize however that without using the vector-valued point of view (and vector-valued restricted-type interpolation), it is difficult to remove the constraint that $s_1>\frac{1}{1+\beta}$, which is implied by splitting both $\Pi^\alpha$ and $\Pi^\beta$ from the beginning, as in \eqref{eq:bdd-Pi-alpha-Pi-beta}. Also, in the case $s_2 \geq 1$, dualizing through $L^{s_2}$ is not enough, and we have to bring forth the trilinear form.
\end{proof}

\section{Proof of Interpolation Proposition \ref{prop:interpolation}}
\label{sec:proofs-interp-etc}
\begin{proof}
The tuples $(p, q, s)$ and $( r_1, r_2, r)$ are fixed, but $( s_1, s_2, \tilde s)$ are allowed to vary in a neighborhood of $( p, q, s)$. We will decompose both $\vec f =\lbrace f_k \rbrace_k$ and $\vec g =\lbrace g_k \rbrace_k$ into pieces that we can control:
\[
\vec f=\sum_{m_1} \vec f_{m_1}=\sum_{m_1} \vec f \cdot \one_{\left\lbrace x: 2^{m_1} \leq \big\| \vec f \big\|_{\ell^{r_1}} <2^{m_1+1} \right \rbrace},
\]
and similarly 
\[
\vec g=\sum_{m_2} \vec g_{m_2}=\sum_{m_2} \vec g \cdot \one_{\left\lbrace 2^{m_2} \leq \big\| \vec g \big\|_{\ell^{r_2}} <2^{m_2+1} \right \rbrace}.
\]
We note that for every $m_1 \in \rr Z$, $\| \vec f_{m_1}  \|_{\ell^{r_1}} \sim 2^{m_1}$ and is supported on a set of finite measure. For simplicity, we will assume that $\big\| \big( \sum_k \lft f_k \rg^{r_1} \big)^{1/{r_1}} \big\|_p = \big\|\big( \sum_k \lft g_k \rg^{r_2} \big)^{1/{r_2}} \big\|_q=1$.

Given a function $\varphi$, we will use the \emph{distribution function} $d_{\varphi}$ for estimating the $L^p$ norm of $\varphi$. We recall that  $d_{\varphi}$ is a function from $\rr R^+$ to $\left[ 0, \infty \right]$, defined by
\[
d_{\varphi}(\lambda):=\lft \left \lbrace x: \lft\varphi(x)\rg >\lambda  \right\rbrace \rg.
\]
Then we have, for any $0<p<\infty$,
\begin{equation}
\label{eq:formula-distrib-Lp}
\| \varphi  \|_p^p=p \int_{\rr R^+} \lambda^{p-1} d_\varphi(\lambda) d\lambda.
\end{equation}
We will be using a discrete variant of the formula above:
\[
\| \varphi  \|_p^p \sim \sum_{n \in \rr Z} 2^{np} d_{\varphi}(2^n).
\]

The assumptions on the $L^p\big( \ell^{r_1} \big)$ and $L^q\big( \ell^{r_2} \big)$ norms of $\vec f$ and $\vec g$ respectively translate into
\begin{equation}
\label{eq:cond-norm1-f}
\big\| \big( \sum_k \lft f_k \rg^{r_1} \big)^{1/{r_1}} \big\|_p^p \sim \sum_{m_1}2^{m_1p} d_{\big\| \vec f  \big\|_{\ell^{r_1}}}(2^{m_1}) \sim 1
\end{equation}
and 
\begin{equation}
\label{eq:cond-norm1-g}
\big\| \big( \sum_k \lft g_k \rg^{r_2} \big)^{1/{r_2}} \big\|_q^q \sim \sum_{m_2}2^{m_2q} d_{\big\| \vec g  \big\|_{\ell^{r_2}}}(2^{m_2}) \sim 1.
\end{equation}
For $T(f, g)$, we have the estimate
\[
\big\| \big(  \sum_k  \lft T(f_k, g_k)\rg^r  \big)^{1/r}   \big\|_{s}^s  \sim \sum_n 2^{ns} d_{\big(  \sum_k  \lft T(f_k, g_k)\rg^r  \big)^{1/r}}(2^n).
\]
However, since $r<1$, the $\| \cdot \|_{\ell^r}$ is not subadditive, and this is a property that plays an important role in the classical proofs of interpolation theorems. We will use instead the subadditivity of $\| \cdot \|_{\ell^r}^r$:
\begin{equation}
\label{eq:estLp_norm_distrib}
\big\| \big(  \sum_k  \lft T(f_k, g_k)\rg^r  \big)^{1/r}   \big\|_{s}^s =\big\| \sum_k  \lft T(f_k, g_k)\rg^r   \big\|_{s/r}^{s/r} \sim 
\sum_n 2^{ns/r} d_{ \sum_k  \lft T(f_k, g_k)\rg^r }(2^n).
\end{equation}
We need to estimate the distribution function of $\sum_k  \lft T(f_k, g_k)\rg^r$. First, we note that 
\[
d_{ \sum_k  \lft T(f_k, g_k)\rg^r }(2^n) \leq \sum_{m_1, m_2} d_{ \sum_k  \lft T(f_{k, m_1}, g_{k, m_2})\rg^r }(c_{n, m_1, m_2}2^n),
\]
where $c_{n, m_1, m_2}>0$ will be chosen later, with the property that $\sum_{m_1, m_2} c_{n, m_1, m_2} \sim 1$.

Condition \eqref{eq:interp-restricted-weak-type} generalizes to a weak-type condition:
{\fontsize{9}{10}
\[
\big\| \big( \sum_k \lft T( F_k, G_k) \rg^{r} \big)^{1/{r}}  \big\|_{\tilde s , \infty}= \big\|  \sum_k \lft T( F_k, G_k) \rg^{r} \big\|_{\frac{\tilde s}{r} , \infty}^{\frac{1}{r}}\leq K_{ s_1, s_2, \tilde s} \big\| \| F \|_{\ell^{r_1}}  \big\|_{s_1} \big\| \| G \|_{\ell^{r_2}}  \big\|_{s_2}
\]}
whenever $\| F \|_{\ell^{r_1}}\sim A_1 \one_{E_1}$ and $\| G \|_{\ell^{r_2}}\sim A_2 \one_{E_2}$. This further implies that
\[
d_{\sum_k \lft T( F_k, G_k) \rg^{r} }(\lambda) \leq K_{s_1, s_2, \tilde s}^{\tilde s} \lambda^{-\frac{\tilde s}{r}} \big\| \| F \|_{\ell^{r_1}}  \big\|_{s_1}^{\tilde s} \big\| \| G \|_{\ell^{r_2}}  \big\|_{s_2}^{\tilde s}.
\]
We will apply this to the functions $\vec f_{m_1}$ and $\vec g_{m_2}$. We also note that, due to the way $\vec f_{m_1}$ and $\vec g_{m_2}$ were defined, we have that 
\[
\big\| \| \vec f_{m_1} \|_{\ell^{r_1}}  \big\|_{s_1}  \lesssim 2^{m_1} d_{\| \vec f \|_{\ell^{r_1}}}(2^{m_1})^{\frac{1}{s_1}}, \quad    \big\| \| \vec g_{m_2} \|_{\ell^{r_2}}  \big\|_{s_2} \lesssim 2^{m_2} d_{\| \vec g \|_{\ell^{r_2}}}(2^{m_2})^{\frac{1}{s_2}}.
\]
Hence, 
{\fontsize{9}{10}
\[
d_{ \sum_k  \lft T(f_{k, m_1}, g_{k, m_2})\rg^r }(c_{n, m_1, m_2}2^n) \lesssim K_{ s_1, s_2, \tilde s}^{\tilde s} \left( c_{n, m_1, m_2}  2^n\right)^{-\frac{\tilde s}{r}} \cdot \left(  2^{m_1} d_{\| \vec f \|_{\ell^{r_1}}}(2^{m_1})^{\frac{1}{s_1}} \right)^{\tilde s} \cdot \left(  2^{m_2} d_{\| \vec g \|_{\ell^{r_2}}}(2^{m_2})^{\frac{1}{s_2}} \right)^{\tilde s},
\]}
where in fact the tuple $( s_1, s_2, \tilde s)$ depends on $n, m_1, m_2$ and is to be chosen later. For simplicity, we don't write down this dependency, but it is an important step in our proof.

All of the above imply that, for $\tilde K$ so that $\tilde K^s:=\sup_{\left( s_1, s_2, \tilde s\right)} K_{ s_1, s_2, \tilde s}^{\tilde s}$, we have
\begin{align*}
\big\| \sum_k \lft T(f_k, g_k)  \rg^r   \big\|_\frac{s}{r}^\frac{s}{r} \lesssim \tilde{K}^s \sum_n 2^{\frac{n \left(s -\tilde s \right)}{r}}\sum_{m_1, m_2} c_{n, m_1, m_2}^{-\frac{\tilde s}{r}} 2^{m_1\tilde s} d_{\| \vec f \|_{\ell^{r_1}}}(2^{m_1})^{\frac{\tilde s}{s_1}} 2^{m_2 \tilde s} d_{\| \vec g \|_{\ell^{r_2}}}(2^{m_2})^{\frac{\tilde s}{s_2}}.
\end{align*}

We clearly need to make use of conditions \eqref{eq:cond-norm1-f} and \eqref{eq:cond-norm1-g}, and of the H\"older condition $\frac{1}{\tilde s}=\frac{1}{s_1}+\frac{1}{s_2}$.  We note that the above expression can be eventually written as
\begin{align*}
& \sum_n \sum_{m_1, m_2} c_{n, m_1, m_2}^{-\frac{\tilde s}{r}} 2^{\tilde s \left( \frac{1}{p}-\frac{1}{s_1}\right) \left(m_1 p- \frac{ns}{r}  \right)} 2^{\tilde s \left( \frac{1}{q}-\frac{1}{s_2}\right) \left(m_2 q- \frac{n s}{r}  \right)} \\
&\qquad \qquad \cdot 2^{m_1 p \frac{\tilde s}{s_1}} d_{\| \vec f \|_{\ell^{r_1}}}(2^{m_1})^{\frac{\tilde s}{s_1}} 2^{m_2 q \frac{\tilde s}{s_2}} d_{\| \vec g \|_{\ell^{r_2}}}(2^{m_2})^{\frac{\tilde s}{s_2}} \\
&=\sum_n\sum_{m_1p-\frac{ns}{r}, m_2q -\frac{ns}{r}} c_{n, m_1, m_2}^{-\frac{\tilde s}{r}} 2^{\tilde s \left( \frac{1}{p}-\frac{1}{s_1}\right) \left(m_1 p- \frac{ns}{r}  \right)} 2^{\tilde s \left( \frac{1}{q}-\frac{1}{s_2}\right) \left(m_2 q- \frac{n s}{r}  \right)} \\
&\quad \qquad \cdot 2^{\frac{\tilde s}{s_1}\left( m_1p-\frac{ns}{r}+\frac{ns}{r} \right)} d_{\| \vec f \|_{\ell^{r_1}}}(2^{\frac{1}{p}\left(m_1 p -\frac{ns}{r} +\frac{ns}{r}  \right)})^{\frac{\tilde s}{s_1}}\cdot 2^{\frac{\tilde s}{s_2}\left( m_2q-\frac{ns}{r}+\frac{ns}{r} \right)} d_{\| \vec g \|_{\ell^{r_2}}}(2^{\frac{1}{q}\left(m_2 q -\frac{ns}{r} +\frac{ns}{r}  \right)})^{\frac{\tilde s}{s_2}}.
\end{align*}

Now we will turn to our advantage the freedom to choose the triples $( s_1, s_2, \tilde s)$ and the numbers $c_{n, m_1, m_2}$. First we fix $\epsilon$ small and we choose $( s_1, s_2, \tilde s)$ sufficiently close to $( p, q, s)$ so that 
\[
\tilde s \left(\frac{1}{p}-\frac{1}{s_1}\right)\left( m_1p-\frac{ns}{r}\right), \tilde s \left(\frac{1}{q}-\frac{1}{s_2}\right)\left( m_2q-\frac{ns}{r}\right) \leq -\epsilon \max \left(\lft m_1 p -\frac{np}{r}\rg, \lft m_2q -\frac{np}{r} \rg \right).
\]
We then choose $c_{n, m_1, m_2}$of the form
\[
c_{n, m_1, m_2}:= \gamma 2^{-\frac{\epsilon r}{3\tilde s}\lft m_1 p -\frac{np}{r}\rg } \cdot 2^{-\frac{\epsilon r}{3\tilde s}\lft m_2 q -\frac{np}{r}\rg }, 
\]
where $\gamma$ is so that $\sum_{m_1, m_2} c_{n, m_1, m_2}=1$, and it depends only on $( p, q, s)$ and $\epsilon$.

Making the change of variables $\bar m_1:= m_1 p -\frac{np}{r}$ and $\bar m_2:= m_2 q -\frac{np}{r}$, we obtain
\begin{align*}
&\big\| \big(  \sum_k  \lft T(f_k, g_k)\rg^r  \big)^{1/r}   \big\|_{s}^s  \lesssim \tilde{K} \gamma \sum_{\bar m_1, \bar m_2} 2^{-\frac{\epsilon}{3} \max\left(\lft \bar m_1\rg, \lft \bar m_2 \rg \right)}\\
& \qquad \cdot \sum_n 2^{\frac{\tilde s}{s_1}\left( \bar m_1 +\frac{ns}{r} \right)} d_{\| \vec f \|_{\ell^{r_1}}}(2^{\frac{1}{p}\left(\bar m_1 +\frac{ns}{r}  \right)})^{\frac{\tilde s}{s_1}} 2^{\frac{\tilde s}{s_2}\left( \bar m_2+\frac{ns}{r} \right)} d_{\| \vec g \|_{\ell^{r_2}}}(2^{\frac{1}{q}\left(\bar m_2 +\frac{ns}{r}  \right)})^{\frac{\tilde s}{s_2}}.
\end{align*}

Applying H\"older's inequality in the last line, and thanks to identities \eqref{eq:cond-norm1-f} and \eqref{eq:cond-norm1-g}, we obtain the $L^p \times L^q \to L^s$ strong-type estimate. Concerning the constant $K$, we can see that 
\[
K_{ p, q, s} \sim \sup_{\left( s_1, s_2, \tilde s \right) \in \ic V_{\left( p, q, s \right)}} K_{ s_1, s_2, \tilde s}.
\]
\end{proof}

The proof of Proposition \ref{prop:interp-multi} is similar, and the fact that we allow for arbitrary measures is of no consequence. In this situation, we use the subadditivity of $\ds \| T(\vec f, \vec g)  \|_{L^R}^{r^{j_0}}$, for some index $1 \leq j_0 \leq N$ as in Proposition \ref{prop:reorder}.


\bibliographystyle{alpha}


\end{document}